\documentclass[12pt]{article}
\usepackage[margin=1in]{geometry}  
\usepackage{graphicx}              
\usepackage{amsmath}               
\usepackage{amsthm}                
\usepackage[ansinew]{inputenc}
\usepackage{array,color}
\usepackage{amsxtra}
\usepackage{amstext}
\usepackage{latexsym}
\usepackage{dsfont}
\usepackage[all]{xy}
\usepackage{amsmath,amsfonts,amssymb}
\usepackage{soul}

\usepackage{marginnote}

\usepackage{enumitem}  
\usepackage{calc}
\setlist{labelindent=1pt,itemsep=.5em}
\setlist[itemize]{leftmargin=1.2cm}
\setlist[enumerate]{itemindent=0em,leftmargin=1.2cm}
\setlist[enumerate,1]{label={\upshape(\roman*)}}

\usepackage{cite}

\usepackage{hyperref}
\hypersetup{
colorlinks   = true,
citecolor    = black,
linkcolor =black,
urlcolor=black,
}

\hypersetup{hidelinks}

\usepackage{authblk}

\newcommand{\email}[1]{%
    \normalsize\href{mailto:#1}{\color{black}{#1} }}

\allowdisplaybreaks

\makeatletter
\newcommand{\subjclass}[2][2020]{%
  \let\@oldtitle\@title%
  \gdef\@title{\@oldtitle\footnotetext{#1 \emph{Mathematics subject classification}: #2}}%
}
\newcommand{\keywords}[1]{%
  \let\@@oldtitle\@title%
  \gdef\@title{\@@oldtitle\footnotetext{\emph{Keywords}: #1}}%
}
\makeatother

\newtheorem{thm}{Theorem}[section]
\newtheorem{cor}[thm]{Corollary}
\newtheorem{lem}[thm]{Lemma}
\newtheorem{prop}[thm]{Proposition}
\theoremstyle{definition}
\newtheorem{defn}[thm]{Definition}
\theoremstyle{remark}
\newtheorem{rmk}[thm]{Remark}
\theoremstyle{remark}

\newtheorem{ex}[thm]{Example}
\newtheorem{exes}[thm]{Examples}
\numberwithin{equation}{section}

\title{Admissible Hom-Novikov-Poisson and Hom-Gelfand-Dorfman color Hom-algebras}

\author[1,2]{Ismail Laraiedh \thanks{\emph{Corresponding author}: Ismail Laraiedh, ismail.laraiedh@gmail.com}}
\author[3]{Sergei Silvestrov}
\affil[1]{\Affilfont Departement of Mathematics, Faculty of Sciences,
\authorcr \Affilfont Sfax University, Box 1171, 3000 Sfax, Tunisia
\authorcr \Affilfont
\email{ismail.laraiedh@gmail.com}}
\affil[2]{Departement of Mathematics,
\authorcr \Affilfont College of Sciences and Humanities Al Quwaiiyah,
\authorcr \Affilfont Shaqra University, Kingdom of Saudi Arabia
\authorcr \Affilfont
\email{ismail.laraiedh@su.edu.sa}}
\affil[3]{\Affilfont Division of Mathematics and Physics,
\authorcr \Affilfont School of Education, Culture and Communication,
\authorcr \Affilfont M\"{a}lardalen University, Box 883, 72123 V{\"a}ster{\aa}s, Sweden
\authorcr \Affilfont
\email{sergei.silvestrov@mdu.se}}

\subjclass[2020]{17B61, 17D30, 17B63, 16D20, 17D25}
\keywords{Hom-Novikov-Poisson color Hom-algebra, Hom-Gelfand-Dorfman color Hom-algebra}

\date{}

\begin{document}
\maketitle

\abstract{The main feature of color Hom-algebras is that the identities defining the structures
are twisted by even linear maps.  The purpose of this paper is to introduce and give some constructions of admissible Hom-Novikov-Poisson color Hom-algebras and Hom-Gelfand-Dorfman color Hom-algebras. Their bimodules and matched pairs are defined and the relevant properties
and theorems are given. Also, the
connections between Hom-Novikov-Poisson color Hom-algebras and Hom-Gelfand-Dorfman color Hom-algebras is proved. Furthermore, we show  that  the class  of  admissible Hom-Novikov-Poisson color Hom-algebras is closed under tensor product.
}


\section{Introduction}

A Novikov algebra has a binary operation such that the associator is left-symmetric and that the right multiplication operators commute.  Novikov algebras play a major role in the studies of Hamiltonian operators and Poisson brackets of hydrodynamic type \cite{bn,dn1,dn2,dg1,dg2,GelfandDorfman1979:Hamoperatandassociatoralgebraic}.  The left-symmetry of the associator implies that every Novikov algebra is Lie admissible, i.e., the commutator bracket $[x,y] = xy - yx$ gives it a Lie algebra structure.

Poisson algebras are used in many fields in mathematics and physics.  In mathematics, Poisson algebras play a fundamental role in Poisson geometry \cite{vaisman}, quantum groups \cite{cp,dri87}, and deformation of commutative associative algebras \cite{ger2}.  In physics, Poisson algebras are a major part of deformation quantization \cite{kont}, Hamiltonian mechanics \cite{arnold}, and topological field theories \cite{ss}.  Poisson-like structures are also used in the study of vertex operator algebras \cite{fb}.

The theory of Hom-algebras has been initiated in \cite{HartwigLarSil:defLiesigmaderiv, LarssonSilvJA2005:QuasiHomLieCentExt2cocyid,LarssonSilv2005:QuasiLieAlg} motivated by quasi-deformations of Lie algebras of vector fields, in particular q-deformations of Witt and Virasoro algebras. Hom-Lie algebras and more general quasi-Hom-Lie algebras were introduced first by Hartwig, Larsson and Silvestrov in  \cite{HartwigLarSil:defLiesigmaderiv} where a general approach to discretization of Lie algebras of vector fields using general twisted derivations ($\sigma$-deriva\-tions) and a general method for construction of deformations of Witt and Virasoro type algebras based on twisted derivations have been developed. The general quasi-Lie algebras, containing the quasi-Hom-Lie algebras and Hom-Lie algebras as subclasses, as well their graded color generalization, the color quasi-Lie algebras including color quasi-Hom-Lie algebras, color Hom-Lie algebras and their special subclasses the quasi-Hom-Lie superalgebras and Hom-Lie superalgebras, have been first introduced in  \cite{HartwigLarSil:defLiesigmaderiv,LarssonSilvJA2005:QuasiHomLieCentExt2cocyid,LarssonSilv2005:QuasiLieAlg,LarssonSilv:GradedquasiLiealg,LarssonSilv:quasidefsl2,SigSilv:CzechJP2006:GradedquasiLiealgWitt}.
Subsequently, various classes of Hom-Lie admissible algebras have been considered in \cite{ms:homstructure}. In particular, in \cite{ms:homstructure}, the Hom-associative algebras have been introduced and shown to be Hom-Lie admissible, that is leading to Hom-Lie algebras using commutator map as new product, and in this sense constituting a natural generalization of associative algebras as Lie admissible algebras leading to Lie algebras using commutator map. Furthermore, in \cite{ms:homstructure}, more general $G$-Hom-associative algebras including Hom-associative algebras, Hom-Vinberg algebras (Hom-left symmetric algebras), Hom-pre-Lie algebras (Hom-right symmetric algebras), and some other Hom-algebra structures, generalizing $G$-associative algebras, Vinberg and pre-Lie algebras respectively, have been introduced and shown to be Hom-Lie admissible, meaning that for these classes of Hom-algebras, the operation of taking commutator leads to Hom-Lie algebras as well. Also, flexible Hom-algebras have been introduced, connections to Hom-algebra generalizations of derivations and of adjoint maps have been noticed, and some low-dimensional Hom-Lie algebras have been described.
In Hom-algebra structures, defining algebra identities are twisted by linear maps.
Since the pioneering works \cite{HartwigLarSil:defLiesigmaderiv,LarssonSilvJA2005:QuasiHomLieCentExt2cocyid,LarssonSilv2005:QuasiLieAlg,LarssonSilv:GradedquasiLiealg,LarssonSilv:quasidefsl2,ms:homstructure}, Hom-algebra structures have developed in a popular broad area with increasing number of publications in various directions.
Hom-algebra structures include their classical counterparts and open new broad possibilities for deformations, extensions to Hom-algebra structures of representations, homology, cohomology and formal deformations, Hom-modules and Hom-bimodules, Hom-Lie admissible Hom-coalgebras, Hom-coalgebras, Hom-bialgebras, Hom-Hopf algebras, $L$-modules, $L$-comodules and Hom-Lie quasi-bialgebras, $n$-ary generalizations of BiHom-Lie algebras and BiHom-associative algebras and generalized derivations, Rota-Baxter operators, Hom-dendriform color Hom-algebras, Rota-Baxter bisystems and covariant bialgebras, Rota-Baxter cosystems, coquasitriangular mixed bialgebras, coassociative Yang-Baxter pairs, coassociative Yang-Baxter equation and generalizations of Rota-Baxter systems and algebras, curved $\mathcal{O}$-operator systems and their connections with tridendriform systems and pre-Lie algebras, BiHom-algebras, BiHom-Frobenius algebras and double constructions, infinitesimal BiHom-bialgebras and Hom-dendriform $D$-bialgebras, Hom-algebras have been considered
\cite{AbdaouiMabroukMakhlouf,
AmmarEjbehiMakhlouf:homdeformation,
AttanLaraiedh:2020ConstrBihomalternBihomJordan,
Bakayoko:LaplacehomLiequasibialg,
Bakayoko:LmodcomodhomLiequasibialg,
BakBan:bimodrotbaxt,
BakyokoSilvestrov:HomleftsymHomdendicolorYauTwi,
BakyokoSilvestrov:MultiplicnHomLiecoloralg,
BenHassineChtiouiMabroukNcib19:CohomLiedeformBiHomleftsym,
BenMakh:Hombiliform,
BenAbdeljElhamdKaygorMakhl201920GenDernBiHomLiealg,
CaenGoyv:MonHomHopf,
ChtiouiMabroukMakhlouf2,
DassoundoSilvestrov2021:NearlyHomass,
EbrahimiFardGuo08,
GrMakMenPan:Bihom1,
HassanzadehShapiroSutlu:CyclichomolHomasal,
HounkonnouDassoundo:centersymalgbialg,
HounkonnouHoundedjiSilvestrov:DoubleconstrbiHomFrobalg,
HounkonnouDassoundo:homcensymalgbialg,
kms:narygenBiHomLieBiHomassalgebras2020,
Laraiedh1:2021:BimodmtchdprsBihomprepois,
LarssonSigSilvJGLTA2008:QuasiLiedefFttN,
LarssonSilvJA2005:QuasiHomLieCentExt2cocyid,
LarssonSilv:quasidefsl2,
LarssonSilvestrovGLTMPBSpr2009:GenNComplTwistDer,
MaZheng:RotaBaxtMonoidalHomAlg,
MabroukNcibSilvestrov2020:GenDerRotaBaxterOpsnaryHomNambuSuperalgs,
MakhSilv:HomDeform,
LiuMakhMenPan:RotaBaxteropsBiHomassalg,
MakhSilv:HomAlgHomCoalg,
MakYau:RotaBaxterHomLieadmis,
RichardSilvestrovJA2008,
RichardSilvestrovGLTbnd2009,
SaadaouSilvestrov:lmgderivationsBiHomLiealgebras,
ShengBai:homLiebialg,
Sheng:homrep,
SigSilv:GLTbdSpringer2009,
SilvestrovParadigmQLieQhomLie2007,
SilvestrovZardeh2021:HNNextinvolmultHomLiealg,
QSunHomPrealtBialg,
SunLi2017:parakahlerhomhomleftsymetric,
Yau:ModuleHomalg,
Yau:HomHom,
Yau:HombialgcomoduleHomalg,
Yau:HomYangBaHomLiequasitribial,
YauHomMalcevHomalternHomJord}.

In \cite{yau5} the author initiated the study of a twisted generalization of Novikov algebras, called Hom-Novikov algebras.
A Hom-Novikov algebra $A$ has a binary operation $\cdot$ and a linear self-map $\alpha$, and it satisfies some $\alpha$-twisted versions of the defining identities of a Novikov algebra.  In \cite{yau5} several constructions of Hom-Novikov algebras were given and some low dimensional Hom-Novikov algebras were classified. Further, Hom-Poisson algebras were defined in \cite{MakhSilv:HomDeform} by Makhlouf and Silvestrov.  It is shown in \cite{MakhSilv:HomDeform} that Hom-Poisson algebras play the same role in the deformation of commutative Hom-associative algebras as Poisson algebras do in the deformation of commutative associative algebras.

In this paper, we introduce and obtain some results on construction of admissible Hom-Novikov-Poisson color Hom-algebras and Hom-Gelfand-Dorfman color Hom-algebras. Their bimodules and matched pairs are defined
and the relevant properties and theorems are obtained. We also show  that  the class  of  admissible Hom-Novikov-Poisson color Hom-algebras are closed under tensor product. In Section \ref{sec:PreliminariesandSomeresults}, we introduce the notions of bimodules and matched pairs of Hom-associative color Hom-algebras, Hom-Novikov color Hom-algebras
and Hom-Lie color Hom-algebras in which we give some results and some examples.
In Section \ref{sec:admissHomNovikovcoloralgebs}, we establish the notions of admissible Hom-Novikov-Poisson color Hom-algebras and we give some explicit constructions.
Their bimodule and matched pair are defined
and their related relevant properties are also given. Finally, we  show  that  the  much  larger  class  of  admissible  Hom-Novikov-Poisson  color Hom-algebras is also closed under tensor products. In Section \ref{sec:homGelDorfmalgs}, we introduce the notions of Hom-Gelfand-Dorfman color Hom-algebras and we discuss some
basic properties and examples of these objects.
Moreover, we
characterize the representations and matched pairs of Hom-Gelfand-Dorfman color Hom-algebras and provide some key constructions.

\section{Preliminaries and some results}
\label{sec:PreliminariesandSomeresults}
Throughout the article, we assume that all linear spaces are over an algebraically closed field $\mathbb{K}$ of characteristic $0$, and denote by $\mathbb{K^{\ast}}=\mathbb{K}\backslash \{0\}$ the group of invertible elements of $\mathbb{K}$
with respect to the multiplication in $\mathbb{K}$.

In this section, we introduce the notions of bimodules and matched pairs of Hom-associative color Hom-algebras, Hom-Novikov color Hom-algebras
and Hom-Lie color Hom-algebras in which we give some results and examples.

Let $\Gamma$ be an abelian group. A linear space $V$ is said to be $\Gamma$-graded, if there is a family
$(V_{\gamma})_{\gamma\in \Gamma}$ of vector subspace of $V$ such that $$V=\bigoplus_{\gamma\in \Gamma}V_{\gamma}.$$
An element $x \in V$ is said to be homogeneous of  degree $\gamma \in \Gamma$ if $x \in V_{\gamma}, \gamma\in \Gamma$, and in this case, $\gamma$ is called the degree of $x$.
In the sequel, we will denote the set of all the homogeneous elements of $V$ by $\mathcal{H}(V)$.
As usual, we denote by $\overline{x}$ the degree of an element $x \in V$. Thus each homogeneous
element $x \in V$ determines a unique group element  $\overline{x} \in \Gamma$ by $x \in V_{\overline{x}}$.
Thus, when no confusion occur, we can drop $"-"$ in notation of degree for convenience of exposition.

Let $V=\bigoplus_{\gamma\in \Gamma}V_{\gamma}$ and $V^{'}=\bigoplus_{\gamma\in \Gamma}V^{'}_{\gamma}$ be two $\Gamma$-graded linear spaces. A linear mapping $f: V \longrightarrow V^{'}$  is said to be homogeneous of  degree $\upsilon \in \Gamma$ if
  $f(V_{\gamma})\subseteq V^{'}_{\gamma+\upsilon}$ for all $\gamma \in \Gamma.$
If in addition $f$ is  homogeneous of degree zero, i.e. $f(V_{\gamma})\subseteq V^{'}_{\gamma}$ holds for any $\gamma \in \Gamma$, then $f$ is said to be even.

An algebra $\mathcal{A}$ is said to be $\Gamma$-graded if its underlying linear space is $\Gamma$-graded, i.e. $\mathcal{A}=\bigoplus_{\gamma\in \Gamma}\mathcal{A}_{\gamma}$, and if, furthermore $\mathcal{A}_{\gamma}\mathcal{A}_{\gamma'}\subseteq \mathcal{A}_{\gamma+\gamma'}$, for all $\gamma, \gamma'\in \Gamma$. It is easy to see
that if $\mathcal{A}$ has a unit element $e$, it follows that $e \in \mathcal{A}_{0}$. A subalgebra of $\mathcal{A}$ is said to be $\Gamma$-graded if it is $\Gamma$-graded as a subspace of $\mathcal{A}$.
Let $\mathcal{A}^{'}$ be another $\Gamma$-graded algebra. A homomorphism $f:\mathcal{A} \longrightarrow \mathcal{A}^{'}$ of $\Gamma$-graded algebras is by definition a homomorphism  of the algebra $\mathcal{A}$ into the algebra  $\mathcal{A}^{'}$, which is, in addition an even mapping.
\begin{defn}[\cite{BahturinMikhPetrZaicevIDLSbk92,ChenSilvestrovOystaeyen:RepsCocycleTwistsColorLie,ChenPetitOystaeyenCOHCHLA,MikhZolotykhCALSbk95,PiontkovskiSilvestrovC3dCLA,ScheunertGLA,ScheunertCOH2,ScheunertZHA,Silvestrov:class3dimcolLiealg}] Let $\mathbb{K}$ be a field and  $\Gamma$ be an abelian group. A map $\varepsilon:\Gamma\times\Gamma\rightarrow \mathbb{K^{\ast}}$ is called a commutation factor on ${\Gamma}$ if the following identities hold, for all $a,b,c$ in $\Gamma$
\begin{enumerate}
\item~~$\varepsilon(a,b)~\varepsilon(b,a)=1,$
\item~~$\varepsilon(a,b+c)=\varepsilon(a,b)~\varepsilon(a,c),$
\item~~$\varepsilon(a+b,c)=\varepsilon(a,c)~\varepsilon(b,c).$
\end{enumerate}
\end{defn}
The definition above implies, in particular, the following relations
$$\varepsilon(a,0)=\varepsilon(0,a)=1,\ \varepsilon(a,a)=\pm1, \  \textrm{for\ all}\  a \in \Gamma.$$
If $x$ and $x'$ are two homogeneous elements of degree $\gamma$ and $\gamma'$ respectively and $\varepsilon$ is a skewsymmetric bicharacter, then we shorten the notation by writing $\varepsilon(x,x')$ instead of $\varepsilon(\gamma,\gamma')$ since the degree of every homogeneous element is unique.
\begin{rmk}\label{remarkepsilon}
Let $A$ and $V$ be two $\Gamma$-graded linear spaces such that
$$ A \oplus V=\bigoplus_{\gamma\in\Gamma}(A \oplus V)_\gamma=\bigoplus_{\gamma\in\Gamma}(A_\gamma \oplus V_\gamma),$$ then, for all $X_1= x_{1}+v_1\in A_{\gamma_1} \oplus V_{\gamma_1},X_2=x_{2}+v_2\in A_{\gamma_2} \oplus V_{\gamma_2}$ we have  $$\varepsilon(x_1,x_2)=\varepsilon(x_1,v_2)=\varepsilon(v_1,x_2)=\varepsilon(v_1,v_2)=\varepsilon(X_1,X_2).$$
\end{rmk}
\begin{ex} Some standard examples of skew-symmetric bicharacters are:
 \begin{enumerate}
\item [1)] $\Gamma=\mathbb{Z}_2,\quad \varepsilon(i, j)=(-1)^{ij}$,
\item [2)] $\Gamma=\mathbb{Z}_2^n, \quad
\varepsilon((\alpha_1, \dots, \alpha_n), (\beta_1, \dots, \beta_n)):= (-1)^{\alpha_1\beta_1+\dots+\alpha_n\beta_n}.$
\item [3)] $\Gamma=\mathbb{Z}_2\times\mathbb{Z}_2,\quad \varepsilon((i_1, i_2), (j_1, j_2))=(-1)^{i_1j_2-i_2j_1}$,
\item [4)] $\Gamma=\mathbb{Z}\times\mathbb{Z} ,\quad \varepsilon((i_1, i_2), (j_1, j_2))=(-1)^{(i_1+i_2)(j_1+j_2)}$.
\end{enumerate}
\end{ex}
\begin{defn}
A color Hom-algebra or a Hom-color algebra, $(A,\cdot,\varepsilon,\alpha)$ is a $\Gamma$-graded linear space $A$ equipped with even bilinear multiplication $\cdot$, even twisting map $\alpha$ and commutation factor $\varepsilon$.
\end{defn}
\begin{defn}
 A derivation of degree $d\in \Gamma$ on a color Hom-algebra $(A, \cdot, \varepsilon, \alpha)$ is a linear map $D : A\rightarrow A$ such that for
any $x,y\in \mathcal{H}(A)$,
$$D(x\cdot y)=D(x)\cdot y+\varepsilon(d, x) x\cdot D(y).$$
In particular, an even derivation $D : A\rightarrow A$ is a derivation of degree zero, i.e., for all $x,y\in \mathcal{H}(A)$,
$$D(x\cdot y)=D(x)\cdot y+x\cdot D(y).$$
\end{defn}
\subsection{\texorpdfstring{$\boldsymbol\varepsilon$-}-Commutative Hom-associative color Hom-algebras}
\begin{defn}[\cite{LY}]\label{def:homasscoloralg}
A Hom-associative color Hom-algebra is a color Hom-algebra
$(A,\cdot,\varepsilon,\alpha)$ satisfying
for $x,y,z\in \mathcal{H}(A)$,
\begin{align}\label{Homass:homassociator}
    as_{A}(x,y,z):=\alpha(x)\cdot(y\cdot z)-(x\cdot y)\cdot\alpha(z)=0. &&  \mbox{(Hom-associativity)}
\end{align}
If in addition, for any $x, y \in \mathcal{H}(A)$,
\begin{equation}\label{commutative:homcoloralgeb}
x\cdot y= \varepsilon(x, y) y\cdot x,
\end{equation}
then $(A, \cdot,\varepsilon,\alpha)$ is said to be a $\varepsilon$-commutative Hom-associative color Hom-algebra.
\end{defn}
\begin{ex}
Let $A=A_0\oplus A_{1}=<e_1>\oplus<e_2,e_3>$  be a $3$-dimensional superspace. Then $A$ is a $\varepsilon$-commutative Hom-associative color Hom-algebra with
\begin{align*}
&\text{the bicharacter:}  \quad
\varepsilon(i,j)=(-1)^{ij},\\
&\text{the multiplication:} \quad
e_1\cdot e_2= e_2\cdot e_1=-2e_3, \\
&\text{the even linear map $\alpha : A\rightarrow A$ defined by}
\begin{array}[t]{l}
\alpha(e_1)=\sqrt{2}e_1,\quad \alpha(e_2)=e_3-e_2,\\
\quad \alpha(e_3)=e_3.
\end{array}
\end{align*}
\end{ex}

In the following, we introduce the notion of bimodule of $\varepsilon$-commutative Hom-associative color Hom-algebra.
\begin{defn}
Let $(A, \cdot, \varepsilon,\alpha)$ be an $\varepsilon$-commutative Hom-associative color Hom-algebra, $(V,\beta)$ be a pair consisting of $\Gamma$-graded linear space $V$ and an even linear map $\beta:V\rightarrow V$, and $ s: A \rightarrow End(V) $ be an even linear map. The triple $(s, \beta, V)$ is called a bimodule of $(A, \cdot, \varepsilon,\alpha)$ if for all $ x, y \in  \mathcal{H}(A), v \in \mathcal{H}(V) $,
\begin{align}
\label{bimod:comm.hom.ass.color.algeb.lara.silvest}
s(x\cdot y)\beta(v)=s(\alpha(x))s(y)v.
\end{align}
\end{defn}
\begin{prop}\label{prop:semidirect.comm.Hom.ass.color.algeb}
Let $(s, \beta, V)$ is a bimodule of a $\varepsilon$-commutative Hom-associative color Hom-algebra $( A, \cdot,\varepsilon, \alpha)$. Then the direct sum of $\Gamma$-graded linear spaces, $$ A \oplus V=\bigoplus_{\gamma\in\Gamma}(A \oplus V)_\gamma=\bigoplus_{\gamma\in\Gamma}(A_\gamma \oplus V_\gamma),$$  is turned into a $\varepsilon$-commutative Hom-associative color Hom-algebra by defining multiplication and the twisting map in $A \oplus V $ for all $X_1= x_{1}+v_1\in A_{\gamma_1} \oplus V_{\gamma_1},X_2=x_{2}+v_2\in A_{\gamma_2} \oplus V_{\gamma_2}$ by
\begin{eqnarray}
\label{operation.semidirect.ass.color.algeb1}
(x_{1} + v_{1})\ast (x_{2} + v_{2}) &=& x_{1} \cdot x_{2} + \big(s(x_{1})v_{2} + \varepsilon(v_1,x_2)s(x_{2})v_{1}\big),\\
\label{operation.semidirect.ass.color.algeb2}
(\alpha\oplus\beta)(x_{1} + v_{1})&=&\alpha(x_{1}) + \beta(v_{1}).
\end{eqnarray}
\end{prop}
\begin{proof}
We prove the commutativity and Hom-associativity in $A\oplus V$.
For all elements $X_i=x_i+v_i\in A_{\gamma_i} \oplus V_{\gamma_i},\ i=1,2,3$,
\begin{align*}
&X_1 \ast X_2 =(x_1+v_1)\ast(x_2+v_2)\\
&\quad=x_1\cdot x_2+\big(s(x_1)v_2+\varepsilon(v_1,x_2)s(x_2)v_1\big)\\
&\quad=\varepsilon(x_1,x_2)x_2\cdot x_1+\big(\varepsilon(v_1,x_2)s(x_2)v_1+s(x_1)v_2\big)\\
&\quad=\varepsilon(X_1,X_2)x_2\cdot x_1+\big(\varepsilon(X_1,X_2)s(x_2)v_1+s(x_1)v_2\big)\\
&\quad=\varepsilon(X_1,X_2)\big(x_2\cdot x_1+s(x_2)v_1+\varepsilon(X_2,X_1)s(x_1)v_2\big)\\
&\quad=\varepsilon(X_1,X_2)\big(x_2\cdot x_1+s(x_2)v_1+\varepsilon(v_2,x_1)s(x_1)v_2\big)\\
&\quad=\varepsilon(X_1,X_2)(x_2+v_2)\ast(x_1+v_1)\\
&\quad=\varepsilon(X_1,X_2)X_2\ast X_1,\\[0,2cm]
&(X_1\ast X_2)\ast(\alpha+\beta)X_3-(\alpha+\beta)X_1\ast(X_2\ast X_3)\\
&\quad=\big((x_1+v_1)\ast(x_2+v_2)\big)\ast(\alpha+\beta)(x_3+v_3)\\&\quad\quad-(\alpha+\beta)(x_1+v_1)\ast\big((x_2+v_2)\ast(x_3+v_3)\big)\\
&\quad=\big(x_1\cdot x_2+s(x_1)v_2+\varepsilon(v_1,x_2)s(x_2)v_1\big)\ast(\alpha(x_3)+\beta(v_3))\\
&\quad\quad-(\alpha(x_1)+\beta(v_1))\ast\big(x_2\cdot x_3+s(x_2)v_3+\varepsilon(v_2,x_3)s(x_3)v_2\big)\\
&\quad=(x_1\cdot x_2)\cdot \alpha(x_3)+s(x_1\cdot x_2)\beta(v_3)+\varepsilon(x_1+x_2,x_3)s(\alpha(x_3))s(x_1)v_2\\
&\quad\quad+\varepsilon(x_1+x_2,x_3)\varepsilon(x_1,x_2)s(\alpha(x_3))s(x_2)v_1-\alpha(x_1)\cdot(x_2\cdot x_3)\\
&\quad\quad-s(\alpha(x_1))s(x_2)v_3-\varepsilon(x_2,x_3)s(\alpha(x_1))s(x_3)v_2\\&\quad\quad-\varepsilon(v_1,x_2+x_3)s(x_2\cdot x_3)\beta(v_1)\\
&\quad=\underbrace{\Big((x_1\cdot x_2)\cdot\alpha(x_3)-\alpha(x_1)\cdot(x_2\cdot x_3)\Big)}_{\text{$=0$ by \eqref{Homass:homassociator}}}\\&\quad\quad
+\underbrace{\Big(s(x_1\cdot x_2)\beta(v_3)-s(\alpha(x_1))s(x_2)v_3\Big)}_{\text{$=0$ by \eqref{bimod:comm.hom.ass.color.algeb.lara.silvest}}}\\
&\quad\quad+\underbrace{\varepsilon(x_2,x_3)\Big(\varepsilon(x_1,x_3)s(\alpha(x_3))s(x_1)v_2-s(\alpha(x_1))s(x_3)v_2\Big)}_{\text{$=0$ by \eqref{commutative:homcoloralgeb} and \eqref{bimod:comm.hom.ass.color.algeb.lara.silvest}}}\\
&\quad\quad+\underbrace{\varepsilon(x_1,x_2+x_3)\Big(\varepsilon(x_2,x_3)s(\alpha(x_3))s(x_2)v_1-s(x_2\cdot x_3)\beta(v_1)\Big)}_{\text{$=0$ by \eqref{commutative:homcoloralgeb} and \eqref{bimod:comm.hom.ass.color.algeb.lara.silvest}}}=0.
\qedhere
\end{align*}
\end{proof}
The $\varepsilon$-commutative Hom-associative color Hom-algebra constructed in Proposition \ref{prop:semidirect.comm.Hom.ass.color.algeb} is denoted
by $(A \oplus V, \ast, \varepsilon,\alpha + \beta)$ or $A \ltimes_{s, \alpha, \beta} V.$
\begin{ex}
Let $(A,\cdot,\varepsilon,\alpha)$ be a $\varepsilon$-commutative Hom-associative color Hom-algebra. Then, $(S,\alpha,A)$ with  $S(x)y=x\cdot y$ for all $x,y\in \mathcal{H}(A)$, is a bimodule of $(A,\cdot,\varepsilon,\alpha)$  called the regular bimodule of $(A,\cdot,\varepsilon,\alpha)$.
\end{ex}
In the following, we introduce the notion of matched pair of $\varepsilon$-commutative Hom-associative color Hom-algebras.
\begin{prop}
\label{prop:matched.pairs.comm.Hom-ass.coloralgeb}
Let $( A, \cdot_A, \varepsilon,\alpha)$ and $(B,\cdot_B, \varepsilon, \beta)$ be $\varepsilon$-commutative
Hom-associative color Hom-algebras. Suppose that there are even linear maps $s_{A}:  A\rightarrow End(B)$ and $s_{
B} :  B\rightarrow End(A)$ such that
 $(s_{A},  \beta, B)$ is
a bimodule of $ A$, and $(s_{B}, \alpha,
 A)$ is a bimodule of $B$,
satisfying, for any $ x, y \in   \mathcal{H}(A), a,b \in  \mathcal{H}(B)$, the following conditions:
\begin{align} \label{matched.pairs:comm.Hom.ass.algeb1}
&\varepsilon(b,x)\beta(a)\cdot_B(s_A(x)b)+\varepsilon(a,b+x)s_A(s_B(b)x)\beta(a)\nonumber\\&=\varepsilon(a+b,x)s_A(\alpha(x))(a\cdot_B b),\\[0.2cm]
\label{matched.pairs:comm.Hom.ass.algeb2}
&\beta(a)\cdot_B(s_A(x)b)+\varepsilon(a,x+b)\varepsilon(x,b)s_A(s_B(b)x)\beta(a)\nonumber\\
&=\varepsilon(a,x)s_A(x)a\cdot_B\beta(b)+s_A(s_B(a)x)\beta(b),\\[0.2cm]
\label{matched.pairs:comm.Hom.ass.algeb3}
&\varepsilon(y,a)\alpha(x)\cdot_A(s_B(a)y)+\varepsilon(x,y+a)s_B(s_A(y)a)\alpha(x)\nonumber\\&=\varepsilon(x+y,a)s_B(\beta(a))(x\cdot_A y),\\[0.2cm]
\label{matched.pairs:comm.Hom.ass.algeb4}
&\alpha(x)\cdot_A(s_B(a)y)+\varepsilon(x,a+y)\varepsilon(a,y)s_B(s_A(y)a)\alpha(x)\nonumber\\
&=\varepsilon(x,a)s_B(a)x\cdot_A\alpha(y)+s_B(s_A(x)a)\alpha(y).
\end{align}
Then, $(A,B,l_A,r_A,\beta,l_B,r_B,\alpha)$ is called a matched pair of $\varepsilon$-commutative
Hom-associ\-ative color Hom-algebras. In this case, there is a $\varepsilon$-commutative
Hom-associative color Hom-algebra structure on the direct sum of the underlying $\Gamma$-graded
linear spaces of $ A$ and $B$, $$ A \oplus B=\bigoplus_{\gamma
\in\Gamma}(A \oplus B)_\gamma = \bigoplus_{\gamma\in\Gamma}(A_\gamma \oplus B_\gamma),$$
given for all $ x+a\in A_{\gamma_1}\oplus B_{\gamma_1}, y+b\in A_{\gamma_2}\oplus B_{\gamma_2}$ by
\begin{eqnarray}
(x + a)\cdot (y + b) &=& \big(x\cdot_A y + s_{ B}(a)y + \varepsilon(x,b)s_{ B}(b)x\big)\nonumber\\
& +& \big (a\cdot_B b +  s_{A}(x)b + \varepsilon(a,y) s_{  A}(y)a\big),\\
(\alpha\oplus\beta)(x + a)&=&\alpha(x) + \beta(a).
\end{eqnarray}
\end{prop}
\begin{proof}
Let $ X=x+a\in A_{\gamma_1}\oplus B_{\gamma_1}, Y=y+b\in A_{\gamma_2}\oplus B_{\gamma_2} ,Z=z+c \in A_{\gamma_3}\oplus B_{\gamma_3}$. First, we prove the commutativity condition:
\begin{align*}
&X\cdot Y-\varepsilon(X,Y) Y\cdot X =(x+a)\cdot(y+b)-\varepsilon(X,Y)(y+b)\cdot(x+a)\\
&\quad=x\cdot_A y+s_B(a)y+\varepsilon(x,b) s_B(b)x+a\cdot_B b+s_A(x)b+\varepsilon(a,y)s_A(y)a\\
&\quad\quad-\varepsilon(X,Y)\big(y\cdot_A x+s_B(b)x+\varepsilon(y,a)s_B(a)y+b\cdot_B a+s_A(y)a
+\varepsilon(b,x) s_A(x)b\big)\\
&\quad=\big(x\cdot_A y-\varepsilon(X,Y) y\cdot_A x\big)+\big(a\cdot_B b-\varepsilon(X,Y) b\cdot_B a\big)\\
&\quad\quad+\big(s_B(a)y-\varepsilon(X,Y)\varepsilon(y,a)s_B(a)y\big)+\big(\varepsilon(x,b)s_B(b)x-\varepsilon(X,Y)s_B(b)x\big)
\\
&\quad\quad+\big(s_A(x)b-\varepsilon(X,Y)\varepsilon(b,x)s_A(x)b\big) +\big(\varepsilon(a,y)s_A(y)a-\varepsilon(X,Y)s_A(y)a\big)=0. \\
&\hspace{9,5cm} \text{\scriptsize (using Remark \ref{remarkepsilon} and \eqref{commutative:homcoloralgeb})}
\end{align*}
Next, we prove the Hom-associativity condition:
\begin{align*}
&(X\cdot Y)\cdot(\alpha+\beta)Z=\big((x+a)\cdot(y+b)\big)\cdot(\alpha+\beta)(z+c)\\
&\text{\scriptsize (using Remark \ref{remarkepsilon})}\\
&\quad=\Big(\big(x\cdot_A y+s_B(a)y+\varepsilon(x,y)s_B(b)x\big)\\
&\quad\quad+\big(a\cdot_B b+s_A(x)b+\varepsilon(x,y)s_A(y)a\big)\Big)\cdot(\alpha(z)+\beta(c))\\
&\quad=(x\cdot_A y)\cdot_A\alpha(z)+s_B(a)y\cdot_A\alpha(z)+\varepsilon(x,y)s_B(b)x\cdot_A\alpha(z)\\
&\quad\quad+s_B(a\cdot_B b)\alpha(z)+s_B(s_A(x)b)\alpha(z)+\varepsilon(x,y)s_B(s_A(y)a)\alpha(z)\\
&\quad\quad+\varepsilon(x+y,z)\Big(s_B(\beta(c))(x\cdot_A y)+s_B(\beta(c))s_B(a)y\\&\quad\quad+\varepsilon(x,y)s_B(\beta(c))s_B(b)x\Big)\\
&\quad\quad+(a\cdot_B b)\cdot_B\beta(c)+s_A(x)b\cdot_B\beta(c)+\varepsilon(x,y)s_A(y)a\cdot_B\beta(c)\\
&\quad\quad+s_A(x\cdot_A y)\beta(c)+s_A(s_B(a) y)\beta(c)+\varepsilon(x,y)s_A(s_B(b)x)\beta(c)\\
&\quad\quad+\varepsilon(x+y,z)
\Big(s_A(\alpha(z))(a\cdot_B b)+s_A(\alpha(z))s_A(x)b+\varepsilon(x,y)s_A(\alpha(z))s_A(y)a\Big)\\
&\text{\scriptsize (using Remark \ref{remarkepsilon})}\\
&\quad=\Big((x\cdot_A y)\cdot_A\alpha(z)+s_B(a)y\cdot_A\alpha(z)+\varepsilon(x,y)s_B(b)x\cdot_A\alpha(z)\\
&\quad\quad+s_B(a\cdot_B b)\alpha(z)+s_B(s_A(x)b)\alpha(z)+\varepsilon(x,y)s_B(s_A(y)a)\alpha(z)\\
&\quad\quad+(a\cdot_B b)\cdot_B\beta(c)+s_A(x)b\cdot_B\beta(c)+\varepsilon(x,y)s_A(y)a\cdot_B\beta(c)\\
&\quad\quad+s_A(x\cdot_A y)\beta(c)+s_A(s_B(a) y)\beta(c)+\varepsilon(x,y)s_A(s_B(b)x)\beta(c)\Big)\\
&\quad\quad+\varepsilon(x+y,z)\Big(s_B(\beta(c))(x\cdot_A y)+s_B(\beta(c))s_B(a)y\\
&\quad\quad+\varepsilon(x,y)s_B(\beta(c))s_B(b)x+s_A(\alpha(z))(a\cdot_B b)+s_A(\alpha(z))s_A(x)b\\&\quad\quad+\varepsilon(x,y)s_A(\alpha(z))s_A(y)a
\Big),\\[0,2cm]
&(\alpha+\beta)X\cdot(Y\cdot Z)\\
&\quad=(\alpha+\beta)(x+a)\cdot\big((y+b)\cdot(z+c)\big)\\
&\quad=(\alpha(x)+\beta(a))\cdot\Big(\big(y\cdot_A z+s_B(b)z+\varepsilon(y,c)s_B(c)y\big)\\
&\quad\quad+\big(b\cdot_B c+s_B(y)c+\varepsilon(b,z)s_A(z)b\big)\Big)\\
&\text{\scriptsize (using Remark \ref{remarkepsilon})}\\
&\quad=\alpha(x)\cdot_A(y\cdot_A z)+\alpha(x)\cdot_As_B(b)z+\varepsilon(y,z)\alpha(x)\cdot_A s_B(c)y\\
&\quad\quad+s_B(\beta(a))(y\cdot_A z)+s_B(\beta(a))s_B(b)z+\varepsilon(y,c)s_B(\beta(a))s_B(c)y\\
&\quad\quad+\varepsilon(x,y+z)\Big(s_B(b\cdot_B c)\alpha(x)+s_B(s_A(y)c)\alpha(x)\\
&\quad\quad+\varepsilon(b,z)s_B(s_A(z)b)\alpha(x)\Big)+\beta(a)\cdot_B(b\cdot_B c)+\beta(a)\cdot_B s_A(y)c\\&\quad\quad+\varepsilon(b,z)\beta(a)\cdot_B s_A(z)b+s_A(\alpha(x))(b\cdot_B c)+s_A(\alpha(x))s_A(y)c\\&\quad\quad+\varepsilon(y,z)s_A(\alpha(x))s_A(z)b+\varepsilon(x,y+z)\Big(s_A(y\cdot_A z)\beta(a)\\
&\quad\quad+s_A(s_B(b)z)\beta(a)+\varepsilon(y,z)s_A(s_B(c)y)\beta(a)\Big).
\end{align*}
Using \eqref{matched.pairs:comm.Hom.ass.algeb1}-\eqref{matched.pairs:comm.Hom.ass.algeb4} and that $(s_A, \beta,B)$ and $(s_B, \alpha,A)$ are bimodules of $( A, \cdot_A, \varepsilon,\alpha)$ and $(B,\cdot_B, \varepsilon, \beta)$, respectively, we derive that $(A \oplus B,\cdot,\varepsilon,\alpha+\beta)$ is $\varepsilon$-commutative Hom-associative color Hom-algebra.
This completes the proof.
\end{proof}
This $\varepsilon$-commutative
Hom-associative color Hom-algebra, constructed in Proposition \ref{prop:matched.pairs.comm.Hom-ass.coloralgeb}, is denoted by
$( A\bowtie B, \cdot,\varepsilon, \alpha + \beta)$ or $ A \bowtie^{s_{  A}, \beta}_{s_{ B}, \alpha} B.$
\subsection{On Hom-Novikov color Hom-algebras}
\begin{defn}[\cite{Bakayoko2016arXiv:Hom-Novikovcoloralgebras}]\label{def:HomNovikovcoloralgeb}
A color Hom-algebra $(A,\cdot,\varepsilon,\alpha)$ is called a Hom-Novikov color Hom-algebra if the
following identities are satisfied for all $x,y,z \in \mathcal{H}(A)$:
\begin{align}
\label{eqhomNovikov1}
&(x\cdot y)\cdot\alpha(z)-\alpha(x)\cdot(y\cdot z)=\varepsilon(x,y)\big((y\cdot x)\cdot\alpha(z)-\alpha(y)\cdot(x\cdot z)\big),\\
\label{eqhomNovikov2}
&(x\cdot y)\cdot\alpha(z)=\varepsilon(y,z)(x\cdot z)\cdot\alpha(y).
\end{align}
\end{defn}
\begin{rmk}
If $\alpha=id_{ A}$ in Definition \ref{def:HomNovikovcoloralgeb}, we recover a Novikov color Hom-algebra. So, Novikov color Hom-algebras are a special case of Hom-Novikov color Hom-algebras when the twisting linear map is the identity map.
\end{rmk}
\begin{exes}[Hom-Novikov color Hom-algebras]
Here there are some examples of Hom-Novikov color Hom-algebras.
\begin{enumerate}
    \item
Any $\varepsilon$-commutative Hom-associative color Hom-algebra is a Hom-Novikov color Hom-algebra.
\item
Let $A=A_0\oplus A_{1}=<e_1,e_2>\oplus<e_3,e_4>$ be a $4$-dimensional superspace. Then $A$ is a Hom-Novikov color Hom-algebra with
\begin{align*}
& \text{the bicharacter} \quad
\varepsilon(i,j)=(-1)^{ij}, \\
& \text{the multiplication:} \quad
\begin{array}[t]{ll}
e_1\cdot e_1=\lambda_1e_2,& e_1\cdot e_3=\lambda_2e_4,\\
e_3\cdot e_3=\lambda_3e_2,& e_3\cdot e_1=\lambda_4 e_4, \quad\lambda_i \in \mathbb{K}
\end{array}\\
& \text{the even linear map $\alpha: A \rightarrow A$ defined by}
\begin{array}[t]{lll}
&\alpha(e_1)=-e_1,& \alpha(e_2)=e_1-e_2,\\
&\alpha(e_3)=e_4,& \alpha(e_4)=e_3+2e_4.
\end{array}
\end{align*}
\end{enumerate}
\end{exes}
\begin{prop}[\cite{Bakayoko2016arXiv:Hom-Novikovcoloralgebras}]\label{prop:twist.Hom.Nov.color.algeb}
Let $\mathcal{A}=(A, \cdot ,\varepsilon)$ be a Novikov color Hom-algebra and
$\alpha :\mathcal{A}\rightarrow \mathcal{A}$ be a Novikov color Hom-algebras
morphism. Define $\cdot_{\alpha}:A \times A\rightarrow A$ for all $x, y\in \mathcal{H}(A)$, by
$x\cdot _{\alpha}y =\alpha (x\cdot y)$.
Then, $\mathcal{A}_\alpha=(A, \cdot _{\alpha},\varepsilon, \alpha)$ is a Hom-Novikov color Hom-algebra called the $\alpha$-twist or Yau twist of $(A, \cdot,\varepsilon)$.
\end{prop}
In the following we introduce the notions of bimodule and matched pair of Novikov color Hom-algebras.
\begin{defn}
Let $(A, \cdot, \varepsilon,\alpha)$ be a Hom-Novikov color Hom-algebra, $(V,\beta)$ is a pair of $\Gamma$-graded linear space $V$ and an even linear map $\beta:V\rightarrow V$. Let $ l,r: A \rightarrow End(V) $ be two even linear maps. The quadruple $(l,r, \beta, V)$ is called a bimodule of $(A, \cdot, \varepsilon,\alpha)$ if for all $ x, y \in  \mathcal{H}(A), v \in \mathcal{H}(V) $,
\begin{align}
\label{Cond1}
l(x\cdot y)\beta(v)-l(\alpha(x))l(y)v=\varepsilon(x,y)\big(l(y\cdot x)\beta(v)-l(\alpha(y))l(x)v\big), \\
\label{Cond2}
r(\alpha(y))l(x)v-l(\alpha(x))r(y)v=\varepsilon(x,v)\big(r(\alpha(y))r(x)v-r(x\cdot y)\beta(v)\big),\\
\label{Cond3}
r(\alpha(y))r(x)v-r(x\cdot y)\beta(v)=\varepsilon(v,x)\big(r(\alpha(y))l(x)v-l(\alpha(x))r(y)v\big),\\
\label{Cond4}
l(x\cdot y)\beta(v)=\varepsilon(y,v)r(\alpha(y))l(x)v, \\
\label{Cond5}
r(\alpha(y))l(x)v=\varepsilon(v,y)l(x\cdot y)\beta(v),\\
\label{Cond6}
r(\alpha(y))r(x)v=\varepsilon(x,y)r(\alpha(x))r(y)v.
\end{align}
\end{defn}
\begin{prop}\label{semi direct Hom Novikov}
Let $(l, r, \beta, V)$ is a bimodule of a Hom-Novikov color Hom-algebra $( A, \cdot,\varepsilon, \alpha)$. Then the direct sum of $\Gamma$-graded linear spaces, $$ A \oplus V=\bigoplus_{\gamma\in\Gamma}(A \oplus V)_\gamma=\bigoplus_{\gamma\in\Gamma}(A_\gamma \oplus V_\gamma),$$
is turned into a Hom-Novikov color Hom-algebra by defining multiplication in $A \oplus V $ for all
$ X_1=x_{1}+v_1\in A_{\gamma_1}\oplus V_{\gamma_1},X_2=x_{2}+v_2\in A_{\gamma_2}\oplus V_{\gamma_2}$ by
\begin{eqnarray}\label{operationsemidirectNov1}
(x_{1} + v_{1})\ast (x_{2} + v_{2}) &=& x_{1} \cdot x_{2} + (l(x_{1})v_{2} + r(x_{2})v_{1}),\\
\label{operationsemidirectNov2}
(\alpha\oplus\beta)(x_{1} + v_{1})&=&\alpha(x_{1}) + \beta(v_{1}).
\end{eqnarray}
\end{prop}
\begin{proof} We prove the axioms \eqref{eqhomNovikov1} and \eqref{eqhomNovikov2} in $A\oplus V$.
For all $X_i=x_i+v_i\in A_{\gamma_i} \oplus V_{\gamma_i},\ i\in\{1;2;3\}$,
\begin{align*}\label{condit. du Bimod.}
&(X_1\ast X_2)\ast(\alpha+\beta)X_3-(\alpha+\beta) X_1\ast(X_2 \ast X_3)\\
&\quad -\varepsilon(X_1,X_2)\big((X_2\ast X_1)\ast(\alpha+\beta)X_3-(\alpha+\beta) X_2\ast(X_1\ast X_3)\big)\\
&=\big((x_1+v_1)\ast(x_2+v_2)\big)\ast(\alpha+\beta)(x_3+v_3)\\
&\quad -(\alpha+\beta)(x_1+v_1)\ast((x_2+v_2)\ast(x_3+v_3))\\
&\quad -\varepsilon(X_1,X_2)\Big(\big((x_2+v_2)\ast(x_1+v_1)\big)\ast(\alpha+\beta)(x_3+v_3)\\
&\quad -(\alpha+\beta)(x_2+v_2)\ast\big((x_1+v_1)\ast(x_3+v_3)\big)\Big)\\
&=\big(x_1\cdot x_2+l(x_1)v_2+r(x_2)v_1\big)\ast(\alpha(x_2)+\beta(v_3))\\
&\quad -(\alpha(x_1)+\beta(v_1))\big(x_2\cdot x_3+l(x_2)v_3+r(x_3)v_2\big)\\
&\quad -\varepsilon(x_1,x_2)\Big(\big(x_2\cdot x_1+l(x_2)v_1+r(x_1)v_2\big)\ast(\alpha(x_3)+\beta(v_3))\\
&\quad-(\alpha(x_2)+\beta(v_2))\ast\big(x_1\cdot x_3+l(x_1)v_3+r(x_3)v_1\big)\Big)\\
&=(x_1\cdot x_2)\cdot\alpha(x_3)+l(x_1\cdot x_2)\beta(v_3)+r(\alpha(x_3))l(x_1)v_2+r(\alpha(x_3))r(x_2)v_1\\
&\quad -\alpha(x_1)(x_2\cdot x_3)-l(\alpha(x_1))l(x_2)v_3-l(\alpha(x_1))r(x_3)v_2-r(x_2\cdot x_3)\beta(v_1)\\
&\quad -\varepsilon(x_1,x_2)\Big((x_2\cdot x_1)\cdot\alpha(x_3)+l(x_2\cdot x_1)\beta(v_3)+r(\alpha(x_3))l(x_2)v_1\\
&\quad+r(\alpha(x_3))r(x_1)v_2-\alpha(x_2)(x_1\cdot x_3)-l(\alpha(x_2))l(x_1)v_3-l(\alpha(x_2))r(x_3)v_1\\
&\hspace{5cm} -r(x_1\cdot x_3)\beta(v_2)\Big)\\
&=\underbrace{\Big((x_1\cdot x_2)\cdot\alpha(x_3)-\alpha(x_1)(x_2\cdot x_3)-\varepsilon(x_1,x_2)((x_2\cdot x_1)\cdot\alpha(x_3)-\alpha(x_2)\cdot(x_1\cdot x_3))\Big)}_{\text{$=0$ by \eqref{eqhomNovikov1} in $A$}}\\
&+\underbrace{\Big(l(x_1\cdot x_2)\beta(v_3)-l(\alpha(x_1))l(x_2)v_3-\varepsilon(x_1,x_2)(l(x_2\cdot x_1)\beta(v_3)-l(\alpha(x_2))l(x_1)v_3)\Big)}_{\text{$=0$ by \eqref{Cond1}}}\\
&+\underbrace{\Big(r(\alpha(x_3))l(x_1)v_2-l(\alpha(x_1))r(x_3)v_2-\varepsilon(x_1,v_2)(r(\alpha(x_3))r(x_1)v_2-r(x_1\cdot x_3)\beta(v_2))\Big)}_{\text{$=0$ by \eqref{Cond2}}}\\
&+\underbrace{\Big(r(\alpha(x_3))r(x_2)v_1-r(x_2\cdot x_3)\beta(v_1)-\varepsilon(v_1,x_2)(r(\alpha(x_3))l(x_2)v_1-l(\alpha(x_2))r(x_3)v_1)\Big)}_{\text{$=0$ by \eqref{Cond3}}}\\
&\quad=0,\\
&(X_1\ast X_2)\ast(\alpha+\beta) X_3-\varepsilon(X_2,X_3)(X_1\ast X_3)\ast(\alpha+\beta)X_2\\
&=\big((x_1+v_1)\ast(x_2+v_2)\big)\ast(\alpha+\beta)(x_3+v_3)\\
&\quad-\varepsilon(X_2,X_3)\big((x_1+v_1)\ast(x_3+v_3)\big)\ast(\alpha+\beta)(x_2+v_2)\\
&=\big(x_1\cdot x_2+l(x_1)v_2+r(x_2)v_1\big)\ast(\alpha(x_3)+\beta(v_3))\\
&\quad -\varepsilon(x_2,x_3)\big(x_1\cdot x_3+l(x_1)v_3+r(x_3)v_1\big)\ast(\alpha(x_2)+\beta(v_2))\\
&=(x_1\cdot x_2)\cdot\alpha(x_3)+l(x_1\cdot x_2)\beta(v_3)+r(\alpha(x_3))l(x_1)v_2+r(\alpha(x_3))r(x_2)v_1\\
&\quad -\varepsilon(x_2,x_3)((x_1\cdot x_3)\cdot\alpha(x_2)+l(x_1\cdot x_3)\beta(v_2)+r(\alpha(x_2))l(x_1)v_3
+r(\alpha(x_2))r(x_3)v_1)\\
&=\underbrace{\Big((x_1\cdot x_2)\cdot\alpha(x_3)-(x_1\cdot x_3)\cdot\alpha(x_2)\Big)}_{\text{$=0$ by \eqref{eqhomNovikov2} in $A$}}+\underbrace{\Big(l(x_1\cdot x_2)\beta(v_3)-\varepsilon(x_2,x_3)r(\alpha(x_2))l(x_1)v_3\Big)}_{\text{$=0$ by \eqref{Cond4}}}\\
&\quad+\underbrace{\Big(r(\alpha(x_2))l(x_1)v_3-\varepsilon(v_2,x_3)l(x_1\cdot x_3)\beta(v_2)\Big)}_{\text{$=0$ by \eqref{Cond5}}}\\
&\hspace{3cm}+\underbrace{\Big(r(\alpha(x_3))r(x_2)v_1-\varepsilon(x_2,x_3)r(\alpha(x_2))r(x_3)v_1\Big)}_{\text{$=0$ by \eqref{Cond6}}}=0.
\qedhere \end{align*}
\end{proof}
The Hom-Novikov color Hom-algebra constructed in Proposition \ref{semi direct Hom Novikov} is denoted
by $(A \oplus V, \ast, \varepsilon,\alpha + \beta)$ or $A \times_{l,r, \alpha, \beta} V.$
\begin{ex}
Let $(A,\cdot,\varepsilon,\alpha)$ be a Hom-Novikov color Hom-algebra. Then $(L,R,\alpha,A)$ is a bimodule of $(A,\cdot,\varepsilon,\alpha)$, where $L(x)y=x\cdot y$ and $R(x)y=y\cdot x$ for all $x,y\in \mathcal{H}(A)$, called the regular bimodule of $(A,\cdot,\varepsilon,\alpha)$.
\end{ex}
\begin{prop}
\label{thm:matchedpairs}
Let $( A, \cdot_A, \varepsilon,\alpha)$ and $(B,\cdot_B, \varepsilon, \beta)$ be two
Hom-Novikov color Hom-algebras. Suppose there are even linear maps $l_{  A}, r_{
 A}:  A\rightarrow End(B)$ and $l_{B}, r_{
B} :  B\rightarrow End(A)$ such that the quadruple
 $(l_{  A}, r_{
 A},  \beta, B)$ is
a bimodule of $ A,$ and $(l_{ B}, r_{B}, \alpha,
 A)$ is a bimodule of $B,$
satisfying, for any $ x, y \in   \mathcal{H}(A), a,b \in  \mathcal{H}(B)$, the following conditions:
\begin{eqnarray}
\label{match. pair1}
&\begin{array}{l}
r_A(\alpha(x))(a\cdot_B b)-\beta(a)\cdot_B(r_A(x)b)-r_A(l_B(b)x)\beta(a)\\
\quad =\varepsilon(a,b)\big(r_A(\alpha(x))(b\cdot_B a)-\beta(b)\cdot_B(r_A(x)a)-r_A(l_B(a)x)\beta(b)\big),
\end{array}
\\[0.2cm]
\label{match. pair2}
&
\begin{array}{l}
(r_A(x)a)\cdot_B\beta(b)+l_A(l_B(a)x)\beta(b)-\beta(a)\cdot_B(l_A(x)b)-r_A(r_B(b)x)\beta(a)\\
\quad =\varepsilon(a,x)\big((l_A(x)a)\cdot_B\beta(b)+l_A(r_B(a)x)\beta(b)-l_A(\alpha(x))(a\cdot_B b)\big),
\end{array}
\\[0.2cm]
\label{match. pair3}
&\begin{array}{l}
(l_A(x)a)\cdot_B\beta(b)-l_A(r_B(a)x)\beta(b)-l_A(\alpha(x))(a\cdot_B b)\\
\quad =\varepsilon(x,a)\big((r_A(x)a)\cdot_B\beta(b)+l_A(l_B(a)x)\beta(b)-\beta(a)\cdot_B(l_A(x)b)\\
\hspace{8cm} \big.-r_A(r_B(b)x)\beta(a)\big),
\end{array}
\\[0.2cm]
\label{match. pair4}
& \begin{array}{l}
r_B(\beta(a))(x\cdot_A y)-\alpha(x)\cdot_A(r_B(a)y)-r_B(l_A(y)a)\alpha(x)\\
\quad =\varepsilon(x,y)\big(r_B(\beta(a))(y\cdot_A x)-\alpha(y)\cdot_A(r_B(a)x)-r_B(l_A(x)a)\alpha(y)\big),
\end{array}
\\[0.2cm]
\label{match. pair5}
&
\begin{array}{l}
(r_B(a)x)\cdot_A\alpha(y)+l_B(l_A(x)a)\alpha(y)-\alpha(x)\cdot_A(l_B(a)y)-r_B(r_A(y)a)\alpha(x) \\
\quad =\varepsilon(x,a)\big((l_B(a)x)\cdot_A\alpha(y)+l_B(r_A(x)a)\alpha(y)-l_B(\beta(a))(x\cdot_A y)\big),
\end{array}\\[0.2cm]
\label{match. pair6}
&
\begin{array}{l}
(l_B(a)x)\cdot_A\alpha(y)-l_B(r_A(x)a)\alpha(y)-l_B(\beta(a))(x\cdot_A y)\\
\quad =\varepsilon(a,x)\big((r_B(a)x)\cdot_A\alpha(y)+l_B(l_A(x)a)\alpha(y)-\alpha(x)\cdot_A(l_B(a)y)\big.\\
\hspace{8cm}
\big.-r_B(r_A(y)a)\alpha(x)\big).
\end{array}
\end{eqnarray}
Then, $(A,B,l_A,r_A,\beta,l_B,r_B,\alpha)$ is called a matched pair of
Hom-Novikov color Hom-algebras. In this case, there is a Hom-Novikov color Hom-algebra structure on the direct sum $ A \oplus B=\bigoplus_{\gamma\in\Gamma}(A \oplus B)_\gamma=\bigoplus_{\gamma\in\Gamma}(A_\gamma \oplus B_\gamma),$ of
the underlying $\Gamma$-graded linear spaces of $ A$ and $B$ given for all $ x+a\in A_{\gamma_1}\oplus B_{\gamma_1}, y+b\in A_{\gamma_2}\oplus B_{\gamma_2}$ by
\begin{equation} \label{matchedpairproduct}
\begin{array}{l}
(x + a)\cdot (y + b) =
\big(x\cdot_A y + l_{ B}(a)y + r_{ B}(b)x\big)
+ \big (a\cdot_B b +  l_{A}(x)b +  r_{  A}(y)a\big),\\
(\alpha\oplus\beta)(x + a)=\alpha(x) + \beta(a).
\end{array}
\end{equation}
\end{prop}
\noindent We denote this Hom-Novikov color Hom-algebra by
$( A\bowtie B, \cdot,\varepsilon, \alpha + \beta)$ or $ A \bowtie^{l_{  A}, r_{  A}, \beta}_{l_{ B}, r_{ B}, \alpha} B.$
\subsection{On Hom-Lie color Hom-algebras}
\begin{defn}[\cite{AbdaouiAmmarMakhloufCohhomLiecolalg2015,CaoChen2012:SplitregularhomLiecoloralg,CalderonDelgado2012:splLiecolor,LarssonSilv2005:QuasiLieAlg,LarssonSilv:GradedquasiLiealg,SigSilv:CzechJP2006:GradedquasiLiealgWitt,LY}]
A Hom-Lie color Hom-algebra is a quadruple
$(A, [\cdot,\cdot],\varepsilon, \alpha)$ consisting of a $\Gamma$-graded vector space $A$, a bi-character $\varepsilon$, an even bilinear
mapping
$[\cdot,\cdot] :A\times A\rightarrow A$,
(i.e. $[A_a, A_b]\subseteq A_{a+b}$, for all $a, b \in \Gamma$) and an even homomorphism $\alpha:A\rightarrow A$ such that for homogeneous elements $x, y, z \in A$ we have
\begin{align}
    &[x,y]=-\varepsilon(x,y)[y,x], && \quad \text{$\varepsilon$-skew symmetry}, \\
    &\circlearrowleft_{x,y,z} \varepsilon(z,x)[\alpha(x),[y,z]]=0, && \quad \text{$\varepsilon$-Hom-Jacobi identity}
\end{align}
where $\circlearrowleft_{x,y,z}$
denotes the cyclic sum over $(x, y, z)$.
\end{defn}
\begin{rmk}
Hom-Lie color Hom-algebras contain ordinary Lie color algebras, Lie superalgebras and Lie algebras, as well as Hom-Lie superalgebras and Hom-Lie algebras for specific choices of the twisting map, grading group and commutation factor.
\begin{enumerate}
\item
When $\alpha =id_A$, one recovers Lie color algebras, and in particular if the grading group is $\mathbb{Z}_2$ and the commutation factor is defined as $\varepsilon(i,j)=(-1)^{ij}$ for all $i,j \in \mathbb{Z}$, then one gets Lie superalgebras.  \cite{BahturinMikhPetrZaicevIDLSbk92,ChenSilvestrovOystaeyen:RepsCocycleTwistsColorLie,ChenPetitOystaeyenCOHCHLA,MikhZolotykhCALSbk95,PiontkovskiSilvestrovC3dCLA,ScheunertGLA,ScheunertCOH2,ScheunertZHA}.
\item
 When $\alpha =id_A$ and $A$ is trivially graded, by the group with one element, we recover Lie algebras.
\item
 When $A$ is trivially graded, while $\alpha$ is an arbitrary linear map, we recover Hom-Lie algebras \cite{HartwigLarSil:defLiesigmaderiv,LarssonSilvJA2005:QuasiHomLieCentExt2cocyid,LarssonSilv:GradedquasiLiealg,LarssonSilv2005:QuasiLieAlg,ms:homstructure}, and if $A$ is graded by the group of two elements $\mathbb{Z}_2$, while $\alpha$ is an arbitrary even linear map, and the commutation factor is defined as $\varepsilon(i,j)=(-1)^{ij}$ for all $i,j \in \mathbb{Z}_2$, then we get Hom-Lie superalgebras.
 \end{enumerate}
\end{rmk}
\begin{prop}[ \cite{Bakayoko2016arXiv:Hom-Novikovcoloralgebras}]\label{Hom-Novik to Hom-Lie}
Let $(A,\cdot,\varepsilon,\alpha)$ be a Hom-Novikov color Hom-algebra. Then, there exists a Hom-Lie color
algebra structure on $ A $ given for all $x, y \in \mathcal{H}(A)$ by
\begin{equation}\label{eqNovtoLie}
    [x, y] = x \cdot y -\varepsilon(x,y) y \cdot x.
\end{equation}
\end{prop}
\begin{ex}
Let $A = A_{0} \oplus A_{1} =< e_1 > \oplus < e_2,e_3>$
 be a $3$-dimensional superspace. The quintuple $(A, \cdot,\varepsilon,\alpha)$ is a Hom-Novikov color Hom-algebra with
\begin{align*}
&\text{the multiplication:} \quad  e_1\cdot e_2=e_3,\quad e_2\cdot e_1= -e_3, \\
&\text{the bicharacter:} \quad \varepsilon(i,j)=(-1)^{ij},\\
&\text{the even linear map $\alpha:A\rightarrow A$ defined by} \quad
\begin{array}[t]{l}
\alpha(e_1)=-2e_1, \quad \alpha(e_2)=e_3,\\
\alpha(e_3)=e_2-e_3.
\end{array}
\end{align*}
Therefore, by Proposition \ref{Hom-Novik to Hom-Lie}, $(A,[\cdot,\cdot],\varepsilon,\alpha)$ is a Hom-Lie color Hom-algebra with $$[e_1,e_2]=-[e_2,e_1]=2e_3.$$
\end{ex}
\begin{prop}[\cite{AbdaouiAmmarMakhloufCohhomLiecolalg2015}]\label{twist Lie color alg}
Let $\mathcal{A}=(A, [\cdot,\cdot] ,\varepsilon)$ be a Hom-Lie color Hom-algebra and
$\alpha :\mathcal{A}\rightarrow \mathcal{A}$ be a Hom-Lie color Hom-algebras
morphism. Define $[\cdot,\cdot]_{\alpha}:A \times A\rightarrow A$ for all $x, y\in A$, by
$[x,y]_\alpha =\alpha ([x,y])$.
Then, $\mathcal{A}_\alpha=(A, [\cdot,\cdot]_{\alpha},\varepsilon, \alpha)$ is a Hom-Lie color Hom-algebra called the $\alpha$-twist or Yau twist of $(A, [\cdot,\cdot],\varepsilon)$.
\end{prop}
\begin{defn}
Let $(A,[\cdot,\cdot],\varepsilon,\alpha)$ be a Hom-Lie color Hom-algebra, $(V,\beta)$ is a pair of $\Gamma$-graded linear space $V$ and an even linear map $\beta:V\rightarrow V$. Let $ \rho: A \rightarrow End(V) $ an even linear map. The triple $(\rho, \beta, V)$ is called a representation of $(A, [\cdot,\cdot], \varepsilon,\alpha)$ if for all $ x, y \in  \mathcal{H}(A), v \in \mathcal{H}(V) $,
\begin{align}
\label{Cond1 Hom Lie}
\rho([x,y])\beta(v)=\rho(\alpha(x))\rho(y)v-\varepsilon(x,y)\rho(\alpha(y))\rho(x)v. \end{align}
\end{defn}
\begin{prop}\label{semi direct Hom-Lie}
Let $(\rho, \beta, V)$ is a representation of a Hom-Lie color Hom-algebra $( A, [\cdot,\cdot],\varepsilon, \alpha)$. Then the direct sum of $\Gamma$-graded linear spaces, $$ A \oplus V=\bigoplus_{\gamma\in\Gamma}(A \oplus V)_\gamma=\bigoplus_{\gamma\in\Gamma}(A_\gamma \oplus V_\gamma),$$  is turned into a Hom-Lie color Hom-algebra by defining multiplication in $A \oplus V $ for all $  X_1=x_{1}+v_1\in A_{\gamma_1}\oplus V_{\gamma_1},X_2=x_{2}+v_2\in A_{\gamma_2}\oplus V_{\gamma_2}$ by
\begin{eqnarray}\label{operationsemidirectLie1}
[x_{1} + v_{1},x_{2} + v_{2}]' &=&[ x_{1} , x_{2}] + \rho(x_{1})v_{2} -\varepsilon(v_1,x_2) \rho(x_{2})v_{1},\\
\label{operationsemidirectLie2}
(\alpha\oplus\beta)(x_{1} + v_{1})&=&\alpha(x_{1}) + \beta(v_{1}).
\end{eqnarray}
\end{prop}
The Hom-Lie color Hom-algebra constructed in previous Proposition is denoted
by $(A \oplus V, [\cdot,\cdot]', \varepsilon,\alpha + \beta)$ or $A \ltimes_{\rho, \alpha, \beta} V.$
\begin{ex}
Let $(A,[\cdot,\cdot],\varepsilon,\alpha)$ be a Hom-Lie algebra. Then $(ad,\alpha,A)$ is a representation of $(A,[\cdot,\cdot],\varepsilon,\alpha)$, where $ad(x)y=[x, y]$ for all $x,y\in \mathcal{H}(A)$, called the adjoint representation of $(A,[\cdot,\cdot],\varepsilon,\alpha)$.
\end{ex}
Now, we introduce the notion of matched pair of Hom-Lie color Hom-algebra
\begin{prop}\label{matched Lie}
Suppose that $(A,[\cdot,\cdot]_A,\varepsilon,\alpha)$ and $(B,[\cdot,\cdot]_{B},\varepsilon,\beta)$ are Hom-Lie color Hom-algebras, and there are even linear maps $\rho_A:A\rightarrow End(B)$
and $\rho_B:B\rightarrow End(A)$ such that $(\rho_A,\beta,B)$ is a representation of $A$ and $(\rho_B,\alpha,A)$ is a representation of $B$ satisfying for any $x,y\in \mathcal{H}(A),~a,b\in \mathcal{H}(B)$,
\begin{eqnarray}
\begin{array}{r}
\varepsilon(x,a)\big(\rho_A(\rho_B(a)x)\beta(b)-[\beta(a),\rho_A(x)b]_B\big)+\varepsilon(a+x,b)\big([\beta(b),\rho_A(x)a]_B\\
-\rho_A(\rho_B(b)x)\beta(a)\big)+\rho_A(\alpha(x))([a,b]_B)=0,
\end{array}
\\[0.2cm]
\begin{array}{r}
\varepsilon(a,x)\big(\rho_B(\rho_A(x)a)\alpha(y)-[\alpha(x),\rho_B(a)y]_A\big)+\varepsilon(x+a,y)\big([\alpha(y),\rho_B(a)x]_A\nonumber\\
-\rho_B(\rho_A(y)a)\alpha(x)\big)+\rho_B(\beta(a))([x,y]_A)=0.
\end{array}
\end{eqnarray}
Then, $(A,B,\rho_A,\beta,\rho_B,\alpha)$ is called a matched pair of
Hom-Lie color Hom-algebras. In this case, there is a Hom-Lie color Hom-algebra structure on the linear
space of the underlying $\Gamma$-graded linear spaces of $A$ and $B$,
$$ A \oplus B=\bigoplus_{\gamma\in\Gamma}(A \oplus B)_\gamma=\bigoplus_{\gamma\in\Gamma}(A_\gamma \oplus B_\gamma),$$
given for all $x+a\in A_{\gamma_1}\oplus B_{\gamma_1}, y+b\in A_{\gamma_2}\oplus B_{\gamma_2}$ by
\begin{eqnarray}
&&\begin{array}{ll}
[x+a,y+b]= & [x,y]_A+\rho_A(x)b-\varepsilon(a,y)\rho_A(y)a \\
& \hspace{1cm}+ [a,b]_B+\rho_B(a)y-\varepsilon(x,b)\rho_B(b)x,
\end{array} \\
&&(\alpha\oplus\beta)(x + a)=\alpha(x) + \beta(a).
\end{eqnarray}
\end{prop}
\section{Admissible Hom-Novikov-Poisson color Hom-algebras}
\label{sec:admissHomNovikovcoloralgebs}
In this Section, we recall the main result of Hom-Novikov-Poisson color Hom-algebras in \cite{attan:SomeconstructionsofcolorHom-Novikov-Poissonalgeb2019} and we introduce their notions of bimodules and matched pairs. Next, we introduce the dfinition of admissible Hom-Novikov-Poisson color Hom-algebras and we give some explicit constructions. Finally, we  show  that  the  much  larger  class  of  admissible  Hom-Novikov-Poisson  color Hom-algebras is also closed under tensor products
\subsection{Constructions and bimodules of (admissible) Hom-Novikov-Poisson color Hom-algebras}
\begin{defn}[\cite{attan:SomeconstructionsofcolorHom-Novikov-Poissonalgeb2019}]\label{def Hom-Nov Poss}
Hom-Novikov-Poisson color Hom-algebras are defined as quintuples $(A,\cdot,\diamond,\varepsilon,\alpha)$ consisting of an $\varepsilon$-commutative Hom-associative color Hom-algebra $(A,\cdot,\varepsilon,\alpha)$ and a Hom-Novikov color Hom-algebra $(A,\diamond,\varepsilon,\alpha)$ obeying for $x,y,z\in \mathcal{H}(A)$,
\begin{align}\label{HomNovPoisscoloralg1}
    (x\cdot y)\diamond\alpha(z) &=\varepsilon(y,z)(x\diamond z)\cdot\alpha(y),\\
    \label{HomNovPoisscoloralg2}
    (x\diamond y)\cdot \alpha(z) -\alpha(x)\diamond(y\cdot z) &=\varepsilon(x,y)\big( (y\diamond x)\cdot\alpha(z)-\alpha(y)\diamond(x\cdot z)\big).
\end{align}
A Hom-Novikov-Poisson color Hom-algebra is called multiplicative if
the linear map $\alpha:A\rightarrow A$ is multiplicative with respect to $\cdot$ and $\diamond$, that is,
for all $x,y \in \mathcal{H}(A)$,
$$
\alpha (x \cdot y) =\alpha (x) \cdot \alpha (y), \quad \alpha (x\diamond y) =\alpha (x) \diamond \alpha (y).
$$
\end{defn}
\begin{rmk}
Hom-Novikov-Poisson color Hom-algebras contain Novikov-Poisson color Hom-algebras, Hom-Novikov-Poisson algebras and Novikov-Poisson algebras for special choices of the twising map and grading group.
\begin{enumerate}
\item
When $\alpha = id$, we get Novikov-Poisson color Hom-algebra.
\item
When $\Gamma = \{e\}$ and $\alpha \neq id$, we get Hom-Novikov-Poisson algebra \cite{Yau:homnovikovpoissonalg}.
\item
When $\Gamma = \{e\}$ and $\alpha= id$, we get Novikov-Poisson algebra \cite{xu1,xu2}.
 \end{enumerate}
\end{rmk}
\begin{ex}
Let $A = A_{0} \oplus A_{1} =< e_1,e_2> \oplus <e_3,e_4>$
 be a $4$-dimensional superspace. Then
$(A,\cdot,\diamond,\varepsilon,\alpha)$ is a Hom-Novikov-Poisson color Hom-algebra with
\begin{align*}
&\text{the bicharacter} \quad \varepsilon(i,j)=(-1)^{ij}, \\
&\text{the multiplications:}
\begin{array}[t]{lllll}
e_2\cdot e_2=\lambda_1 e_1, \quad e_2\cdot e_4=e_4\cdot e_2=\lambda_2 e_3,\quad\lambda_i\in\mathbb{K},\\
e_2\diamond e_4=\mu_2e_3, \quad e_4\diamond e_2=\mu_3e_3,\quad
e_4\diamond e_4=\mu_4 e_1,\quad \mu_i\in\mathbb{K},
\end{array} \\
&\text{the even linear map $\alpha: A \rightarrow A$ defined by}
\begin{array}[t]{l}
\alpha(e_1)=2e_1,\quad \alpha(e_2)=e_2-e_1,\\
\alpha(e_3)=-e_4,\quad \alpha(e_4)=e_3.
\end{array}
\end{align*} \end{ex}
\begin{defn}\label{morphHom-Nov-Poisscoloralgebra}
Let $( A, \cdot, \diamond, \alpha)$ and $( A', \cdot',  \diamond', \alpha')$ be Hom-Novikov-Poisson color Hom-algebras. A linear map of degree zero $f:  A\rightarrow  A'$ is a Hom-Novikov-Poisson color Hom-algebra morphism if
\begin{eqnarray*}
\cdot'\circ(f\otimes f)= f\circ\cdot,\quad \diamond'\circ(f\otimes f)= f\circ\diamond \mbox{ and } f\circ\alpha= \alpha'\circ f.
\end{eqnarray*}
\end{defn}
\begin{thm}[\cite{attan:SomeconstructionsofcolorHom-Novikov-Poissonalgeb2019}]\label{twist-Hom-Nov-Poiss}
Let $\mathcal{A}=(A, \cdot ,\diamond,\varepsilon)$ be a Hom-Novikov-Poisson color Hom-algebra and
$\alpha :\mathcal{A}\rightarrow \mathcal{A}$ be a Hom-Novikov-Poisson color Hom-algebras
morphism. Define $\cdot_{\alpha}, \ \diamond_\alpha:A \times A\rightarrow A$ for all $x, y\in \mathcal{H}(A)$, by
$x\cdot _{\alpha}y =\alpha (x\cdot y)$ and $x\diamond_{\alpha}y =\alpha(x\diamond  y)$.
Then, $\mathcal{A}_\alpha=(A_{\alpha}=A, \cdot _{\alpha},\diamond_\alpha,\varepsilon, \alpha)$ is a Hom-Novikov-Poisson color Hom-algebra called the $\alpha$-twist or Yau twist of $(A, \cdot ,\diamond,\varepsilon)$.
\end{thm}

\begin{defn}
Let $(A,\cdot,\diamond,\varepsilon,\alpha)$ be a Hom-Novikov-Poisson color Hom-algebra. A bimodule of $(A,\cdot,\diamond,\varepsilon,\alpha)$ is a quintuple $(s,l,r,\beta,V)$ such that $(s,\beta,V)$ is a bimodule of the $\varepsilon$-commutative Hom-associative color Hom-algebra  $(A,\cdot,\varepsilon,\alpha)$ and $(l,r,\beta,V)$ is a bimodule of the Hom-Novikov color Hom-algebra $(A,\diamond,\varepsilon,\alpha)$ and  satisfying, for all $ x, y \in  \mathcal{H}(A), v \in \mathcal{H}(V) $,
\begin{align}
\label{Condit1}
& l(x\cdot y)\beta(v)=\varepsilon(x,y)s(\alpha(y))l(x)v,\\*[0,2cm]
\label{Condit2}
& r(\alpha(y))s(x)v=\varepsilon(v,y)s(x\diamond y)\beta(v),\\*[0,2cm]
\label{Condit3}
& r(\alpha(y))s(x)v=s(\alpha(x))r(y)v,\\*[0,2cm]
\label{Condit4}
& s(x\diamond y)\beta(v)-l(\alpha(x))s(y)v=\varepsilon(x,y)\big(s(y\diamond x)\beta(v)-l(\alpha(y))s(x)v\big),\\*[0,2cm]
\label{Condit5}
& \begin{array}{l}
\varepsilon(x+v,y)\big(s(\alpha(y))l(x)v-\varepsilon(x,v)s(\alpha(y))r(x)v\big)=\varepsilon(v,y)l(\alpha(x))s(y)v\\
\hspace{9cm} -\varepsilon(x,v)r(x\cdot y)\beta(v).
\end{array}
\end{align}
\end{defn}
\begin{prop}\label{pro2:dirsumHomNovikovPoissoncolorHomalg}
 Let $A\oplus V=\bigoplus_{\gamma\in\Gamma}(A \oplus V)_\gamma=\bigoplus_{\gamma\in\Gamma}(A_\gamma \oplus V_\gamma)$, the direct sum of $\Gamma$-graded linear spaces. Then, $(A\oplus V,\cdot',\diamond',\varepsilon,\alpha+\beta)$ is a Hom-Novikov-Poisson color Hom-algebra, where $(A\oplus V,\cdot',\varepsilon,\alpha+\beta)$ is the semi-direct product $\varepsilon$-commutative Hom-associative color Hom-algebra $A\ltimes_{s,\alpha,\beta} V$ and $(A\oplus V,\diamond',\varepsilon,\alpha+\beta)$ is the semi-direct product Hom-Novikov color Hom-algebra $A\ltimes_{l,r,\alpha,\beta} V$.
\end{prop}
\begin{proof}
Let $(A,\cdot,\diamond,\varepsilon,\alpha)$ be a Hom-Novikov-Poisson color Hom-algebra, and $(s,l,r,\beta,V)$ be a bimodule. By Proposition \ref{prop:semidirect.comm.Hom.ass.color.algeb} and Proposition \ref{semi direct Hom Novikov},
$(A\oplus V,\cdot',\varepsilon,\alpha+\beta)$ is a $\varepsilon$-commutative Hom-associative color Hom-algebra, and
$(A\oplus V,\diamond',\varepsilon,\alpha+\beta)$ is a Hom-Novikov color Hom-algebra respectively.
Now, we show that the compatibility conditions \eqref{HomNovPoisscoloralg1}-\eqref{HomNovPoisscoloralg2} are satisfied.
For all $X_i=x_i+v_i\in A_{\gamma_i} \oplus V_{\gamma_i},\ i=1,2,3$ we have
\begin{align*}
&(X_1\cdot'X_2\big)\diamond'(\alpha+\beta) X_3-\varepsilon(X_2,X_3)(X_1\diamond' X_3)\cdot'(\alpha+\beta)X_2\\
&\quad=\big((x_1+v_1)\cdot'(x_2+v_2)\big)\diamond'(\alpha+\beta)(x_3+v_3)\\&\quad\quad-\varepsilon(X_2,X_3)\big((x_1+v_1)\diamond'(x_3+v_3)\big)\cdot'(\alpha+\beta)(x_2+v_2)\\
&\quad=\big(x_1\cdot x_2+s(x_1)v_2+\varepsilon(v_1,x_2)s(x_2)v_1\big)\diamond'(\alpha(x_3)+\beta(v_3))\\
&\quad\quad-\varepsilon(x_2,x_3)\big((x_1\diamond x_3+l(x_1)v_3+r(x_3)v_1\big)\cdot'(\alpha(x_2)+\beta(v_2))\\
&\quad=(x_1\cdot x_3)\diamond\alpha(x_3)+l(x_1\cdot x_2)\beta(v_3)+r(\alpha(x_3))s(x_1)v_2\\&\quad\quad+\varepsilon(v_1,x_2)r(\alpha(x_3))s(x_2)v_1-\varepsilon(x_2,x_3)\big(x_1\diamond x_3)\cdot\alpha(x_2)+s(x_1\diamond x_2)\beta(v_2)\big)\\
&\quad\quad-\varepsilon(x_1,x_2)\big(s(\alpha(x_2))l(x_1)v_3+s(\alpha(x_2))r(x_3)v_1\big)\\
&\quad=\underbrace{\Big((x_1\cdot x_2)\diamond\alpha(x_3)-\varepsilon(x_2,x_3)(x_1\diamond x_3)\cdot\alpha(x_2)\Big)}_{\text{$=0$ by \eqref{HomNovPoisscoloralg1}}}\\
&\quad\quad+\underbrace{\Big(l(x_1\cdot x_2)\beta(v_3)-\varepsilon(x_1,x_2)s(\alpha(x_2))l(x_1)v_3\Big)}_{\text{$=0$ by \eqref{Condit1}}}\\
&\quad\quad+\underbrace{\Big(r(\alpha(x_3))s(x_1)v_2-\varepsilon(x_2,x_3)s(x_1\diamond x_3)\beta(v_2)\Big)}_{\text{$=0$ by \eqref{Condit2} and Remark \ref{remarkepsilon}}}\\
&\quad\quad+\underbrace{\varepsilon(x_1,x_2)\Big(r(\alpha(x_3))s(x_2)v_1-s(\alpha(x_2))r(x_3)v_1\Big)}_{\text{$=0$ by \eqref{Condit3}}}=0,\\*[0,2cm]
&(X_1\diamond'X_2)\cdot'(\alpha+\beta)X_3-(\alpha+\beta)X_1\diamond'(X_2\cdot'X_3)\\
&-\varepsilon(X_1,X_2)\big((X_2\diamond'X_1)\cdot'(\alpha+\beta)X_3-(\alpha+\beta)X_2\diamond'(X_1\cdot'X_3)\big)\\
&\quad=\big((x_1+v_1)\diamond'(x_2+v_2)\big)\cdot'(\alpha+\beta)(x_3+v_3)\\&\quad\quad-(\alpha+\beta)(x_1+v_1)\diamond'\big((x_2+v_2)\cdot'(x_3+v_3)\big)\\
&\quad\quad-\varepsilon(X_1,X_2)\big((x_2+v_2)\diamond'(x_1+v_1)\big)\cdot'(\alpha+\beta)(x_3+v_3)\\
&\quad\quad-(\alpha+\beta)(x_2+v_2)\diamond'\big((x_1+v_1)\cdot'(x_3+v_3)\big)\\
&\quad=(x_1\diamond x_2)\alpha(x_3)+s(x_1\diamond x_2)\beta(v_3)+\varepsilon(x_1+x_2,v_3)s(\alpha(x_3))l(x_1)v_2\\
&\quad\quad+\varepsilon(x_1+x_2,v_3)s(\alpha(x_3))r(x_2)v_1-\alpha(x_1)\diamond(x_2\cdot x_3)\\
&\quad\quad-l(\alpha(x_1))s(x_2)v_3+\varepsilon(v_2,x_3)l(\alpha(x_1))s(x_3)v_2+r(x_2\cdot x_3)\beta(v_1)\\
&\quad\quad-\varepsilon(x_1,x_2)\Big(x_2\diamond x_1)\cdot\alpha(x_3)+s(x_2\diamond x_1)\beta(v_3)+\varepsilon(x_1+x_2,v_3)s(\alpha(x_3))l(x_2)v_1\\&\quad\quad+
\varepsilon(x_1+x_2,v_3)s(\alpha(x_3))r(x_1)v_2-\alpha(x_2)\diamond(x_1\cdot x_3)-l(\alpha(x_2))s(x_1)v_3\\
&\quad\quad+\varepsilon(v_1,x_3)l(\alpha(x_2))s(x_3)v_1+r(x_1\cdot x_3)\beta(v_2)\Big)\\
&\quad=\underbrace{\Big((x_1\diamond x_2)\cdot\alpha(x_3)-\alpha(x_2)\diamond(x_2\cdot x_3)} \\ &\hspace{4cm}-\underbrace{\varepsilon(x_1,x_2)\big((x_2\diamond x_1)\cdot\alpha(x_3)-\alpha(x_2)\diamond(x_1\cdot x_3))\big)\Big)}_{\text{$=0$ by \eqref{HomNovPoisscoloralg2}}}\\
&\quad\quad+\underbrace{\Big(s(x_1\diamond x_2)\beta(v_3)-l(\alpha(x_1))s(x_2)v_3}\\
&\hspace{4cm} -\underbrace{\varepsilon(x_1,x_2)(s(x_2\diamond x_1)\beta(v_3)-l(\alpha(x_2))s(x_1)v_3\Big)}_{\text{$=0$ by \eqref{Condit4}}}\\*[0,5cm]
&\quad\quad+\underbrace{\Big(\varepsilon(x_1+x_2,v_3)\big(s(\alpha(x_3))l(x_1)v_2-\varepsilon(x_1,x_2)s(\alpha(x_3))r(x_1)v_2\big)}\\&\hspace{5cm}-\underbrace{\varepsilon(x_1,x_2)r(x_1\cdot x_3)\beta(v_2)+\varepsilon(x_2,v_3)l(\alpha(x_1))s(x_3)v_2\Big)}_{\text{$=0$ by \eqref{Condit5} and Remark \ref{remarkepsilon}}}\\*[0,5cm]
&\quad\quad-\underbrace{\varepsilon(x_1,x_2)\Big(\varepsilon(x_1+x_2,v_3)\big(s(\alpha(x_3))l(\alpha(x_2))v_1-\varepsilon(x_2,x_1)s(\alpha(x_3))r(x_2)v_1\big)}\\
&\hspace{4cm}-\underbrace{\varepsilon(x_1,x_3)l(\alpha(x_2))s(x_3)v_1+\varepsilon(x_2,x_1)r(x_2\cdot x_3)\beta(v_1)\Big)}_{\text{$=0$ by \eqref{Condit5} and Remark \ref{remarkepsilon}}}=0.
\end{align*}
Hence, $(A\oplus V,\cdot',[\cdot,\cdot]',\varepsilon,\alpha+\beta)$ is a Hom-Novikov-Poisson color Hom-algebra.
\end{proof}
The Hom-Novikov-Poisson color Hom-algebra constructed in previous Proposition is denoted by $A \ltimes_{s,l,r, \alpha, \beta} V.$
\begin{exes}
\begin{enumerate}
\item
Let $(A,\cdot,\diamond,\varepsilon,\alpha)$ be a Hom-Novikov-Poisson color Hom-algebra, and let $S(x)y=x\cdot y=\varepsilon(x,y)y\cdot x$, $L(x)y=x\diamond y$ and $R(x,y)=y\diamond x$, for all $x,y\in \mathcal{H}(A)$.
Then, $(S,L,R,\alpha,A)$ is a bimodule of $(A,\cdot,\diamond,\varepsilon,\alpha)$, called the regular bimodule of $(A,\cdot,\diamond,\varepsilon,\alpha)$.
\item
If $f:\mathcal{A}=(A,\cdot_1,\diamond_1,\varepsilon,\alpha)\longrightarrow(A',\cdot_2,\diamond_2,\varepsilon,\beta)$ is a Hom-Novikov-Poisson color Hom-algebras morphism, then
$(s,l,r,\beta,A')$
becomes a bimodule of $\mathcal{A}$ via $f$, that is, for all $(x,y)\in \mathcal{H}(A)\times \mathcal{H}(A')$,
$s(x)y=f(x)\cdot_2 y,$ \ $l(x)y=f(x)\diamond_2 y,$ \ $r(x)y=y \diamond_2 f(x).$
\end{enumerate}
\end{exes}

\begin{thm}
Suppose that $\mathcal{A}=(A,\cdot_A,\diamond_A,\varepsilon,\alpha)$ and $\mathcal{B}=(B,\cdot_B,\diamond_{B},\varepsilon,\beta)$ are two Hom-Novikov-Poisson color Hom-algebras, and there are even linear maps $s_A,l_A,r_A:A\rightarrow End(B)$
and $s_B,l_B,r_B:B\rightarrow End(A)$ such that $A\bowtie^{s_A,\beta}_{s_B,\alpha}B$ is a matched pair of $\varepsilon$-commutative Hom-associative color Hom-algebras,   $A\bowtie^{l_A,r_A,\beta}_{l_B,r_B,\alpha}B$ is a matched pair of Hom-Novikov color Hom-algebras, and for all $x,y\in \mathcal{H}(A),~a,b\in \mathcal{H}(B)$, the following equalities hold:
\begin{align}
&\ r_A(\alpha(x))(a\cdot_B b)=\varepsilon(b,x)\big(r_A(x)a\cdot_B\beta(b)+s_A(l_B(a)x)\beta(b),\\*[0,2cm]
&\ l_A(s_B(a)x)\beta(b)+\varepsilon(a,x)s(x)a\diamond_B\beta(b)=\varepsilon(x,b)\big(\varepsilon(a+b,x)s(\alpha(x))(a\diamond_B b)\big),\\*[0,2cm]
&\begin{array}{l} \varepsilon(a,x)l_A(s_B(a)x)\beta(b)+s_A(x)a\diamond_B\beta(b)\\
\hspace{4cm} =\varepsilon(a,b)\big(l_A(x)b\cdot_B\beta(a)+s_A(r_B(b)x)\beta(a)\big),
\end{array}\\*[0,2cm]
&\begin{array}{l}
\varepsilon(a+b,x)s(\alpha(x))(a\diamond_B b)-\varepsilon(b,x)\beta(a)\diamond_B s_A(x)b-r_A(s_B(b)x)\beta(a)\\
\quad =\varepsilon(a,b)\big(\varepsilon(a+b,x)s(\alpha(x))(b\diamond_B a)
-\varepsilon(a,x)\beta(b)\diamond_B(s_A(x)a)\big.\\
\hspace{8cm} \big.-r_A(s_B(a)x)\beta(b)\big),
\end{array}\\
&
\begin{array}{l}
r_A(x)b\cdot_B\beta(b)+s_A(l_B(a)x)\beta(b)-\beta(a)\diamond_B(s_A(x)b)-\varepsilon(x,b)r_A(s_B(b)x)\beta(a) \\
\quad =\varepsilon(a,x)\big(l_A(x)b\cdot_B\beta(b)-s_A(r_B(a)x)\beta(b)-l_A(\alpha(x))(a\cdot_B b)\big),
\end{array} \\*[0,2cm]
&
\begin{array}{l}
l_A(x)a\cdot_B\beta(b)+s_A(r_B(a)x)\beta(b)-l_A(\alpha(x))(a\cdot_B b)\\
\hspace{3cm} =\varepsilon(x,a)\big(r_A(x)a\cdot_B\beta(b)+s_A(l_B(a)x)\beta(b)\big.\\
\hspace{4cm} \big.-\beta(a)\diamond_B(s_A(x)b)-\varepsilon(x,b)r_A(s_B(b)x)\beta(a)\big),
\end{array}
\\*[0,2cm]
&\ r_B(\beta(a))(x\cdot_A y)=\varepsilon(y,a)\big(r_B(a)x\cdot_A\alpha(y)+s_B(l_A(x)a)\alpha(y),\\*[0,2cm]
&\ l_B(s_A(x)a)\alpha(y)+\varepsilon(x,a)s(a)x\diamond_A\alpha(y)=\varepsilon(a,y)\big(\varepsilon(x+y,a)s(\beta(a))(x\diamond_A y)\big),\\*[0,2cm]
&
\begin{array}{l}
\varepsilon(x,a)l_B(s_A(x)a)\alpha(y)+s_B(a)x\diamond_A\alpha(y) \\
\hspace{4cm} =\varepsilon(x,y)\big(l_B(a)y\cdot_A\alpha(x)+s_B(r_A(y)a)\alpha(x)\big),
\end{array}
\\*[0,2cm]
&\begin{array}{l}
\varepsilon(x+y,a)s(\beta(a))(x\diamond_A y)-\varepsilon(y,a)\alpha(x)\diamond_A s_B(a)y-r_B(s_A(y)a)\alpha(x)\\
\quad =\varepsilon(x,y)\big(\varepsilon(x+y,a)s(\beta(a))(y\diamond_A x)-\varepsilon(x,a)\alpha(y)\diamond_A(s_B(a)x)\big.\\
\hspace{8cm} \big.-r_B(s_A(x)a)\alpha(y)\big),
\end{array} \\
&\begin{array}{l}
r_B(a)y\cdot_A\alpha(y)+s_B(l_A(x)a)\alpha(y)-\alpha(x)\diamond_A(s_B(a)y)-\varepsilon(a,y)r_B(s_A(y)a)\alpha(x)\nonumber\\
\quad
=\varepsilon(x,a)\big(l_B(a)y\cdot_A\alpha(y)-s_B(r_A(x)a)\alpha(y)-l_B(\beta(a))(x\cdot_A y)\big),
\end{array}
\\*[0,2cm]
&\begin{array}{l}
l_B(a)x\cdot_A\alpha(y)+s_B(r_A(x)a)\alpha(y)-l_B(\beta(a))(x\cdot_A y)\\
\hspace{3cm} =\varepsilon(a,x)\big(r_B(a)x\cdot_A\alpha(y)+s_B(l_A(x)a)\alpha(y)\big.\\
\hspace{4cm} \big.-\alpha(x)\diamond_A(s_B(a)y)-\varepsilon(a,y)r_B(s_A(y)a)\alpha(x)\big).
\end{array}
\end{align}
Then, $(A,B,s_A,l_A,r_A,\beta,s_B,l_B,r_B,\alpha)$ is called a matched pair of the Hom-Novikov-Poisson color Hom-algebras. In this case, on the direct sum
$A\oplus B$ of the underlying linear spaces of $\mathcal{A}$ and $\mathcal{B}$, there is a Hom-Novikov-Poisson color Hom-algebra structure which is given for any $ x+a\in A_{\gamma_1}\oplus B_{\gamma_1}, y+b\in A_{\gamma_2}\oplus B_{\gamma_2}$ by
\begin{eqnarray}
&&\begin{array}{lcl}
(x + a) \cdot(y + b)&=&x \cdot_A y + (s_A(x)b + \varepsilon(a,y)s_A(y)a)\\
&&+a \cdot_B b + (s_B(a)y + \varepsilon(x,b)s_B(b)x),
\end{array} \\
&&\begin{array}{lcl}
(x + a) \diamond(y + b)&=&
x \diamond_A y + (l_A(x)b + r_A(y)a)\\
&& +a \diamond_B b + (l_B(a)y + r_B(b)x).
\end{array}
\end{eqnarray}
\end{thm}
\begin{proof}
By Proposition \ref{prop:matched.pairs.comm.Hom-ass.coloralgeb} and Proposition \ref{thm:matchedpairs}, we deduce that $(A\oplus B,\cdot,\varepsilon,\alpha+\beta)$
is a $\varepsilon$-commutative Hom-associative color Hom-algebra and $(A\oplus B,\diamond,\alpha+\beta)$ is a Hom-Novikov color Hom-algebra.
Now, the
rest, it is easy (in a similar way as for Proposition \ref{prop:matched.pairs.comm.Hom-ass.coloralgeb}) to verify the compatibility conditions are satisfied.
\end{proof}
\begin{defn}
A transposed Hom-Poisson color Hom-algebra is defined as a quintuple
$(A,\cdot,[\cdot,\cdot],\varepsilon,\alpha)$, where $(A,\cdot,\varepsilon,\alpha)$ is a
$\varepsilon$-commutative Hom-associative color Hom-algebra and
$(A,[\cdot,\cdot],\varepsilon,\alpha)$ is a Hom-Lie color Hom-algebra, satisfying the
transposed Hom-$\varepsilon$-Leibniz identity for $x,y,z\in \mathcal{H}(A)$,
\begin{equation}\label{leibniz}
2\alpha(z)\cdot[x,y]=[z\cdot x,\alpha(y)]+\varepsilon(z,x)[\alpha(x),z\cdot y].
\end{equation}
\end{defn}
\begin{prop}
Let $(A,\cdot, [\cdot,\cdot],\varepsilon,\alpha)$ be a multiplicative
transposed Hom-Posson color Hom-algebra. Then the following identities hold for all $h,x,y,z\in \mathcal{H}(A)$,
\begin{eqnarray}
&&\circlearrowleft_{x,y,z}\varepsilon(z,x)\alpha(x)\cdot[y,z]=0,\label{eq:gi1}\\
&&\circlearrowleft_{x,y,z}\varepsilon(z,x)[\alpha(h)\cdot[x,y],\alpha^2(z)]=0,\label{eq:gi2}\\
&&\circlearrowleft_{x,y,z}\varepsilon(z,x)[\alpha(h)\cdot \alpha(x),[\alpha(y),\alpha(z)]]=0,\label{eq:gi3}\\
&&\circlearrowleft_{x,y,z}\varepsilon(z,x)[\alpha(h),\alpha(x)]\cdot [\alpha(y),\alpha(z)]=0.\label{eq:gi4}
\end{eqnarray}
\end{prop}
\begin{proof}
\underline{\bf Proof of \eqref{eq:gi1}:}
Let $x, y, z\in\mathcal{H}(A)$. By the
transposed Hom-$\varepsilon$-Leibniz identity,
\begin{align*}
&\circlearrowleft_{x,y,z}\varepsilon(z,x)\Big(2\alpha(x)\cdot[y,z]\Big)=\circlearrowleft_{x,y,z}\varepsilon(z,x)\Big([x\cdot y,\alpha(z)]+\varepsilon(x,y)[\alpha(y),x\cdot z]\Big)\\&\quad=\circlearrowleft_{x,y,z}\varepsilon(z,x)\Big([x\cdot y,\alpha(z)]+\varepsilon(x+y,z)[z\cdot x,\alpha(y)]\Big)\\
&\quad=\circlearrowleft_{x,y,z}\Big(\varepsilon(z,x)[x\cdot y,\alpha(z)]-\varepsilon(y,z)[z\cdot x,\alpha(y)]\Big)=0.
\end{align*}
Then we have \eqref{eq:gi1}.

\underline{\bf Proof of \eqref{eq:gi2}:}
Let $x, y, z, h\in\mathcal{H}(A)$.
 First, by \eqref{leibniz}, we have
\begin{align*}\circlearrowleft_{x,y,z}\varepsilon(z,x)\Big(2\alpha^2(h)[[x,y],\alpha(z)]\Big)&=\circlearrowleft_{x,y,z}\varepsilon(z,x)\Big([\alpha(h),[x,y],\alpha^2(z)]\\&+\varepsilon(h,x+y)[\alpha([x,y]),\alpha(h\cdot z)]\Big).\end{align*}
Applying the Hom-Jacobi identity of the above equality, we obtain
\begin{equation}\label{first equation1}
   \begin{array}{l}
   \circlearrowleft_{x,y,z}\varepsilon(z,x)\Big([\alpha(h),[x,y],\alpha^2(z)]\Big) \\
   \quad \quad +\circlearrowleft_{x,y,z}\varepsilon(z,x)\Big(\varepsilon(h,x+y)[\alpha([x,y]),\alpha(h\cdot z)]\Big)
   =0.
   \end{array}
\end{equation}
Next, by the Hom-Jacobi identity, we have
\begin{align*}
\varepsilon(h,x+y)[[\alpha(x),\alpha(y)],\alpha(h\cdot z)]&+\varepsilon(x+y,z)[[h\cdot z,\alpha(x)],\alpha^2(y)]\\&+\varepsilon(h,y)\varepsilon(x,y+z)[[\alpha(y),h\cdot z],\alpha^2(x)]=0,
\end{align*}
and by \eqref{leibniz} we have
\begin{align*}\varepsilon(h,y)\varepsilon(x,y+z)[[\alpha(y),h\cdot z],\alpha^2(x)]=&2\varepsilon(x,y+z)[\alpha(h)\cdot [y,z],\alpha^2(x)]\\&-\varepsilon(x,y+z)[[h\cdot y,\alpha(z)],\alpha^2(x)]=0.\end{align*}
Thus we obtain
\begin{align*}
&\varepsilon(h,x+y)[[\alpha(x),\alpha(y)],\alpha(h\cdot z)]+\varepsilon(x+y,z)[[h\cdot z,\alpha(x)],\alpha^2(y)]\\
&+\varepsilon(x,y+z)\big(2[\alpha(h)\cdot[y,z],\alpha^2(x)]-\varepsilon(x,y+z)[[h\cdot y,\alpha(z)],\alpha^2(x)]\big)=0.
\end{align*}
Similarly, we have
\begin{align*}
& \circlearrowleft_{x,y,z}\varepsilon(z,x)\Big(\varepsilon(h,x+y)[[\alpha(x),\alpha(y)],\alpha(h\cdot z)]+ \varepsilon(x+y,z)[[h\cdot z,\alpha(x)],\alpha^2(y)]\\
&+\varepsilon(x,y+z)\big(2[\alpha(h)\cdot[y,z],\alpha^2(x)]-\varepsilon(x,y+z)[[h\cdot y,\alpha(z)],\alpha^2(x)]\big)\Big)=0.
\end{align*}
Taking the above sum, we obtain
\begin{equation}\label{second equation2}
\begin{array}{l}
    \circlearrowleft_{x,y,z}\varepsilon(z,x)\Big(\varepsilon(h,x+y)[[\alpha(x),\alpha(y)],\alpha(h\cdot z)]\Big) \\
    \quad \quad +\circlearrowleft_{x,y,z}\varepsilon(z,x)\Big(2[\alpha(h)\cdot[x,y],\alpha^2(x)]\Big)=0.
\end{array}
\end{equation}
Finally, taking the difference between the two equations \eqref{first equation1} and \eqref{second equation2} we obtain \eqref{eq:gi2}.

\underline{\bf Proof of \eqref{eq:gi3}:}
Taking the difference between the two equations \eqref{eq:gi2} and \eqref{first equation1} we obtain \eqref{eq:gi3}

\underline{\bf Proof of \eqref{eq:gi4}:}
Let $x, y, z, h\in\mathcal{H}(A)$.
By \eqref{leibniz} we have
\begin{align*}&\circlearrowleft_{x,y,z}\varepsilon(z,x)\Big(2\varepsilon(h,x+y)[\alpha(x),\alpha(y)]\cdot[\alpha(h),\alpha(z)]\Big)\\
&\quad =\circlearrowleft_{x,y,z}\varepsilon(z,x)\Big([\alpha(h)\cdot[x,y],\alpha^2(z)]+\varepsilon(x+y,z)[\alpha^2(h),\alpha(z)\cdot[x,y].]\Big).\end{align*}
Applying equations \eqref{eq:gi1} and \eqref{eq:gi2} to above equality, we obtain \eqref{eq:gi4}.
\end{proof}
\begin{prop}\label{prop:homnovikovpoisson to transposedpoissonalgebs}
Let $(A,\cdot,\diamond,\varepsilon,\alpha)$ be a Hom-Novikov-Poisson color Hom-algebra. Define
\begin{equation}[x,y]=x\diamond y-\varepsilon(x,y)y\diamond x,\;\; \forall x,y\in \mathcal{H}(A).\label{eq:com}\end{equation}
Then
$(A,\cdot,[\cdot,\cdot],\varepsilon,\alpha)$ is a Hom-transposed-Poisson color Hom-algebra.
\end{prop}
\begin{proof}
By definition, we have $(A,\cdot,\varepsilon,\alpha)$ is a $\varepsilon$-commutative Hom-associative color Hom-algebra and by Proposition \ref{Hom-Novik to Hom-Lie}, $(A,[\cdot,\cdot],\varepsilon,\alpha)$ is a Hom-Lie color Hom-algebra.
Now, we show that the $\varepsilon$-transposed Hom-Leibniz identity is satisfied. For any $x,y,z \in \mathcal{H}(A)$ we have
\begin{align*}
&\varepsilon(z,x)[x\cdot z,\alpha(y)]+\varepsilon(z,x+y)[\alpha(x),y\cdot z]-2\alpha(z)\cdot[x,y]\\
    &=\varepsilon(z,x)\Big((x\cdot z)\diamond\alpha(y)-\varepsilon(x+z,y)\alpha(y)\diamond(x\cdot z)\Big)\\
  & +\varepsilon(z,x+y)\Big(\alpha(x)\diamond(y\cdot z)-\varepsilon(x,y+z)(y\cdot z)\diamond\alpha(x)\Big)-2\alpha(z)\cdot\Big(x\cdot y-\varepsilon(x,y) y\cdot x\Big)\\
   & \text{\scriptsize (using Eqs. \eqref{HomNovPoisscoloralg1} and \eqref{HomNovPoisscoloralg2})}\\
   &=\varepsilon(z,x)(x\cdot z)\diamond\alpha(y)-\alpha(z)(x\diamond y)+\varepsilon(x,y)\alpha(z)\cdot (y\diamond x)\\
 &\quad-\varepsilon(x+z,y)(y\cdot z)\diamond\alpha(x)-\Big(\varepsilon(z,x+y) (x\diamond y)\cdot\alpha(z)\\
   &\quad-\varepsilon(x,y)(y\diamond x)\cdot\alpha(z)-\varepsilon(z,x+y)\alpha(x)\diamond(y\cdot z)+\varepsilon(x,y)\alpha(y)\diamond(x\cdot z)\Big)=0.
\end{align*}
Hence the conclusion holds.
\end{proof}
\begin{ex}
Let $A = A_{0} \oplus A_{1} =< e_1,e_2> \oplus <e_3,e_4>$
 be a $4$-dimensional superspace. The quintuple
$(A,\cdot,\diamond,\varepsilon,\alpha)$ is a Hom-Novikov-Poisson color Hom-algebra with
\begin{align*}
&\text{the bicharacter:} \quad \varepsilon(i,j)=(-1)^{ij},\\
&\text{the multiplications:} \quad
\begin{array}[t]{l}
e_2\cdot e_2=e_1, \quad e_2\cdot e_4=e_4\cdot e_2= e_3\\
e_2\diamond e_4=-e_3, \quad e_4\diamond e_2=e_3,\quad
e_4\diamond e_4=2e_1,
\end{array}
\\
&\text{the even linear map $\alpha: A \rightarrow A$ defined by}
\begin{array}[t]{ll}
\alpha(e_1)=2e_1, &\quad \alpha(e_2)=-e_2,\\
\alpha(e_3)=-e_4, &\quad \alpha(e_4)=e_3.
\end{array}
\end{align*}
Therefore, using the Proposition \ref{prop:homnovikovpoisson to transposedpoissonalgebs}, $(A,\cdot,[\cdot,\cdot],\varepsilon,\alpha)$ is a transposed Hom-Poisson color Hom-algebra with
$$[e_2,e_4]=[e_4,e_2]=-2e_3,\quad [e_4,e_4]=4e_1.$$
\end{ex}
\begin{defn}[\cite{bakayokoattan}]
A \emph{Hom-Poisson color Hom-algebra} is defined as a quintuple $(A,\cdot, [\cdot,\cdot],\varepsilon, \alpha)$ such that   $(A,\cdot, \varepsilon, \alpha)$ is a $\varepsilon$-commutative Hom-associative color Hom-algebra,
and $(A, [\cdot,\cdot], \varepsilon, \alpha)$ is a Hom-Lie color Hom-algebra, satisfying
for all $x, y, z$ in $\mathcal{H}(A)$, the Hom-$\varepsilon$-Leibniz identity,
\begin{equation}\label{CompatibiltyPoisson}
[\alpha (x) , y\cdot z]=\varepsilon(x,y)\alpha (y)\cdot
[x,z]+ \varepsilon(x+y,z)\alpha (z)\cdot [x,y].
\end{equation}
\end{defn}
Condition \eqref{CompatibiltyPoisson}, expressing the compatibility between the
multiplication and the Poisson bracket, can be reformulated equivalently as
\begin{equation}\label{CompatibiltyPoissonLeibform}
[x\cdot y,\alpha (z) ]=\varepsilon(y,z) [x,z]\cdot\alpha
(y)+\alpha (x)\cdot[y,z].
\end{equation}
\begin{defn}
Let $(A,\cdot,\diamond,\varepsilon,\alpha)$ be Hom-Novikov-Poisson color Hom-algebra. Then $A$ is called {\bf admissible} if $(A,\cdot,[\cdot,\cdot],\varepsilon,\alpha)$ is a Hom-Poisson color Hom-algebra, where
\begin{equation}
    [x,y]=x\diamond y-\varepsilon(x,y)y\diamond x.
\end{equation}
\end{defn}
\begin{lem}\label{lemmaHomNovikovalgeb}
Let $(A,\cdot,\diamond,\varepsilon,\alpha)$ be a Hom-Novikov-Poisson color Hom-algebra. Then for any $x,y,z\in\mathcal{H}(A)$,
\begin{equation}(x\cdot y)\diamond\alpha(z)=\alpha(x)\diamond(y\cdot z).\end{equation}
\end{lem}
\begin{proof}
For all $x,y,z\in \mathcal{H}(A)$, we have
\begin{align*}
& \text{\scriptsize ($(A,\cdot,\varepsilon,\alpha)$ is $\varepsilon$-commutative)}\\
&(x\cdot y)\circ\alpha(z)=\varepsilon(x,y)(y\cdot x)\circ\alpha(z)\\
&\text{\scriptsize (by \eqref{HomNovPoisscoloralg1})}\\
&=\varepsilon(x,y+z)(y\circ z)\cdot \alpha(x)\\
&\text{\scriptsize ($(A,\cdot,\varepsilon,\alpha)$ is $\varepsilon$-commutative)}\\   &=\alpha(x)\cdot(y\circ z).
\qedhere \end{align*}
\end{proof}
The following result gives a necessary and sufficient condition under which a Hom-Novikov-Poisson color
algebra is admissible.
\begin{thm}\label{thmadmissibleleftassociator}
Let $(A,\cdot,\diamond,\varepsilon,\alpha)$ be Hom-Novikov-Poisson color Hom-algebra.
Then $A$ is an admissible if and only
if
\begin{equation}\label{eqadmissibleleftassociator}
    as^{l}_A(x,y,z)=(x\cdot y)\diamond\alpha(z)-\alpha(x)\diamond(y\cdot z)=0.
\end{equation}
\end{thm}
\begin{proof}
By definition, $(A,\cdot,\varepsilon,\alpha)$ is a $\varepsilon$-commutative Hom-associative color Hom-algebra and by Proposition \ref{Hom-Novik to Hom-Lie},
$(A,[\cdot,\cdot],\varepsilon,\alpha)$ is a Hom-Lie color Hom-algebra. Therefore,
\begin{align*}
   &\text{\scriptsize (using \eqref{eqNovtoLie})}\\
    & [\alpha(x),y\cdot z]=\alpha(x)\diamond(y\cdot z)-\varepsilon(x+y,z)(y\cdot z)\diamond\alpha(x)\\
     &\text{\scriptsize (by \eqref{HomNovPoisscoloralg2} and Lemma \ref{lemmaHomNovikovalgeb})}\\
     &=(x\diamond y)\cdot\alpha(z)-\varepsilon(x,y)(y\diamond x)\cdot\alpha(z)+\varepsilon(x,y)\alpha(y)\diamond(x\cdot z) -\varepsilon(x,y+z)\alpha(y)\cdot(z\diamond x),\\
    & [x,y]\cdot\alpha(x)+\varepsilon(x,y)\alpha(y)[x,z]\\
    &\text{\scriptsize (using \eqref{eqNovtoLie})}\\
    &=(x\diamond y)\cdot\alpha(z)-\varepsilon(x,y)(y\diamond x)\cdot \alpha(z)+\varepsilon(x,y)\alpha(y)\cdot(x\diamond z) -\varepsilon(x,y+z)\alpha(y)\cdot(z\diamond x)\\
    &\text{\scriptsize (by Lemma \ref{lemmaHomNovikovalgeb})} \\
    &=(x\diamond y)\cdot\alpha(z)-\varepsilon(x,y)(y\diamond x)\cdot\alpha(z)+\varepsilon(x,y)(y\cdot x)\diamond\alpha(z)-\varepsilon(x,y+z)\alpha(y)\cdot(z\diamond x).
\end{align*}
Then $A$ satisfies the Hom-$\varepsilon$-Leibniz-identity if and only if $\alpha(y)\diamond(x\cdot z)=(y\cdot x)\diamond \alpha(z).$
\end{proof}
\begin{ex}
Let $A = A_{0} \oplus A_{1} =< e_1,e_2> \oplus <e_3,e_4>$
 be a $4$-dimensional super\-space. The quintuple
$(A,\cdot,\diamond,\varepsilon,\alpha)$ is a Hom-Novikov-Poisson color Hom-algebra with
\begin{align*}
&\text{the bicharacter:} \quad \varepsilon(i,j)=(-1)^{ij}, \\
&\text{the multiplications:} \quad
\begin{array}[t]{rclrcll}
e_2\cdot e_2&=&\lambda_1 e_1,& \quad e_2\cdot e_4&=&e_4\cdot e_2=\lambda_2 e_3, \quad\lambda_i\in\mathbb{K}\\
e_2\diamond e_2&=&\mu_1e_1,& \quad e_4\diamond e_2&=&\mu_2e_3,\\
\quad
e_4\diamond e_4&=&\mu_3 e_1,\quad &\mu_i\in\mathbb{K}\
\end{array} \\
&\text{the even linear map $\alpha: A \rightarrow A$ defined by}
\begin{array}[t]{l}
\alpha(e_1)=2e_1-e_2,\quad \alpha(e_2)=e_1,\\
\alpha(e_3)=-e_4,\quad \alpha(e_4)=e_3-e_4.
\end{array}
\end{align*}
Then by Theorem \ref{thmadmissibleleftassociator}, the Hom-Novikov-Poisson color Hom-algebra $(A,\cdot,\diamond,\varepsilon,\alpha)$ is admissible.
\end{ex}
\begin{thm}\label{twist-admissible Hom-Nov-Poiss}
Let $\mathcal{A}=(A, \cdot ,\diamond,\varepsilon)$ be an admissible Novikov-Poisson color Hom-algebra and
$\alpha :\mathcal{A}\rightarrow \mathcal{A}$ be an admissible Novikov-Poisson color Hom-algebras
morphism. Define $\cdot_{\alpha}, \ \diamond_\alpha:A \times A\rightarrow A$ for all $x, y\in \mathcal{H}(A)$, by
$x\cdot _{\alpha}y =\alpha (x\cdot y)$ and $x\diamond_{\alpha}y =\alpha(x\diamond y)$.
Then, $\mathcal{A}_\alpha=(A, \cdot _{\alpha},\diamond_\alpha,\varepsilon, \alpha)$ is an admissible Hom-Novikov-Poisson color Hom-algebra called the $\alpha$-twist or Yau twist of $(A, \cdot ,\diamond,\varepsilon)$.
\end{thm}
\begin{proof}
By Theorem \ref{morphHom-Nov-Poisscoloralgebra}, $\mathcal{A}_\alpha$ is a Hom-Novikov-Poisson color Hom-algebra. Moreover, the left Hom-associators in $\mathcal{A}$ and $\mathcal{A}_\alpha$ are related as
\begin{equation*}
as_{\mathcal{A}_\alpha}^{l}(x,y,z)=\alpha^2 as_\mathcal{A}^{l}(x,y,z)\quad for~ all~ x,y,z\in \mathcal{H}(A).\end{equation*}
Since $\mathcal{A}$ is left Hom-associative by Theorem \ref{thmadmissibleleftassociator}, it follows that so is $\mathcal{A}_\alpha$. Therefore, by Theorem \ref{thmadmissibleleftassociator}
again $\mathcal{A}_\alpha$ is admissible.
\end{proof}
\begin{cor}
If $\mathcal{A}=(A, \cdot,\diamond,\varepsilon,\alpha)$ is a multiplicative admissible Novikov-Poisson color
algebra, then for any $n\in\mathbb{N}^{\ast}$,
\begin{enumerate}
\item
The $n{\rm th}$ derived admissible Hom-Novikov-Poisson color Hom-algebra of type $1$ of $\mathcal{A}$ is
defined by
$$\mathcal{A}_{1}^{n}=(A,\cdot^{(n)}=\alpha^{n}\circ\cdot,\diamond^{(n)}=\alpha^{n}\circ\diamond,\varepsilon,\alpha^{n+1}).$$
\item
The $n{\rm th}$ derived admissible Hom-Novikov-Poisson color Hom-algebra of type $2$ of $A$ is
defined by
$$\mathcal{A}_{2}^{n}=(A,\cdot^{(2^n-1)}=\alpha^{2^n-1}\circ\cdot,\diamond^{(2^n-1)}=\alpha^{2^n-1}\circ\diamond,\varepsilon,\alpha^{2^n}).$$
\end{enumerate}
\end{cor}
\begin{ex} Let $A = A_{0} \oplus A_{1} =< e_1,e_2> \oplus <e_3,e_4>$
 be a $4$-dimensional superspace. There is a multiplicative admissible Hom-Novikov-Poisson color Hom-algebra $(A, \cdot,\diamond,\varepsilon,\alpha)$
with the bicharacter, $\varepsilon(i,j)=(-1)^{ij},$ and the multiplications tables for a basis $\{e_1, e_2, e_3,e_4\}$:
\begin{center}
\begin{tabular}{c|cccc}
$\cdot$ & $e_1$ & $e_2$ & $e_3$ & $e_4$ \\ \hline
$e_1$ & $0$ & $0$ & $0$ & $0$ \\
$e_2$ & $0$ & $e_1$ & $0$ & $4e_3$ \\
$e_3$ & $0$ & $0$ & $0$& $0$ \\
$e_4$ & $0$ & $4e_3$ & $0$& $0$ \\
\end{tabular}
\hspace{1 cm}
\begin{tabular}{c|cccc}
$\diamond$ & $e_1$ & $e_2$ & $e_3$ & $e_4$ \\ \hline
$e_1$ & $0$ & $0$ & $0$ & $0$ \\
$e_2$ & $0$ & $4e_3$ & $0$ & $4e_3$ \\
$e_3$ & $0$ & $0$ & $0$& $0$ \\
$e_4$ & $0$ & $0$ & $0$& $e_1$ \\
\end{tabular}
\end{center}
$$\begin{array}{ll}
\alpha(e_1)=4e_1, &\alpha(e_2)=-2e_2,\\
\alpha(e_3)=e_3,&\alpha(e_4)=-2e_4.
\end{array}$$
Then there are admissible Hom-Novikov-Poisson color Hom-algebras $\mathcal{A}_{1}^{n}$ and $\mathcal{A}_{2}^{n}$ with multiplications
tables respectively:
\begin{center}
\begin{tabular}{c|cccc}
$\cdot^{(n)}$ & $e_1$ & $e_2$ & $e_3$ & $e_4$ \\ \hline
$e_1$ & $0$ & $0$ & $0$ & $0$ \\
$e_2$ & $0$ & $2^{2n}e_1$ & $0$ & $4e_3$ \\
$e_3$ & $0$ & $0$ & $0$& $0$ \\
$e_4$ & $0$ & $4e_3$ & $0$& $0$ \\
\end{tabular}
\hspace{1 cm}
\begin{tabular}{c|cccc}
$\diamond^{(n)}$ & $e_1$ & $e_2$ & $e_3$ & $e_4$ \\ \hline
$e_1$ & $0$ & $0$ & $0$ & $0$ \\
$e_2$ & $0$ & $4e_3$ & $0$ & $4e_3$ \\
$e_3$ & $0$ & $0$ & $0$& $0$ \\
$e_4$ & $0$ & $0$ & $0$& $2^{2n}e_1$ \\
\end{tabular}
\end{center}
$$\begin{array}{ll}
    \alpha^{n+1}(e_1)=4^{n+1}e_1, &\alpha^{n+1}(e_2)=(-2)^{n+1}e_2,\\
    \alpha^{n+1}(e_3)=e_3,&\alpha^{n+1}(e_4)=(-2)^{n+1}e_4,
\end{array}$$
\vspace{0,1 cm}
\begin{center}
\begin{tabular}{c|cccc}
$\cdot^{(2^n-1)}$ & $e_1$ & $e_2$ & $e_3$ & $e_4$ \\ \hline
$e_1$ & $0$ & $0$ & $0$ & $0$ \\
$e_2$ & $0$ & $2^{2(2^n-1)}e_1$ & $0$ & $4e_3$ \\
$e_3$ & $0$ & $0$ & $0$& $0$ \\
$e_4$ & $0$ & $4e_3$ & $0$& $0$ \\
\end{tabular}
\hspace{1 cm}
\begin{tabular}{c|cccc}
$\diamond^{(2^n-1)}$ & $e_1$ & $e_2$ & $e_3$ & $e_4$ \\ \hline
$e_1$ & $0$ & $0$ & $0$ & $0$ \\
$e_2$ & $0$ & $4e_3$ & $0$ & $4e_3$ \\
$e_3$ & $0$ & $0$ & $0$& $0$ \\
$e_4$ & $0$ & $0$ & $0$& $2^{2(2^n-1)}e_1$ \\
\end{tabular}
\end{center}
$$\begin{array}{ll}
    \alpha^{2^n}(e_1)=4^{2^n}e_1, &\alpha^{2^n}(e_2)=2^{2^{n}}e_2,\\
    \alpha^{2^n}(e_3)=e_3,&\alpha^{2^n}(e_4)=2^{2^{n}}e_4.
\end{array}$$
\end{ex}
\subsection{Tensor products of admissible Hom-Novikov-Poisson color Hom-algebras}
Now, we show that the much larger class of admissible Hom-Novikov-Poisson color Hom-algebras is also closed under tensor
products.
\begin{thm}\label{tensor product}
Let $(A_1, \cdot_1 ,\diamond_1,\varepsilon,\alpha_1)$ and $(A_2, \cdot_2 ,\diamond_2,\varepsilon,\alpha_2)$ be admissible Hom-Novikov-Poisson color Hom-algebras and let $A = A_1\otimes A_2$.
Define the operations $\alpha: A \rightarrow A$ and $\cdot,\diamond: A\otimes A\rightarrow A$ by the following formulae for
all $x_i, y_i \in \mathcal{H}(A_i), i\in\{1;2\}$,
\begin{eqnarray*}
    &\alpha=\alpha_1\otimes\alpha_2,\\
   &(x_1\otimes x_2)\cdot(y_1\otimes y_2)=\varepsilon(x_2,y_1)(x_1\cdot_1 y_1)\otimes(x_2\cdot_2 y_2),\\
   &  (x_1\otimes x_2)\diamond(y_1\otimes y_2)=\varepsilon(x_2,y_1)\big((x_1\diamond_1 y_1)\otimes(x_2\cdot_2 y_2)+(x_1\cdot_1 y_1)\otimes(x_2\diamond_2 y_2)\big),
\end{eqnarray*}
Then $(A,\cdot,\diamond, \varepsilon,\alpha)$ is an admissible Hom-Novikov-Poisson color Hom-algebra.
\end{thm}
\begin{proof}
 Pick $x = x_1\otimes x_2, y = y_1\otimes y_2$
and $z = z_1 \otimes z_2$ homogeneous elements in $A$.

\underline{\bf Step 1:} We show that $(A,\cdot,\varepsilon,\alpha)$ is $\varepsilon$-commutative Hom-associative-color Hom-algebra:
\begin{align*}
x \cdot y &= (x_1\otimes x_2)\cdot(y_1\otimes y_2)\\
    &=\varepsilon(x_2,y_1)(x_1\cdot_1 y_1)\otimes(x_2\cdot_2 y_2)\\
     &=\varepsilon(x_2,y_1) \varepsilon(x_1,y_1)\varepsilon(x_2,y_2)(y_1\cdot_1 x_1)\otimes(y_2\cdot_2 x_2)\\
    &=  \varepsilon(x_1+x_2,y_1+y_2)\big(\varepsilon(y_2,x_1)(y_1\cdot_1 x_1)\otimes(y_2\cdot_2 x_2)\big)\\
     &=  \varepsilon(x_1+x_2,y_1+y_2)(y_1\otimes y_2)\cdot(x_1\otimes x_2)\\
      &=  \varepsilon(x,y)y\cdot x,\\
(x \cdot y)\cdot\alpha(z)&= ((x_1\otimes x_2)\cdot(y_1\otimes y_2))\cdot(\alpha_1\otimes \alpha_2)(z_1\otimes z_2)\\
    &=\Big(\varepsilon(x_2,y_1)\cdot(x_1\cdot_1 y_1)\otimes(x_2\otimes_2 y_2)\Big)\cdot(\alpha_1\otimes \alpha_2)(z_1\otimes z_2)\\
    &=\varepsilon(x_2+y_2,z_1)\varepsilon(x_2,y_1)\Big((x_1\cdot_1 y_1)\cdot_1\alpha_1(z_1)\otimes(x_2\cdot_2 y_2)\cdot_2\alpha_2(z_2)\Big)\\
    &=\varepsilon(x_2,y_1+z_1)\varepsilon(y_2,z_1)\Big(\alpha_1(x_1)\cdot_1(y_1\cdot_1 z_1)\otimes\alpha_2(x_2)\cdot_2( y_2\cdot_2 z_2)\Big)\\
    &=(\alpha_1\otimes\alpha_2)(x_1\otimes x_2)\cdot ((y_1\otimes y_2)\cdot(z_1\otimes z_2))\\
      &=\alpha(x)\cdot(y\cdot z).
\end{align*}
Hence, $(A_1\otimes A_2,\cdot,\varepsilon,\alpha)$ is a $\varepsilon$-commutative Hom-associative color Hom-algebra.

\underline{\bf Step 2:} We show that $(A,\diamond,\varepsilon,\alpha)$ is Hom-Novikov-color Hom-algebra.
\begin{align*}
    &(x\diamond y)\diamond\alpha(z)-\alpha(x)\diamond(y\diamond z)-\varepsilon(x,y)\big((y\diamond x)\diamond\alpha(z)-\alpha(y)\diamond(x\diamond z))\\
    &=\big((x_1\otimes x_2)\diamond(y_1\otimes y_2)\big)\diamond\alpha(z_1\otimes z_2)-(\alpha_1\otimes \alpha_2)(x_1\otimes x_2)\diamond\big((y_1\otimes y_2)\diamond(z_1\otimes z_2)\big)\\
   &\quad-\varepsilon(x_1+x_2,y_1+y_2)\Big(\big((y_1\otimes y_2)\diamond(x_1\otimes x_2)\big)\diamond(\alpha_1\otimes \alpha_2)(z_1\otimes z_2)\\
   &\quad-(\alpha_1\otimes \alpha_2)(y_1\otimes y_2)\diamond\big((x_1\otimes x_2)\diamond(z_1\otimes z_2)\big)\Big)\\
    &=\varepsilon(x_2,y_1)\big((x_1\diamond_1 y_1)\otimes(x_2\cdot_2 y_2)+(x_1\cdot_1 y_1)\otimes(x_2\diamond_2 y_2)\big)\diamond(\alpha_1\otimes\alpha_2)(z_1\otimes z_2)\\
    &\quad-\varepsilon(y_2,z_1)(\alpha_1\otimes \alpha_2)(x_1\otimes x_2)\diamond\big((y_1\diamond_1 z_1)\otimes(y_2\cdot_2 z_2)+(y_1\cdot_1 z_1)\otimes(y_2\diamond_2 z_2)\big)\\
   &\quad-\varepsilon(x_1+x_2,y_1+y_2)\Big(\varepsilon(y_2,x_1)\big((y_1\diamond_1 x_1)\otimes(y_2\cdot_2 x_2)+(y_1\cdot_1 x_1)\otimes(y_2\diamond_2 x_2)\big)\diamond\\&\quad\quad\quad(\alpha_1\otimes\alpha_2)(z_1\otimes z_2)-\varepsilon(x_2,z_1)(\alpha_1\otimes \alpha_2)(y_1\otimes y_2)\diamond\big((x_1\diamond_1 z_1)\otimes(x_2\cdot_2 z_2)\\
   &\quad +(x_1\cdot_1 z_1)\otimes(x_2\diamond_2 z_2)\big)\Big)\\
    &=\varepsilon(x_1+x_2,y_1+y_2)\times\\
    &\Big[\underbrace{(x_1\diamond_1 y_1)\diamond_1\alpha_1(z_1)\otimes(x_2\cdot_2 y_2)\cdot_2\alpha_2(z_2)}_{A_1}\underbrace{ +
   (x_1\diamond_1 y_1)\cdot_1\alpha_1(z_1)\otimes(x_2\cdot_2 y_2)\diamond_2\alpha_2(z_2)}_{A_2} \Big.\\
   &\underbrace{ +
   (x_1\cdot_1 y_1)\diamond_1\alpha_1(z_1)\otimes(x_2\diamond_2 y_2)\cdot_2\alpha_2(z_2)}_{A_3}\underbrace{ +
   (x_1\cdot_1 y_1)\cdot_1\alpha_1(z_1)\otimes(x_2\diamond_2 y_2)\diamond_2\alpha_2(z_2)}_{A_4}\\
   &\underbrace{ -\alpha_1(x_1)\diamond_1(y_1\diamond_1 z_1)\otimes\alpha_2(x_2)\cdot_2(y_2\cdot_2 z_2)}_{A_5}
   \underbrace{- \alpha_1(x_1)\cdot_1(y_1\diamond_1 z_1)\otimes\alpha_2(x_2)\diamond_2(y_2\cdot_2 z_2)}_{A_6}\\
   &
  \underbrace{  -\alpha_1(x_1)\diamond_1(y_1\cdot_1 z_1)\otimes\alpha_2(x_2)\cdot_2(y_2\diamond_2 z_2)}_{A_7}
   \underbrace{- \alpha_1(x_1)\cdot_1(y_1\cdot_1 z_1)\otimes\alpha_2(x_2)\diamond_2(y_2\diamond_2 z_2)}_{A_8}\\
    &\quad -\varepsilon(x_1,y_1)\varepsilon(x_2,y_2)\times\\
    &\Big(
   \underbrace{  (y_1\diamond_1 x_1)\diamond_1\alpha_1(z_1)\otimes(x_2\cdot_2 x_2)\cdot_2\alpha_2(z_2)}_{B_1}\underbrace{+
   (x_1\diamond_1 x_1)\cdot_1\alpha_1(z_1)\otimes(x_2\cdot_2 x_2)\diamond_2\alpha_2(z_2)}_{B_2}\\
   &\underbrace{ +
   (x_1\cdot_1 x_1)\diamond_1\alpha_1(z_1)\otimes(x_2\diamond_2 x_2)\cdot_2\alpha_2(z_2)}_{B_3}\underbrace{+
   (x_1\cdot_1 x_1)\cdot_1\alpha_1(z_1)\otimes(x_2\diamond_2 x_2)\diamond_2\alpha_2(z_2)}_{B_4}\\
   &\underbrace{ -\alpha_1(x_1)\diamond_1(x_1\diamond_1 z_1)\otimes\alpha_2(x_2)\cdot_2(x_2\cdot_2 z_2)}_{B_5}
   \underbrace{- \alpha_1(x_1)\cdot_1(x_1\diamond_1 z_1)\otimes\alpha_2(x_2)\diamond_2(x_2\cdot_2 z_2)}_{B_6}\\
   &
   \underbrace{ -\alpha_1(x_1)\diamond_1(x_1\cdot_1 z_1)\otimes\alpha_2(x_2)\cdot_2(x_2\diamond_2 z_2)}_{B_7} \\
   &\hspace{4cm}
  \underbrace{ \Big.\Big.-\alpha_1(x_1)\cdot_1(x_1\cdot_1 z_1)\otimes\beta(x_2)\diamond_2(x_2\diamond_2 z_2)\Big)\Big]}_{B_8}.
\end{align*}
Furthermore, we have
\begin{align*}
&(A_1+A_5)-\varepsilon(x_1,y_1)\varepsilon(x_2,y_2)(B_1+B_5)\\
 & \text{\scriptsize (by \eqref{Homass:homassociator}  and \eqref{commutative:homcoloralgeb})}\\
 & \quad=\Big[\big((x_1\diamond_1 y_1)\diamond_1\alpha_1(z_1)-\alpha_1(x_1)\diamond_1(y_1\diamond_1 z_1)\big)\\
 & \quad\quad-\varepsilon(x_1,y_1)\big((y_1\diamond_1 x_1)\diamond_1\alpha_1(z_1)-\alpha_1(y_1)\diamond_1(x_1\diamond_1 z_1)\big)\Big]\otimes (x_2\cdot_2 y_2)\cdot_2\alpha_2(z_2)\\
&   \text{\scriptsize (by \eqref{eqhomNovikov1})} \quad=0,\\
&(A_4+A_8)-\varepsilon(x_1,y_1)\varepsilon(x_2,y_2)(B_4+B_8)\\
&  \text{ \scriptsize (by \eqref{Homass:homassociator}  and \eqref{commutative:homcoloralgeb})} \\
&  \quad=(x_1\cdot_1 y_1)\cdot_1\alpha_1(z_1)\otimes\Big[\big((x_2\diamond_2 y_2)\diamond_2\alpha_2(z_2)\\
&   \quad\quad-\alpha_2(x_2)\diamond_2(y_2\diamond_2 z_2)\big)-\varepsilon(x_2,y_2)\big((y_2\diamond_2 x_2)\diamond_2\alpha_2(z_2)-\alpha_2(y_2)\diamond_2(x_2\diamond_2 z_2)\big)\Big]\\
&   \text{\scriptsize (by \eqref{eqhomNovikov1})} \quad=0,\\
&(A_2+A_7)-\varepsilon(x_1,y_1)\varepsilon(x_2,y_2)(B_2+B_7)\\
&  \text{\scriptsize (by \eqref{Homass:homassociator}, \eqref{HomNovPoisscoloralg1} and \eqref{commutative:homcoloralgeb})}\\ &\quad=\Big[\big((x_1\diamond_1 y_1)\cdot_1\alpha_1(z_1)-\alpha_1(x_1)\diamond_1(y_1\cdot_1 z_1)\big)\\
 &  \quad\quad-\varepsilon(x_1,y_1)\big((y_1\diamond_1 x_1)\cdot_1\alpha_1(z_1)-\alpha_1(y_1)\diamond_1(x_1\cdot_1 z_1)\big)\Big]\otimes(x_2\cdot_2 y_2)\diamond_2\alpha_2(z_2)\\
  & \text{\scriptsize (by \eqref{HomNovPoisscoloralg2})} \quad=0,\\
&(A_3+A_6)-\varepsilon(x_1,y_1)\varepsilon(x_2,y_2)(B_3+B_6)\\
  &\text{\scriptsize (by \eqref{Homass:homassociator}, \eqref{HomNovPoisscoloralg1} and \eqref{commutative:homcoloralgeb})} \\  &\quad=(x_1\cdot_1 y_1)\diamond_1\alpha_1(z_1)\otimes\Big[\big((x_2\diamond_2 y_2)\cdot_2\alpha_2(z_2)\\
  & \quad\quad-\alpha_2(x_2)\diamond_2(y_2\diamond_2 z_2)\big)-\varepsilon(x_2,y_2)\big((y_2\diamond_2 x_2)\cdot_2\alpha_2(z_2)-\alpha_2(y_2)\diamond_2(x_2\cdot_2 z_2)\big)\Big]\\
   &\text{\scriptsize (by \eqref{HomNovPoisscoloralg2})} \quad=0.
\end{align*}
Then, we obtain
\begin{align*}
&(x\diamond y)\diamond\alpha(z)-\alpha(x)\diamond(y\diamond z)-\varepsilon(x,y)\big((y\diamond x)\diamond\alpha(z)-\alpha(y)\diamond(x\diamond z)\big)=0.\\
&(x\diamond y)\diamond\alpha(z)-\varepsilon(y,z)\big((x\diamond z)\diamond\alpha(y)\\
    &\quad=\big((x_1\otimes x_2)\diamond(y_1\otimes y_2)\big)\diamond\alpha(z_1\otimes z_2)\\
  &\quad\quad-\varepsilon(y_1+y_2,z_1+z_2)\Big(\big((x_1\otimes x_2)\diamond(z_1\otimes z_2)\big)\diamond(\alpha_1\otimes \alpha_2)(y_1\otimes y_2)\Big)\\
  &\quad=\varepsilon(x_2,y_1)\Big((x_1\diamond_1 y_1)\otimes(x_2\cdot_2 y_2)+(x_1\cdot_1 y_1)\otimes(x_2\diamond_2 y_2)\Big)\diamond(\alpha_1\otimes\alpha_2)(z_1\otimes z_2)\\
   &\quad\quad -\varepsilon(y_1+y_2,z_1+z_2)\Big(\varepsilon(x_2,z_1)\big((x_1\diamond_1 z_1)\otimes(x_2\cdot_2 z_2)+(x_1\cdot_1 z_1)\otimes(x_2\diamond_2 z_2)\big)\diamond\\&\quad\quad\quad(\alpha_1\otimes\alpha_2)(y_1\otimes y_2)\Big)\\
    &\quad=\varepsilon(x_2,y_1)\varepsilon(x_2+y_2,z_1)\times\\
    &\Big[\underbrace{(x_1\diamond_1 y_1)\diamond_1\alpha_1(z_1)\otimes(x_2\cdot_2 y_2)\cdot_2\alpha_2(z_2)}_{C_1}\underbrace{ +
   (x_1\diamond_1 y_1)\cdot_1\alpha_1(z_1)\otimes(x_2\cdot_2 y_2)\diamond_2\alpha_2(z_2)}_{C_2}\\
   &\underbrace{ +
   (x_1\cdot_1 y_1)\diamond_1\alpha_1(z_1)\otimes(x_2\diamond_2 y_2)\cdot_2\alpha_2(z_2)}_{C_3}\underbrace{ +
   (x_1\cdot_1 y_1)\cdot_1\alpha_1(z_1)\otimes(x_2\diamond_2 y_2)\diamond_2\alpha_2(z_2)}_{C_4}\Big]\\
    &\quad\quad-\varepsilon(y_1+y_2,z_1+z_2)\varepsilon(x_2,z_1)\varepsilon(x_2+z_2,y_1)\times\\
    &\Big[\underbrace{(x_1\diamond_1 z_1)\diamond_1\alpha_1(y_1)\otimes(x_2\cdot_2 z_2)\cdot_2\alpha_2(y_2)}_{D_1}\underbrace{ +
   (x_1\diamond_1 z_1)\cdot_1\alpha_1(y_1)\otimes(x_2\cdot_2 z_2)\diamond_2\alpha_2(y_2)}_{D_2}\\
   &\underbrace{ +
   (x_1\cdot_1 z_1)\diamond_1\alpha_1(y_1)\otimes(x_2\diamond_2 z_2)\cdot_2\alpha_2(y_2)}_{D_3}\underbrace{ +
   (x_1\cdot_1 z_1)\cdot_1\alpha_1(y_1)\otimes(x_2\diamond_2 z_2)\diamond_2\alpha_2(y_2)}_{D_4}\Big].
   \end{align*}
   Furthermore, we have
   \begin{align*}
   & \varepsilon(x_2,y_1)\varepsilon(x_2+y_2,z_1)C_1-\varepsilon(y_1+y_2,z_1+z_2)\varepsilon(x_2,z_1)\varepsilon(x_2+z_2,y_1)D_1\\
  & \text{\scriptsize (by \eqref{commutative:homcoloralgeb})} \\  &\quad=\varepsilon(x_2,y_1+y_2)\varepsilon(x_2+y_2,z_1)\times\\
  &\quad\quad\big((x_1\diamond_1 y_1)\diamond_1\alpha_1(z_1)\big)\otimes(y_2\cdot_2 x_2)\cdot_2\alpha_2(z_2)\\
  &\quad\quad -\varepsilon(y_1+y_2,z_1+z_2)\varepsilon(x_2+z_2,y_1+y_2)\varepsilon(x_2,z_1)\times\\
  &\quad\quad \big((x_1\diamond_1 z_1)\diamond_1\alpha_1(y_1)\big)\otimes\alpha_2(y_2)\cdot_2(x_2\cdot_2 z_2)\\
  &\text{\scriptsize (by \eqref{Homass:homassociator})}\\
  &\quad=\varepsilon(x_2,y_1+y_2)\varepsilon(x_2+y_2,z_1)\times\\
  &\quad\quad \Big[(x_1\diamond_1 y_1)\diamond_1\alpha_1(z_1)-\varepsilon(y_1,z_1)(x_1\diamond z_1)\diamond_1\alpha_1(y_1)\Big]\otimes\alpha_2(y_2)\cdot_2(x_2\cdot_2 z_2)\\
  &\text{\scriptsize (by \eqref{eqhomNovikov2})} \quad =0.\\
      & \varepsilon(x_2,y_1)\varepsilon(x_2+y_2,z_1)C_4-\varepsilon(y_1+y_2,z_1+z_2)\varepsilon(x_2,z_1)\varepsilon(x_2+z_2,y_1)D_4\\
  &\text{\scriptsize (by \eqref{commutative:homcoloralgeb})}\\
  &\quad=\varepsilon(x_1+x_2,y_1)\varepsilon(x_2+y_2,z_1)\times\\
  &\quad\quad (y_1\cdot_1 x_1)\cdot_1\alpha_1(z_1)\otimes\big((x_2\diamond_2 y_1)\diamond_1\alpha_1(z_2)\big)\\
  &\quad\quad -\varepsilon(y_1+y_2,z_1+z_2)\varepsilon(x_2+z_2,y_1)\varepsilon(x_1+y_1,z_1)\times\\
  &\quad\quad \alpha_1(y_1)\cdot_1(x_1\cdot_2 1_2)\otimes\big((x_2\diamond_2 z_2)\diamond_2\alpha_2(y_2)\big)\\
  &\text{\scriptsize (by \eqref{Homass:homassociator})} \\
  &\quad=\varepsilon(x_1+x_2,y_1)\varepsilon(x_2+y_2,z_1)\times\\
  &\quad\quad (y_1\cdot_1 x_1)\cdot_1\alpha_1(z_1)\Big[(x_2\diamond_2 y_2)\diamond_2\alpha_2(z_2)-\varepsilon(y_2,z_2)(x_2\diamond_2 z_2)\diamond_2\alpha_2(y_2)\Big]\\
  & \text{\scriptsize (by \eqref{eqhomNovikov2})} \quad= 0.\\
  & \varepsilon(x_2,y_1)\varepsilon(x_2+y_2,z_1)C_2-\varepsilon(y_1+y_2,z_1+z_2)\varepsilon(x_2,z_1)\varepsilon(x_2+z_2,y_1)D_3
  \\
  &\text{\scriptsize(by \eqref{HomNovPoisscoloralg1})} \\
  &\quad=\varepsilon(x_2,y_1)\varepsilon(x_2+y_2,z_1)\Big[x_1\diamond_1 y_1)\cdot_1\alpha_1(z_1)\otimes(x_2\cdot_2 y_2)\diamond_2\alpha_2(z_2)\\
  &\quad\quad-(x_1\diamond_1 y_1)\cdot_1\alpha_1(z_1)\otimes(x_2\cdot_2 y_2)\diamond_2\alpha_2(z_2)\Big]=0,\\
 & \varepsilon(x_2,y_1)\varepsilon(x_2+y_2,z_1)C_3-\varepsilon(y_1+y_2,z_1+z_2)\varepsilon(x_2,z_1)\varepsilon(x_2+z_2,y_1)D_2\\
  & \text{\scriptsize (by \eqref{HomNovPoisscoloralg1})} \\
  &\quad=\varepsilon(x_2,y_1)\varepsilon(x_2+y_2,z_1)\Big[x_1\cdot_1 y_1)\diamond_1\alpha_1(z_1)\otimes(x_2\diamond_2 y_2)\cdot_2\alpha_2(z_2)\\
  &\quad\quad -(x_1\cdot_1 y_1)\diamond_1\alpha_1(z_1)\otimes(x_2\diamond_2 y_2)\cdot_2\alpha_2(z_2)\Big]=0.
\end{align*}
Then, we obtain
$$(x\diamond y)\diamond\alpha(z)-\varepsilon(y,z)\big((x\diamond z)\diamond\alpha(y)=0.$$
Hence, $(A_1\otimes A_2,\diamond,\varepsilon,\alpha)$ is a Hom-Novikov color Hom-algebra.

\underline{\bf Step 3:} We show that the compatibility conditions of Hom-Novikov-Poisson color Hom-algebras are satisfied
\begin{align*}
    &(x\cdot y)\diamond\alpha(z)-\varepsilon(y,z)(x\diamond z)\cdot\alpha(y)\\
   &\quad=\big((x_1\otimes x_2)\cdot(y_1\otimes y_2)\big) \diamond(\alpha_1\otimes\alpha_2)(z_1\otimes z_2)\\
   &\quad\quad-\varepsilon(y_1+y_2,z_1+z_2)\big((x_1\otimes x_2)\diamond(z_1\otimes z_2)\big)\cdot(\alpha_1\otimes\alpha_2)(y_1\otimes y_2)\\
   &\quad=\varepsilon(x_2,y_1)\big((x_1\cdot_1 y_1)\otimes(x_2\cdot_2 y_2)\big)\diamond(\alpha_1(z_1)\otimes\alpha_2(z_2))\\
 &\quad\quad  -\varepsilon(y_1+y_2,z_1+z_2)\varepsilon(x_2,z_1\Big((x_1\diamond_1 z_1)\otimes(x_2\diamond_2 z_2)\\
 &\quad\quad+(x_1\cdot_1 z_1)\otimes(x_2\cdot_2 z_2)\Big)\cdot(\alpha_1\otimes\alpha_2)(y_1\otimes y_2)\\
  &\quad\quad=\varepsilon(x_2,y_1)\varepsilon(x_2+y_2,z_1)\times\\&\quad\quad\Big(\underbrace{\big((x_1\cdot_1 y_1)\diamond_1\alpha_1(z_1)\big)\otimes\big((x_2\cdot_2 y_2)\cdot_2\alpha_2(z_2)\big)}_{E_1}\\
  &\quad\quad\underbrace{\big((x_1\cdot_1 y_1)\cdot_1\alpha_1(z_1)\big)\otimes\big((x_2\cdot_2 y_2)\diamond_2\alpha_2(z_2)\big)}_{E_2}\Big)\\
  &\quad\quad-\varepsilon(y_1+y_2,z_1+z_2)\varepsilon(x_2,z_1\varepsilon(x_2+z_2,y_1)\times\\
  &\quad\quad\Big(\underbrace{\big(x_1\diamond_1 z_1)\cdot_1\alpha_1(y_1)\big)\otimes\big((x_2\cdot_2 z_2)\cdot_2\alpha_2(y_2)\big)}_{F_1}\\
  &\quad\quad+\underbrace{\big((x_1\cdot_1 z_1)\cdot_1\alpha_1(y_1)\big)\otimes\big((x_2\diamond_2 z_2)\cdot_2\alpha_2(y_2)\big)}_{F_2}\Big).
\end{align*}
Furthermore, we have
\begin{align*}
&\varepsilon(x_2,y_1)\varepsilon(x_2+y_2,z_1)E_1-\varepsilon(y_1+y_2,z_1+z_2)\varepsilon(x_2,z_1)\varepsilon(x_2+z_2,y_1)F_1\\
 & \text{\scriptsize (by \eqref{commutative:homcoloralgeb})} \\
 &\quad =\varepsilon(x_2,y_1)\varepsilon(x_2+y_2,z_1)\varepsilon(x_2,y_2)\big((x_1\cdot_1 y_1)\diamond_1\alpha_1(z_1)\big)\otimes(y_2\cdot_2 x_2)\cdot_2\alpha_2(z_2)\\
&\quad-\varepsilon(y_1+y_2,z_1+z_2)\varepsilon(x_2,z_1)\varepsilon(x_2+z_2,y_1+y_2)\times\\
&\quad\quad\big((x_1\cdot_1 z_1)\cdot\alpha_1(y_1)\big)\otimes\big(\alpha_2(y_2)\cdot_2(x_2\cdot_2 z_2)\big)\\
&\text{\scriptsize (by \eqref{Homass:homassociator})}\\
 &\quad=\varepsilon(x_2,y_1+y_2)\varepsilon(x_2+y_2,z_1)\times\\
&\quad\quad \Big(
(x_1\cdot_1 y_1)\diamond_1\alpha_1(z_1)-\varepsilon(y_1,z_1)(x_1\diamond z_1)\cdot_1\alpha_1(y_1)\Big)\otimes\alpha_2(y_2)\cdot_2(x_2\cdot_2 z_2)\\
& \text{\scriptsize (by \eqref{HomNovPoisscoloralg1})} \quad=0,\\
&\varepsilon(x_2,y_1)\varepsilon(x_2+y_2,z_1)E_2-\varepsilon(y_1+y_2,z_1+z_2)\varepsilon(x_2,z_1)\varepsilon(x_2+z_2,y_1)F_2\\
& \text{\scriptsize (by \eqref{commutative:homcoloralgeb})}\\
&\quad =\varepsilon(x_2,y_1)\varepsilon(x_2+y_2,z_1)\varepsilon(x_1+y_1,z_1)\times\\
&\quad\quad \big(\alpha_1(z_1)\cdot(x_1\cdot y_1)\big)\otimes\big((x_2\cdot_2 y_2)\diamond_2\alpha_2(z_2)\big)\\&\quad-\varepsilon(y_1+y_2,z_1+z_2)\varepsilon(x_1+x_2,z_1)\varepsilon(x_2+z_2,y_1)\times\\
&\quad\quad\big((z_1\cdot_1 x_1)\cdot_1\alpha_1(y_1)\big)\otimes\big((x_2\diamond_2 z_2)\diamond_2\alpha_2(y_2)\big)\\
&\text{\scriptsize (by \eqref{Homass:homassociator})}\\
&\quad=\varepsilon(x_2,y_1)\varepsilon(x_1+x_2+y_1+y_2,z_1)\times\\
&\quad\alpha_1(z_1)\cdot_1(x_1\cdot_1 y_1)\otimes\Big((x_2\cdot_2 y_2)\diamond_2\alpha_2(z_2)
-\varepsilon(y_2,z_2)(x_2\diamond_2 z_2)\cdot_2\alpha_2(y_2)
\Big)\\
&\text{\scriptsize (by \eqref{HomNovPoisscoloralg1})} \quad=0.
\end{align*}
Then,
$(x\cdot y)\diamond\alpha(z)-\varepsilon(y,z)(x\diamond z)\cdot\alpha(y)=0.$
Similarly, $$(x\diamond y)\cdot \alpha(z) -\alpha(x)\diamond(y\cdot z) =\varepsilon(x,y)\big( (y\diamond x)\cdot\alpha(z)-\alpha(y)\diamond(x\cdot z)\big).$$
Hence, $(A_1\otimes A_2,\cdot,\diamond,\varepsilon,\alpha)$ is a Hom-Novikov-Poisson color Hom-algebra.

\underline{\bf Step 4:} We show that the Equation \eqref{eqadmissibleleftassociator} is satisfied:
\begin{align*}
    &\alpha(x)\cdot(y\diamond z)=\alpha(x_1\otimes x_2)\cdot\varepsilon(y_2,z_1)\Big((y_1\diamond_1 z_1)\otimes(y_2\cdot_2 z_2)+(y_1\cdot_1 z_1)\otimes(y_2\diamond_2 z_2)\Big)\\
    &\quad=\varepsilon(y_2,z_1)\varepsilon(x_2,y_1+z_1)\Big(\big(\alpha(x_1)\cdot_1(y_1\diamond_1 z_1)\big)\otimes\big(\alpha_2(x_2)\cdot_2(y_2\cdot_2 z_2)\big)\\
    &\quad\quad+\big(\alpha(x_1)\cdot_1(y_1\cdot_1 z_1)\big)\otimes\big(\alpha_2(x_2)\cdot_2(y_2\diamond z_2)\big)\Big),\\*[0,2cm]
   & \alpha(x)\diamond(y\cdot z)=(\alpha(x_1)\otimes \alpha_2(x_2))\diamond\big((y_1\otimes y_2)\cdot(z_1\otimes z_2)\big)\\
    &\quad=(\alpha_1(x_1)\otimes\alpha_2(x_2))\diamond\Big(\varepsilon(y_2,z_1)(y_1\diamond_1 z_1)\otimes(y_2\cdot_2 z_2)\Big)\\
        &\quad=\varepsilon(y_2,z_1)\varepsilon(x_2,y_1+z_1)\Big(\big(\alpha(x_1)\diamond_1(y_1\cdot_1 z_1)\big)\otimes\big(\alpha_2(x_2)\cdot_2(y_2\cdot_2 z_2)\big)\\
    &\quad\quad+\big(\alpha(x_1)\cdot_1(y_1\cdot_1 z_1)\big)\otimes\big(\alpha_2(x_2)\diamond_2(y_2\cdot_2 z_2)\big)\Big).
\end{align*}
Now, using the Theorem \ref{thmadmissibleleftassociator} we conclude that $\alpha(x)\cdot(y\diamond z)=\alpha(x)\diamond(y\cdot z)$. Therefore $as_A^{l}(x,y,z)=0$ and hence, $(A=A_1\otimes A_2,\cdot,\diamond,\varepsilon,\alpha)$ is an admissible Hom-Novikov-Poisson color Hom-algebra.
\end{proof}
By taking in Theorem \ref{tensor product},  $\alpha_1 =id_{A_1}$ and $\alpha_2 =id_{A_2}$, we have the following result
\begin{cor}
Let $(A_1, \cdot_1 ,\diamond_1,\varepsilon)$ and $(A_2, \cdot_2 ,\diamond_2,\varepsilon)$ be admissible Novikov-Poisson color Hom-algebras and let $A = A_1\otimes A_2$.
Define the operations $\cdot,\diamond: A\otimes A\rightarrow A$ by the following formulae for $x_i
, y_i \in \mathcal{H}(A_i),\ i\in\{1;2\}$,
\begin{eqnarray*}
   &(x_1\otimes x_2)\cdot(y_1\otimes y_2)=\varepsilon(x_2,y_1)(x_1\cdot_1 y_1)\otimes(x_2\cdot_2 y_2),\\
   &  (x_1\otimes x_2)\diamond(y_1\otimes y_2)=\varepsilon(x_2,y_1)\big((x_1\diamond_1 y_1)\otimes(x_2\cdot_2 y_2)+(x_1\cdot_1 y_1)\otimes(x_2\diamond_2 y_2)\big).
\end{eqnarray*}
Then $(A,\cdot,\diamond, \varepsilon)$ is an admissible Novikov-Poisson color Hom-algebra.
\end{cor}

By taking in Theorem \ref{tensor product}, $\Gamma=\{e\}$, we recover the following result
\begin{cor}[\cite{Yau:homnovikovpoissonalg}]
Let $(A_1, \cdot_1 ,\diamond_1,\alpha_1)$ and $(A_2, \cdot_2 ,\diamond_2,\alpha_2)$ be admissible Hom-Novikov-Poisson algebras and let $A = A_1\otimes A_2$.
Define the operations $\alpha: A \rightarrow A$ and $\cdot,\diamond: A\otimes A\rightarrow A$ by the following formulae for $x_i
, y_i \in A_i,\ i\in\{1;2\}$,
\begin{eqnarray*}
    &\alpha=\alpha_1\otimes\alpha_2,\\
   &(x_1\otimes x_2)\cdot(y_1\otimes y_2)=(x_1\cdot_1 y_1)\otimes(x_2\cdot_2 y_2),\\
   &  (x_1\otimes x_2)\diamond(y_1\otimes y_2)=\big((x_1\diamond_1 y_1)\otimes(x_2\cdot_2 y_2)+(x_1\cdot_1 y_1)\otimes(x_2\diamond_2 y_2)\big).
\end{eqnarray*}
Then $(A,\cdot,\diamond,\alpha)$ is an admissible Hom-Novikov-Poisson algebra.
\end{cor}

By taking in Theorem \ref{tensor product}, $\alpha_1 =id_{A_1},\alpha_2 =id_{A_2}$ and $\Gamma=\{e\}$, we have the following result
\begin{cor}[\cite{xu1}]
Let $(A_1, \cdot_1 ,\diamond_1)$ and $(A_2, \cdot_2 ,\diamond_2)$ be admissible Novikov-Poisson algebras and let $A = A_1\otimes A_2$.
Define the operations $\cdot,\diamond: A\otimes A\rightarrow A$ by the following formulae for $x_i
, y_i \in A_i,\ i\in\{1;2\}$,
\begin{eqnarray*}
   &(x_1\otimes x_2)\cdot(y_1\otimes y_2)=(x_1\cdot_1 y_1)\otimes(x_2\cdot_2 y_2),\\
   &  (x_1\otimes x_2)\diamond(y_1\otimes y_2)=\big((x_1\diamond_1 y_1)\otimes(x_2\cdot_2 y_2)+(x_1\cdot_1 y_1)\otimes(x_2\diamond_2 y_2)\big).
\end{eqnarray*}
Then $(A,\cdot,\diamond)$ is an admissible Novikov-Poisson algebra.
\end{cor}
\section{Hom-Gelfand-Dorfman color Hom-algebras}
\label{sec:homGelDorfmalgs}
In this section, our goals are to introduce Hom-Gelfand-Dorfman color Hom-algebras and to discuss some
basic properties and examples of these objects.
Moreover we
characterize the representation of Hom-Gelfand-Dorfman color Hom-algebras and provide some key constructions.
\begin{defn}\label{def of GF}
A Gelfand-Dorfman color Hom-algebra is a quadruple
$(A,\cdot,[\cdot,\cdot],\varepsilon)$ such that $(A,\cdot,\varepsilon)$ is a
Novikov color Hom-algebra and $(A,[\cdot,\cdot],\varepsilon)$ is a Lie color Hom-algebra satisfying for all $x,y,z\in \mathcal{H}(A)$, the following compatibility condition:
\begin{equation}
    \label{eq:GD}y\cdot[x,z]=\varepsilon(y,x)[x,y\cdot z]-\varepsilon(x+y,z)[z,y\cdot x]+[y,x]\cdot z-\varepsilon(x,z)[y,z]\cdot x.
  \end{equation}
\end{defn}
\begin{defn}\label{def of Hom-GF}
A Hom-Gelfand-Dorfman color Hom-algebra is defined as a quintuple
$(A,\cdot,[\cdot,\cdot],\varepsilon,\alpha)$ such that $(A,\cdot,\varepsilon,\alpha)$ is a
Hom-Novikov color Hom-algebra and $(A,[\cdot,\cdot],\varepsilon,\alpha)$ is a Hom-Lie color Hom-algebra satisfying for all $x,y,z\in \mathcal{H}(A)$, the following compatibility condition:
\begin{equation}
    \label{eq:Hom-GD}
\begin{array}{rcl}
\alpha(y)\cdot[x,z]&=&
\varepsilon(y,x)[\alpha(x),y\cdot z]-\varepsilon(x+y,z)[\alpha(z),y\cdot x]+[y,x]\cdot \alpha(z)\\
    &&\hspace{6cm} -\varepsilon(x,z)[y,z]\cdot \alpha(x).
\end{array}
\end{equation}
A Hom-Gelfand-Dorfman color Hom-algebra is called multiplicative if
the even linear map $\alpha:A\rightarrow A$ is multiplicative with respect to $\cdot$ and $[\cdot,\cdot]$, that is,
for all $x,y \in \mathcal{H}(A)$,
$$
\alpha (x \cdot y) =\alpha (x) \cdot \alpha (y), \quad \alpha ([x ,y]) =[\alpha (x) ,\alpha (y)].
$$
\end{defn}
\begin{rmk} Hom-Gelfand-Dorfman color Hom-algebras contain the Gelfand-Dorfman algebras and the Hom-Gelfand-Dorfman Hom-algebras
for special choices of grading group and the twisting map.
\begin{enumerate}
\item
When $\Gamma = \{e\}$ and $\alpha= id$, we get Gelfand-Dorfman algebra \cite{GelfandDorfman1979:Hamoperatandassociatoralgebraic,Xu:quqdrqtic conformalsuperalgbs}.
\item
When $\Gamma = \{e\}$ and $\alpha \neq id$, we get Hom-Gelfand-Dorfman Hom-algebra \cite{Hom-GelfDorf}.
 \end{enumerate}
\end{rmk}
\begin{ex}
Let $\Gamma = \mathbb{Z}_2\times\mathbb{Z}_2$ be an abelian group and $A$ be a $4$-dimensional $\Gamma$-graded
linear space with one-dimensional homogeneous subspaces
$$A_{(0,0)}=<e_1>,\quad A_{(0,1)}=<e_2>,\quad A_{(1,0)}=<e_3>,\quad A_{(1,1)}=<e_4>.$$
Then $(A,\cdot,[\cdot,\cdot],\varepsilon,\alpha)$ is a Hom-Gelfand-Dorfman color Hom-algebra with
\begin{align*}
&\text{the bicharacter:} \quad \varepsilon\big((i_1,i_2),(j_1,j_2)\big)=(-1)^{i_1j_1+i_2j_2},\\
&\text{the multiplication:} \quad e_2\cdot e_3=\lambda_1e_4,\quad e_3\cdot e_2=\lambda_2e_4,\quad e_3\cdot e_3=\lambda_3 e_1,\quad\lambda_i\in\mathbb{K},\\
&\text{the bracket:} \quad [e_2,e_2]=\mu_1e_1,\quad [e_3,e_2]=\mu_2e_4,\quad\mu_i\in\mathbb{K},\\
&\text{the even linear map $\alpha:A\rightarrow A$ given by} \quad
\begin{array}[t]{l}
\alpha(e_1)=-e_1,\quad \alpha(e_2)=2e_2,\\
\alpha(e_3)=-2e_3,\quad \alpha(e_4)=e_4.\\
\end{array}
\end{align*}
\end{ex}
\begin{defn}\label{morphHom-GelfandDorfmancoloralgebra}
Let $( A, \cdot, [\cdot,\cdot], \varepsilon,\alpha)$ and $( A', \cdot',  [\cdot,\cdot]',\varepsilon', \alpha')$ be Hom-Gelfand-Dorfman color Hom-algebras. A linear map of degree zero $f:  A\rightarrow  A'$ is a Hom-Gelfand-Dorfman color Hom-algebra morphism if
\begin{eqnarray*}
\cdot'\circ(f\otimes f)= f\circ\cdot,\quad  [\cdot,\cdot]'\circ(f\otimes f)= f\circ[\cdot,\cdot], \quad f\circ\alpha= \alpha'\circ f.
\end{eqnarray*}
\end{defn}
\begin{prop}
Let $(A,\cdot,\varepsilon,\alpha)$ be a Hom-Novikov color Hom-algebra. For all $x,y\in \mathcal{H}(A)$, let
\begin{equation}
[x,y]=x\cdot y-\varepsilon(x,y) y\cdot x.
\end{equation}
Then $(A,\cdot,[\cdot,\cdot],\varepsilon,\alpha)$ is a Hom-Gelfand-Dorfman color Hom-algebra.
\end{prop}
\begin{proof}
Let $(A,\cdot,\varepsilon,\alpha)$ be a Hom-Novikov color Hom-algebra. By Proposition \ref{Hom-Novik to Hom-Lie} $(A,[\cdot,\cdot],\varepsilon,\alpha)$ is a Hom-Lie color Hom-algebra. Now, we show that the compatibility condition \eqref{eq:Hom-GD} is satisfied. For any $x,y,z\in\mathcal{H}(A)$ we have
\begin{align*}
    &\alpha(y)\cdot[x,z]-\varepsilon(y,x)[\alpha(x),y\cdot z]+\varepsilon(x+y,z)[\alpha(z),y\cdot x]-[y,x]\cdot\alpha(z)\\&\quad\quad+\varepsilon(x,z)[y,z]\cdot\alpha(x)\\
    &\quad=\alpha(y)\big(x\cdot z-\varepsilon(x,y+z)(y\cdot z)\cdot\alpha(x)\big)+\varepsilon(x+y,z)\big(\alpha(z)\cdot(y\cdot x)\\&\quad\quad-\varepsilon(z,y+x)(y\cdot x)\cdot\alpha(z)\big)-\big(y\cdot x-\varepsilon(y,x) x\cdot y)\big)\cdot\alpha(z)+\varepsilon(x,z)\big(y\cdot z\\&\quad\quad-\varepsilon(y,z) z\cdot y\big)\cdot\alpha(x)\\
    &\quad=\alpha(y)\cdot(x\cdot z)-\varepsilon(x,z)\alpha(y)\cdot(z\cdot x)-\varepsilon(y,x)\alpha(x)\cdot(y\cdot z)+\varepsilon(x,z)(y\cdot z)\cdot\alpha(x)\\&\quad\quad+\varepsilon(x+y,z)\alpha(z)\cdot(y\cdot x)-(y\cdot x)\cdot\alpha(z)-(y\cdot x)\cdot\alpha(z)+\varepsilon(y,x)(x\cdot y)\cdot\alpha(z)\\&\quad\quad+\varepsilon(x,z)(y\cdot z)\cdot\alpha(x)-\varepsilon(x+y,z)(x\cdot y)\cdot\alpha(x)+\varepsilon(y,x)(x\cdot y)\cdot\alpha(z)\\
    &\quad=\underbrace{\Big(\alpha(y)\cdot(x\cdot z)-(y\cdot x)\cdot\alpha(z)-\varepsilon(y,x)\big(\alpha(x)\cdot(y\cdot z)-(x\cdot y)\cdot\alpha(z)\big)\Big)}_{\text{$=0$ by \eqref{eqhomNovikov1}}}\\
    &\quad\quad+\varepsilon(x,z)\underbrace{\Big((y\cdot z)\cdot\alpha(x)-\alpha(y)\cdot(z\cdot x)-\varepsilon(y,z)\big(\alpha(z)\cdot(y\cdot x)-(z\cdot y)\cdot\alpha(x)\Big)}_{\text{$=0$ by \eqref{eqhomNovikov1}}}\\
    &\quad\quad-\underbrace{\Big((y\cdot x)\cdot\alpha(z)-\varepsilon(x,z)(y\cdot z)\cdot\alpha(x)\Big)}_{\text{$=0$ by \eqref{eqhomNovikov2}}}=0.
\end{align*}
Hence, $(A,\cdot,[\cdot,\cdot],\varepsilon,\alpha)$ is a Hom-Gelfand-Dorfman color Hom-algebra.
\end{proof}
 \begin{defn}
Let $(A,\cdot,[\cdot,\cdot],\varepsilon,\alpha)$ be a Hom-Gelfand-Dorfman color Hom-algebra. A $\Gamma$-graded subspace $H$ of $A$ is called
\begin{enumerate}
\item
color Hom-subalgebra of $(A,\cdot,[\cdot,\cdot],\varepsilon,\alpha)$ if
$$\alpha(H)\subseteq H,\ H\cdot H \subseteq H,\ [H,H]\subseteq H,$$
\item
color Hom-ideal of $(A,\cdot,[\cdot,\cdot],\varepsilon,\alpha)$ if
$$\alpha(H)\subseteq H,\ A\cdot H\subseteq H,\ H\cdot A\subseteq H,\ [A,H]\subseteq H.$$
  \end{enumerate}
\end{defn}
\begin{prop}
Let $(A,\cdot,[\cdot,\cdot],\varepsilon,\alpha)$ a Hom-Gelfand-Dorfman color Hom-algebra and $I$ a color Hom-ideal of
$(A,\cdot,[\cdot,\cdot],\varepsilon,\alpha)$. Then $(A/ I,\overline{\cdot},\{\cdot,\cdot\},\varepsilon,\alpha)$ is a Hom-Gelfand-Dorfman color Hom-algebra where $\overline{x}~\overline{\cdot}~\overline{y} =\overline{x\cdot y}$, $\{\overline{x},\overline{y}\}=\overline{[x,y]},\overline{\alpha}(\overline{x})=\overline{\alpha(x)}$ and $\varepsilon(\overline{x},\overline{y})=\varepsilon(x,y)$, for all $\overline{x},\overline{y} \in \mathcal{H}(A/I)$.
\end{prop}
\begin{proof}
It follows from a straightforward computation.
\end{proof}
\begin{prop}
 Any transposed Hom-Poisson color Hom-algebra is a Hom-Gelfand-Dorfman color Hom-algebra.
\end{prop}
\begin{proof}
Let $(A,\cdot,[\cdot,\cdot],\varepsilon,\alpha)$ be a transposed Hom-Poisson color Hom-algebra.
By definition $(A,\cdot,\varepsilon,\alpha)$ be a $\varepsilon$-commutative Hom-associative color Hom-algebra, then $(A,\cdot,\varepsilon,\alpha)$ is a Hom-Novikov color Hom-algebra and $(A,[\cdot,\cdot],\varepsilon,\alpha)$ is a Hom-Lie color Hom-algebra. Now, we show that the compatibility condition \eqref{eq:Hom-GD} is satisfied. For any $x,y,z\in\mathcal{H}(A)$ we have
\begin{align*}
    &\alpha(y)\cdot[x,z]-\varepsilon(y,x)[\alpha(x),y\cdot z]+\varepsilon(x+y,z)[\alpha(z),y\cdot x]\\
    &-[y,x]\cdot\alpha(z)+\varepsilon(x,z)[y,z]\cdot\alpha(x)\\
    &\quad=\alpha(y)\cdot\big(x\cdot z-\varepsilon(x,z) z\cdot x\big)-\varepsilon(y,x)\big(\alpha(x)\cdot(y\cdot z)-\varepsilon(x,y+z)(y\cdot z)\cdot\alpha(x)\big)\\
    &\quad\quad-\varepsilon(x+y,z)\big(\alpha(z)\cdot(y\cdot x)-\varepsilon(z,y+x)(y\cdot x)\cdot\alpha(z)\big)\\
    &\quad\quad+\big(y\cdot x-\varepsilon(y,x) x\cdot y\big)\cdot\alpha(z)-\varepsilon(x,z)\big(y\cdot z-\varepsilon(y,z)z\cdot y\big)\cdot\alpha(x)\\
    &\quad=\alpha(y)\cdot(x\cdot z)-\varepsilon(x,z)\alpha(y)\cdot(z\cdot x)-\varepsilon(y,x)\alpha(x)\cdot(y\cdot z)+\varepsilon(x,z)(y\cdot z)\cdot\alpha(x)\\
    &\quad\quad+\varepsilon(x+y,z)\alpha(z)\cdot(y\cdot x)-(y\cdot x)\cdot\alpha(z)-(y\cdot x)\cdot\alpha(z)+\varepsilon(y,x)(x\cdot y)\cdot\alpha(z)\\
    &\quad\quad+\varepsilon(x,z)(y\cdot z)\cdot\alpha(x)-\varepsilon(x+y,z)(z\cdot y)\cdot\alpha(x)\\
   & \quad=\underbrace{\big(\alpha(y)\cdot(x\cdot z)-(y\cdot x)\cdot\alpha(z)\big)}_{\text{$=0$ by \eqref{Homass:homassociator}}} -\varepsilon(x,z)\underbrace{\big(\alpha(y)\cdot(z\cdot x)-(y\cdot z)\cdot\alpha(x)\big)}_{\text{$=0$ by \eqref{Homass:homassociator}}}\\
   &\quad\quad-\varepsilon(y,x)\underbrace{\big(\alpha(x)\cdot(y\cdot z)-(x\cdot y)\cdot\alpha(z)\big)}_{\text{$=0$ by \eqref{Homass:homassociator}}}+\varepsilon(x+y,z)\underbrace{\big(\alpha(z)\cdot(y\cdot x)-(z\cdot y)\cdot\alpha(x)\big)}_{\text{$=0$ by \eqref{Homass:homassociator}}}\\&\quad\quad+\underbrace{\big(\varepsilon(x,z)(y\cdot z)\cdot\alpha(x)-(y\cdot x)\cdot\alpha(z)\big)}_{\text{$=0$ by \eqref{Homass:homassociator} and \eqref{commutative:homcoloralgeb}}}=0,
\end{align*}
which completes the proof.
\end{proof}
\begin{lem}\label{lemma commass to Nov}\cite{Bakayoko2016arXiv:Hom-Novikovcoloralgebras}
Let $(A, \cdot, \varepsilon,\alpha)$ be a $\varepsilon$-commutative Hom-associative color Hom-algebra with an even derivation $D$ such that $\alpha\circ D=D\circ\alpha$. Define
\begin{equation}\label{commass to Nov} x \diamond y = x\cdot D(y),\quad for~all~x,y\in \mathcal{H}(A). \end{equation}
Then $(A, \diamond, \varepsilon,\alpha)$ is a Hom-Novikov color Hom-algebra.
\end{lem}
\begin{thm}
Let $(A,\cdot,[\cdot,\cdot],\varepsilon,\alpha)$ be a Hom-Poisson color Hom-algebra with an even derivation $D$ relative to the
both products. Define a new operation $\diamond$ on $A$ in the following way:
\begin{equation}\label{special}
    x\diamond y=x\cdot D(y).
\end{equation}
Then $(A, \diamond, [\cdot,\cdot],\varepsilon,\alpha)$ is a Hom-Gelfand-Dorfman color Hom-algebra.
\end{thm}
\begin{proof}
By Lemma \ref{lemma commass to Nov}, $(A, \diamond,\varepsilon,\alpha)$ is a Hom-Novikov color Hom-algebra and by definition of Hom-Poisson color Hom-algebra, $(A, [\cdot,\cdot],\varepsilon,\alpha)$ is a Hom-Lie color Hom-algebra. Now, we show that the compatibility condition \eqref{eq:Hom-GD} is satisfied. For any $x,y,z\in\mathcal{H}(A)$,
\begin{align*}
&\alpha(y)\diamond[x,z]-\varepsilon(y,x)[\alpha(x),y\diamond z]+\varepsilon(x+y,z)[\alpha(z),y\diamond x]\\
&\quad \quad -[y,x]\diamond\alpha(z)+\varepsilon(x,z)[y,z]\diamond\alpha(x)\\
& \text{\scriptsize (using \eqref{commass to Nov})}\\
&\quad=\alpha(y)\cdot D([x,z])-\varepsilon(y,x)[\alpha(x),y\cdot D(z))]+\varepsilon(x+y,z)[\alpha(z),y\cdot D(x)]\\
&\quad\quad-[y,x]\cdot D(\alpha(z))+\varepsilon(x,z)[y,z]\cdot D(\alpha(x))\\
&\text{\scriptsize ($D$ is derivation)}\\
&\quad=\alpha(y)\cdot[D(x),z]+\alpha(y)\cdot[x,D(z)]-\varepsilon(y,x)[\alpha(x),y\cdot D(z)]\\&\quad\quad+\varepsilon(x+y,z)[\alpha(z),y\cdot D(x)]-[y,x]\cdot\alpha(D(x))\\&\quad\quad+\varepsilon(x,z)[y,z]\cdot\alpha(D(x))\\
&\quad=-\varepsilon(y,x)\underbrace{\Big([\alpha(x),y\cdot D(z)]-\varepsilon(x,y)\alpha(y)\cdot[x,D(z)]-\varepsilon(y+x,z)\alpha(D(x))\cdot[x,y]\Big)}_{\text{$=0$ by \eqref{CompatibiltyPoissonLeibform}}}\\*[0,5cm]
&\quad\quad+\varepsilon(x+y,z)\underbrace{\Big([\alpha(z),y\cdot D(x)]-\varepsilon(z,y)\alpha(y)\cdot[z,D(x)]-\alpha(D(x))\cdot[z,y]\Big)}_{\text{$=0$ by \eqref{CompatibiltyPoissonLeibform}}}=0.
\qedhere \end{align*}
\end{proof}
Let us call a Hom-Gelfand-Dorfman color Hom-algebra $(A,\cdot,[\cdot,\cdot],\varepsilon,\alpha)$ is special if it can be embedded into a differential Hom-Poisson
color Hom-algebra with operations $[\cdot,\cdot]$ and $\diamond$ given by \eqref{special}.
\begin{defn}
Let $(A,\cdot,[\cdot,\cdot],\varepsilon,\alpha)$ be a Hom-Gelfand-Dorfman color Hom-algebra. A representation of $(A,\cdot,[\cdot,\cdot],\varepsilon,\alpha)$ is a quintuple $(l,r,\rho,\beta,V)$ such that $(l,r,\beta,V)$ is a bimodule of the Hom-Novikov color Hom-algebra $(A,\cdot,\varepsilon,\alpha)$ and $(\rho,\beta,V)$ is a representation of the Hom-Lie color Hom-algebra  $(A,[\cdot,\cdot],\varepsilon,\alpha)$ satisfying, for all $ x, y \in  \mathcal{H}(A), v \in \mathcal{H}(V) $,
\begin{eqnarray}
\label{Cond1 GD}
l(\alpha(y))\rho(x)v=\rho(y\cdot x)\beta(v)&+&\varepsilon(y,x)\rho(\alpha(x))l(y)v\nonumber\\&-&\varepsilon(x,v)r(\alpha(x))\rho(y)v+l([y,x])\beta(v),\\
\label{Cond2 GD}
r([x,y])\beta(v)&=&\varepsilon(v,x)(\rho(\alpha(x))r(y)v-r(\alpha(y))\rho(x)v)\nonumber\\&+&\varepsilon(x+v,y)(r(\alpha(x))\rho(y)v-\rho(\alpha(y))r(x)v).
\end{eqnarray}
\end{defn}
\begin{prop}\label{pro2:repHomGelfandDorfmancolorHomalg}
Let $(A,\cdot,[\cdot,\cdot],\varepsilon,\alpha)$ be a Hom-Gelfand-Dorfman color Hom-algebra, and let $(l,r,\rho,\beta,V)$ its representation. Then, $(A\oplus V,\cdot',[\cdot,\cdot]',\varepsilon,\alpha+\beta)$ is a Hom-Gelfand-Dorfman color Hom-algebra, where $(A\oplus V,\cdot',\varepsilon,\alpha+\beta)$ is the semi-direct product Hom-Novikov color Hom-algebra $A\ltimes_{l,r,\alpha,\beta} V$, and $(A\oplus V,[\cdot,\cdot]',\varepsilon,\alpha+\beta)$ is the semi-direct product Hom-Lie color Hom-algebra $A\ltimes_{\rho,\alpha,\beta} V$.
\end{prop}
\begin{proof}
Let $(l,r,\rho,\beta,V)$ be a representation of a Hom-Gelfand-Dorfman color Hom-algebra $(A,\cdot,[\cdot,\cdot],\varepsilon,\alpha)$.
By Proposition \ref{semi direct Hom Novikov} and Proposition \ref{semi direct Hom-Lie},
$(A\oplus V,\cdot',\varepsilon,\alpha+\beta)$ is a Hom-Novikov color Hom-algebra, and
$(A\oplus V,[\cdot,\cdot]',\varepsilon,\alpha+\beta)$ is a Hom-Lie color Hom-algebra respectively.
Now, we show that the compatibility condition \eqref{eq:Hom-GD} is satisfied.
For all $X_i=x_i+v_i\in A_{\gamma_i} \oplus V_{\gamma_i},\quad i=1,2,3$ we have
\begin{align*}
&(\alpha+\beta)(x_2+v_2)\ast[x_1+v_1,x_3+v_3]'\\&-\varepsilon(x_2+v_2,x_1+v_1)[(\alpha+\beta)(x_1+v_1),(x_2+v_2)\cdot'(x_3+v_3)]'\\&+\varepsilon(x_1+x_2,x_3)[(\alpha+\beta)(x_3+v_3),(x_2+v_2)\cdot'(x_1+v_1)]'\\
&-[(x_2+v_2),(x_1+v_1)]'\cdot'(\alpha+\beta)(x_3+v_3)\\
&+\varepsilon(x_1,x_3)[(x_2+v_2),(x_3+v_3)]'\cdot'(\alpha+\beta)(x_1+v_1)\\
&\quad=(\alpha(x_2)+\beta(v_2))\cdot'([x_1,x_3]+\rho(x_1)v_3-\varepsilon(x_1,x_3)\rho(x_3)v_1)\\
&\quad\quad-\varepsilon(x_2,x_1)[\alpha(x_1)+\beta(v_1),x_2\cdot x_3+l(x_2)v_3+r(x_3)v_2]'\\
&\quad\quad+\varepsilon(x_1+x_2,x_3)[\alpha(x_3)+\beta(v_3),x_2\cdot x_1+l(x_2)v_1+r(x_1)v_2]'\\
&\quad\quad-([x_2,x_1]+\rho(x_2)v_1-\varepsilon(x_2,x_1)\rho(x_1)v_2)\cdot'(\alpha(x_3)+\beta(v_3))\\
&\quad\quad+\varepsilon(x_1,x_3)([x_2,x_3]+\rho(x_2)v_3-\varepsilon(v_2,x_3)\rho(x_3)v_2\\
&\quad=\alpha(x_2)\cdot[x_1,x_3]+l(\alpha(x_2))\rho(x_1)v_3\\
&\quad\quad-\varepsilon(x_1,x_3)l(\alpha(x_2))\rho(x_3)v_1+r([x_1,x_3])\beta(v_2)\\
&\quad\quad-\varepsilon(x_2,x_1)\Big([\alpha(x_1),x_2\cdot x_3]+\rho(\alpha(x_1))l(x_2)v_3\\
&\quad\quad+\rho(\alpha(x_1))r(x_3)v_2-\varepsilon(x_1,x_2+x_3)\rho(x_2\cdot x_3)\beta(v_1)\Big)\\
&\quad\quad+\varepsilon(x_1+x_2,x_3)\Big([\alpha(x_3),x_2\cdot x_1]+\rho(\alpha(x_3))l(x_2)v_1\\
&\quad\quad+\rho(\alpha(x_3))r(x_1)v_2-\varepsilon(v_3,x_1+x_2)\rho(x_2\cdot x_1)\beta(v_3)\Big)\\
&\quad\quad-\Big([x_2,x_1]\cdot\alpha(x_3)+l([x_2,x_1])\beta(v_3)\\
&\quad\quad+r(\alpha(x_3))\rho(x_2)v_1-\varepsilon(x_2,x_1)r(\alpha(x_3))\rho(x_1)v_2\\
&\quad\quad+\varepsilon(x_1,x_3)\Big([x_2,x_3]\cdot\alpha(x_1)+l([x_2,x_3])\beta(v_1)\\
&\quad\quad+r(\alpha(x_1))\rho(x_2)v_3-\varepsilon(x_2,x_3)r(\alpha(x_1))\rho(x_3)v_2\\
&\quad =\alpha(x_2)\cdot[x_1,x_3]+l(\alpha(x_2))\rho(x_1)v_3\\
&\quad\quad-\varepsilon(x_1,x_3)l(\alpha(x_2))\rho(x_3)v_1+r([x_1,x_3])\beta(v_2)\\
&\quad\quad-\varepsilon(x_2,x_1)[\alpha(x_1),x_2\cdot x_3]-\varepsilon(x_2,x_1)\rho(\alpha(x_1))l(x_2)v_3\\
&\quad\quad-[x_2,x_1]\cdot\alpha(x_3)-l([x_2,x_1])\beta(v_3)\\
&\quad\quad-r(\alpha(x_3))\rho(x_2)v_1+\varepsilon(x_2,x_1)r(\alpha(x_3))\rho(x_1)v_2\\
&\quad\quad+\varepsilon(x_1,x_3)[x_2,x_3]\cdot\alpha(x_1)+\varepsilon(x_1,x_3)l([x_2,x_3])\beta(v_1)\\
&\quad\quad\varepsilon(x_1,x_3)r(\alpha(x_1))\rho(x_2)v_3-\varepsilon(x_1+x_2,x_3)r(\alpha(x_1))\rho(x_3)v_2\\
&\quad=\Big(\alpha(x_2)\cdot[x_1,x_3]-\varepsilon(x_2,x_1)[\alpha(x_1),x_2\cdot x_3]\\
&\quad\quad+\varepsilon(x_1+x_2,x_3)[\alpha(x_3),x_2\cdot x_1]-[x_2,x_1]\cdot\alpha(x_3)\\
&\quad\quad+\varepsilon(x_1,x_3)[x_2,x_3]\cdot\alpha(x_1)\Big)\\
&\quad\quad+\Big(l(\alpha(x_2))\rho(x_1)v_3-\rho(x_2\cdot x_1)\beta(v_3)-\varepsilon(x_2,x_1)\rho(\alpha(x_1))l(x_2)v_3\\
&\quad\quad+\varepsilon(x_1,x_3)r(\alpha(x_1))\rho(x_2)v_3-l([x_2,x_1]\beta(v_3)\Big)\\
&\quad\quad-\varepsilon(x_1,x_3)\Big(l(\alpha(x_2))\rho(x_3)v_1-\rho(x_2\cdot x_3)\beta(v_1)\\&\quad\quad-\varepsilon(x_2,x_3)\rho(\alpha(x_3))l(x_2)v_1
+\varepsilon(x_3,x_1)r(\alpha(x_3))\rho(x_2)v_1-l([x_2,x_3])\beta(v_1)\Big)\\
&\quad\quad+\Big(r([x_1,x_3])\beta(v_2)-\varepsilon(x_2,x_1)(\rho(\alpha(x_1))r(x_3)v_2-r(\alpha(x_3))\rho(x_1)v_2)\\
&\quad\quad-\varepsilon(x_1+x_2,x_3)(r(\alpha(x_1))\rho(x_3)v-\rho(\alpha(x_3))r(x_1)v_2)\Big)=0.
\end{align*}
Hence, $(A\oplus V,\cdot',[\cdot,\cdot]',\varepsilon,\alpha+\beta)$ is a Hom-Gelfand-Dorfman color Hom-algebra.
\end{proof}
\begin{exes} Some important examples of representations of Hom-Gelfand-Dorfman color Hom-algebras can be constructed as follows.
\begin{enumerate}
\item
Let $(A,\cdot,[\cdot,\cdot],\varepsilon,\alpha)$ be a Hom-Gelfand-Dorfman color Hom-algebra. If
$$L(a)b=a\cdot b, \quad R(a)b=b\cdot a, \quad ad(a)b=[a, b]=-\varepsilon(a,b)[b,a], \quad \forall a,b\in \mathcal{H}(A),$$
then $(L,R,ad,\alpha,A)$ is a representation of $(A,\cdot,[\cdot,\cdot],\varepsilon,\alpha)$.
\item
If $f:\mathcal{A}=(A,\cdot_1,[\cdot,\cdot]_1,\varepsilon,\alpha)\rightarrow(A',\cdot_2,[\cdot,\cdot]_2,\varepsilon,\beta)$ is a Hom-Gelfand-Dorfman color Hom-algebras morphism, then
$(l,r,\rho,\beta,A')$
becomes a representation of $\mathcal{A}$ via $f$, that is, for all $(x,y)\in \mathcal{H}(A)\times \mathcal{H}(A')$,
$$l(x)y=f(x)\cdot_2 y,\quad  r(x)y=y\cdot_2 f(x), \quad \rho(x)y=[f(x), y]_2.$$
\end{enumerate}
\end{exes}
\begin{thm}
Let $\mathcal{A}=(A,\cdot_A,[\cdot,\cdot]_A,\varepsilon,\alpha)$ and $\mathcal{B}=(B,\cdot_B,[\cdot,\cdot]_{B},\varepsilon,\beta)$ be Hom-Gelfand-Dorfman color Hom-algebras. Suppose that there are such even linear maps $l_A,r_A,\rho_A:A\rightarrow End(B)$
and $l_B,r_B,\rho_B:B\rightarrow End(A)$ that $A\bowtie^{\rho_A,\beta}_{\rho_B,\alpha}B$ is a matched pair of Hom-Lie color Hom-algebras, and  $A\bowtie^{l_A,r_A,\beta}_{l_B,r_B,\alpha}B$ is a matched pair of Hom-Novikov color Hom-algebras, and for all $x,y\in \mathcal{H}(A),~a,b\in \mathcal{H}(B)$, the following equalities hold:
\begin{align}\label{101}
&\begin{array}{l}
r_A(\rho_B(a)x)\beta(b)-r_A(\alpha(x))([b,a])\\
=\varepsilon(a,x)\big(\beta(b)\cdot_B\rho_A(x)a-\rho_A(l_B(b)x)\beta(a)\\
\quad -r_A(\rho_B(b)x)\beta(a)\big)-\varepsilon(a+b,x)\big(\rho_A(\alpha(x))(b\cdot_B a)\\
\quad -\rho_A(x)b\cdot_B\beta(a) +\varepsilon(b,a)[\beta(a),r_A(x)b]_B,
\end{array} \\[0.2cm]
\label{102}
&\begin{array}{l}
l_A(\alpha(x))([a,b]_B)-\rho_A(x)a\cdot_B\beta(b)-\rho_A(r_B(a)x)\beta(b)\\
=\varepsilon(x,a)\big([\beta(a),l_A(x)b]_B-r_A(\rho_B(a)x)\beta(b)\big)\\
\quad +\varepsilon(a,b)\big(\rho_A(r_B(b)x)\beta(a)-\rho_A(x)b\cdot_B\beta(a)\big)\\
\quad +\varepsilon(a+x,b)\big(l_A(\rho_B(b)x)\beta(a)-[\beta(b),l_A(x)a]_B\big),
\end{array}
\\[0.2cm]
\label{103}
&\begin{array}{l}
r_B(\rho_A(x)a)\alpha(y)-r_B(\beta(a))([y,x])\\
=\varepsilon(x,a)\big(\alpha(y)\cdot_A\rho_B(a)x-\rho_B(l_A(y)a)\alpha(x)\\
\quad -r_B(\rho_A(y)a)\alpha(x)\big)-\varepsilon(x+y,a)\big(\rho_B(\beta(a))(y\cdot_A x)\\
\quad -\rho_B(a)y\cdot_A\alpha(x)+\varepsilon(y,x)[\alpha(x),r_B(a)y]_A,
\end{array}\\[0.2cm]
\label{104}
&\begin{array}{l}
l_B(\beta(a))([x,y]_A)-\rho_B(a)x\cdot_A\alpha(y)-\rho_B(r_A(x)a)\alpha(y)\\
=\varepsilon(a,x)\big([\alpha(x),l_B(a)y]_A-r_B(\rho_A(x)a)\beta(y)\big)\\
\quad +\varepsilon(x,y)\big(\rho_B(r_A(y)a)\alpha(x)-\rho_B(a)y\cdot_A\alpha(x)\big)\\
\quad +\varepsilon(a+x,y)\big(l_B(\rho_A(y)a)\alpha(x)-[\alpha(y),l_B(a)x]_A\big),
\end{array}
\end{align}
Then, $(A,B,l_A,r_A,\rho_A,\beta,l_B,r_B,\rho_B,\alpha)$ is called a matched pair of the Hom-Gelfand-Dorfman color Hom-algebras. In this case, on the direct sum
$A\oplus B$ of the underlying linear spaces of $\mathcal{A}$ and $\mathcal{B}$, there is a Hom-Gelfand-Dorfman color Hom-algebra structure which is given for any $x+a\in A_{\gamma_1}\oplus B_{\gamma_1}, y+b\in A_{\gamma_2}\oplus B_{\gamma_2}$ by
\begin{eqnarray}
(x + a) \cdot(y + b)&=&x \cdot_A y + (s_A(x)b + \varepsilon(a,y)s_A(y)a)\cr&+& a \cdot_B b + (s_B(a)y + \varepsilon(x,b)s_B(b)x),\\
 ~[x + a ,y + b] &=&[x ,y]_A + (\rho_A(x)b -\rho_A(y)a)\cr&+&[a , b]_B + (\rho_B(a)y -\rho_B(b)x).
\end{eqnarray}
\end{thm}
\begin{proof}
By Proposition \ref{thm:matchedpairs}  and Proposition \ref{matched Lie},  we deduce that $(A\oplus B,\cdot,\varepsilon,\alpha+\beta)$
is a Hom-Novikov color Hom-algebra and $(A\oplus B,[\cdot,\cdot],\alpha+\beta)$ is a Hom-Lie color Hom-algebra.
Now, the rest, it is easy (in a similar way as for Proposition \ref{prop:matched.pairs.comm.Hom-ass.coloralgeb}) to verify the compatibility condition satisfied.
\end{proof}
Taking the color $\varepsilon$-commutator in a Hom-Novikov-Poisson color Hom-algebra, we obtain the following result.
\begin{thm}\label{homNovikovPoissoncoloralg to HGD}
Let $(A,\cdot,\diamond,\varepsilon,\alpha)$ be a Hom-Novikov-Poisson color Hom-algebra, and
\begin{equation}\label{Nov-Poiss to GD}
    [x, y] = x \diamond y - \varepsilon(x,y)y \diamond x, \quad\forall x, y \in \mathcal{H}(A).
\end{equation}
Then $(A, \cdot, [\cdot,\cdot],\varepsilon,\alpha)$ is a Hom-Gelfand-Dorfman color Hom-algebra.
\end{thm}

\begin{proof}
By definition $(A,\cdot,\varepsilon,\alpha)$ is a $\varepsilon$-commutative Hom-associative color Hom-algebra. Then $(A,\cdot,\varepsilon,\alpha)$ is a Hom-Novikov color Hom-algebra. Moreover, by Proposition \ref{Hom-Novik to Hom-Lie},  $(A,[\cdot,\cdot],\varepsilon,\alpha)$ is a Hom-Lie color Hom-algebra. Now, we show that the compatibility condition \eqref{eq:Hom-GD} is satisfied. For any $x,y,z\in\mathcal{H}(A)$,
\begin{align*}
    &\alpha(y)\cdot[x,z]-\varepsilon(y,x)[\alpha(x),y\cdot z]+\varepsilon(x+y,z)[\alpha(z),y\cdot x]\\
    &\quad\quad-[y,x]\cdot\alpha(z)+\varepsilon(x,z)[y,z]\cdot\alpha(x)\\
    &\text{\scriptsize (using \eqref{Nov-Poiss to GD})}\\
    &=\alpha(y)\cdot\big(x\diamond z-\varepsilon(x,z) z\diamond x\big)-\varepsilon(y,x)\big(\alpha(x)\diamond(y\cdot z)\\
    &\quad\quad-\varepsilon(x,y+z)(y\cdot z)\diamond\alpha(x)\big)+\varepsilon(x+y,z)\big(\alpha(z)\diamond(y\cdot x)\\
    &\quad\quad-\varepsilon(z,y+x)(y\cdot x^)\diamond\alpha(z)\big)-\big(y\diamond x-\varepsilon(y,x) x\diamond y\big)\cdot\alpha(z)\\
    &\quad\quad+\varepsilon(x,z)\big((y\diamond z-\varepsilon(y,z) z\diamond y)\cdot\alpha(x)\big)\\
    &=\alpha(y)\cdot(x\diamond z)-\varepsilon(x,z)\alpha(y)\cdot(z\diamond x)-\varepsilon(y,x)\alpha(x)\diamond(y\cdot z)\\&\quad\quad+\varepsilon(x,z)(y\cdot z)\diamond\alpha(x)+\varepsilon(x+y,z)\alpha(z)\diamond(y\cdot x)\\&\quad\quad-(y\cdot x)\diamond\alpha(z)-(y\diamond x)\cdot\alpha(z)+\varepsilon(y,x)(x\diamond y)\cdot\alpha(z)\\
    &\quad\quad+\varepsilon(x,z)(y\diamond z)\cdot\alpha(x)-\varepsilon(x+y,z)(z\circ y)\cdot\alpha(x)\\
    &=\varepsilon(y,x)\underbrace{\Big((x\diamond y)\cdot\alpha(z)-\alpha(x)\diamond(y\cdot z)-\varepsilon(x,y)\big((y\diamond x)\cdot\alpha(z)-\alpha(y)\cdot(x\diamond z)\big)\Big)}_{\text{$=0$ by \eqref{HomNovPoisscoloralg2}}}\\
    &-\varepsilon(x+y,z)\underbrace{\Big((z\diamond y)\cdot\alpha(x)-\alpha(z)\diamond(y\cdot x)-\varepsilon(z,y)\big((y\diamond z)\cdot\alpha(x)-\alpha(y)\diamond(z\cdot x)\big)\Big)}_{\text{$=0$ by \eqref{HomNovPoisscoloralg2}}}\\
    &\quad\quad-\underbrace{\Big((y\cdot x)\diamond\alpha(z)-\varepsilon(x,z)(y\diamond z)\cdot\alpha(x)\Big)}_{\text{$=0$ by \eqref{HomNovPoisscoloralg1}}}=0.
\qedhere
\end{align*}
\end{proof}

\begin{ex}
Let $\Gamma = \mathbb{Z}_2\times\mathbb{Z}_2$ be an abelian group and $A$ be a $4$-dimensional $\Gamma$-graded
linear space defined by
$A_{(0,0)}=<e_1>,\quad A_{(0,1)}=<e_2>,\quad A_{(1,0)}=<e_3>$ and $A_{(1,1)}=<e_4>.$
The quintuple $(A,\cdot,\diamond,\varepsilon,\alpha)$ is a Hom-Novikov-Poisson color Hom-algebra with
\begin{align*}
&\text{the bicharacter:} \quad  \varepsilon\big((i_1,i_2),(j_1,j_2)\big)=(-1)^{i_1j_1+i_2j_2},\\
&\text{the multiplication $"\cdot"$:} \quad e_2\cdot e_3=e_3\cdot e_2=\mu e_4,\quad\mu\in\mathbb{K},\\
&\text{the multiplication $"\diamond"$:} \quad  e_2\diamond e_3=\lambda_1e_4,\quad e_3\diamond e_2=\lambda_2e_4,\quad e_3\diamond e_3=\lambda_3 e_1,\quad\lambda_i\in\mathbb{K},\\
&\text{the even linear map $\alpha:A\rightarrow A$ given by:} \quad
\begin{array}[t]{ll}
\alpha(e_1)=2e_1, &\quad \alpha(e_2)=-e_2,\\
\alpha(e_3)=-e_3, & \quad \alpha(e_4)=-2e_4.
\end{array}
\end{align*}
Therefore, $(A,\cdot,[\cdot,\cdot],\varepsilon,\alpha)$ is a Hom-Gelfand-Dorfman color Hom-algebra with
$$[e_2,e_3]=-[e_3,e_2]=(\lambda_1-\lambda_2)e_4,\quad[e_3,e_3]=2\lambda_3e_3.$$
\end{ex}

\begin{lem}[\cite{attan:SomeconstructionsofcolorHom-Novikov-Poissonalgeb2019}]
\label{LemmacommutHomass to HomNovPoissonalg}
Let $(A, \cdot,\varepsilon,\alpha)$ be a $\varepsilon$-commutative Hom-associative color Hom-algebra and $D$ be an even derivation. Define
a bilinear operation $\diamond$ on $A$ by
\begin{equation}
    x\diamond y=x\cdot D(y),\quad \forall x,y\in \mathcal{H}(A).
\end{equation}
Then
    $ (A,\cdot,\diamond,\varepsilon,\alpha)$ is a Hom-Novikov-Poisson color Hom-algebras.
\end{lem}
Combining Theorem \ref{homNovikovPoissoncoloralg to HGD} and Lemma \ref{LemmacommutHomass to HomNovPoissonalg} leads to the following corollary.
\begin{cor}
Let $(A, \cdot,\varepsilon,\alpha)$ be a $\varepsilon$-commutative Hom-associative color Hom-algebra and $D$
be an even derivation. Then $(A, \cdot, [\cdot,\cdot],\varepsilon,\alpha)$ is a Hom-Gelfand-Dorfman color Hom-algebra, where
for all $x,y\in \mathcal{H}(A)$,
\begin{equation}
    [x,y]=x\cdot D(y)-\varepsilon(x,y)y\cdot D(x).
\end{equation}
\end{cor}

Next theorem provides a procedure for construction of the Hom-Gelfand-Dorfman color Hom-algebras
from Gelfand-Dorfman color Hom-algebras and morphisms.
\begin{thm}\label{twist-Gelfand-Dorfman}
Let $\mathcal{A}=(A, \cdot ,[\cdot,\cdot],\varepsilon)$ be a Gelfand-Dorfman color Hom-algebra and
$\alpha :\mathcal{A}\rightarrow \mathcal{A}$ be a Gelfand-Dorfman color Hom-algebras
morphism. Define $\cdot_{\alpha}, \ [\cdot,\cdot]_\alpha:A \times A\rightarrow A$ for all $x, y\in \mathcal{H}(A)$, by
$x\cdot _{\alpha}y =\alpha (x\cdot y)$ and $[x,y]_{\alpha} =\alpha([x,y])$.
Then, $\mathcal{A}_\alpha=(A_{\alpha}=A, \cdot _{\alpha},[x,y]_\alpha,\varepsilon, \alpha)$ is a Hom-Gelfand-Dorfman color Hom-algebra called the $\alpha$-twist or Yau twist of $(A, \cdot ,[\cdot,\cdot],\varepsilon)$.
Moreover, assume that $\mathcal{A}'=(A', \cdot',[\cdot,\cdot]',\varepsilon)$ is another Hom-Gelfand-Dorfman color Hom-algebra and $\alpha':\mathcal{A}'\rightarrow \mathcal{A}'$ is a Hom-Gelfand-Dorfman color Hom-algebras morphism. Let $f:\mathcal{A}\rightarrow \mathcal{A}'$ be a
Hom-Gelfand-Dorfman color Hom-algebras morphism satisfying $f\circ \alpha =\alpha
'\circ f$. Then, $f:\mathcal{A}_{\alpha }\rightarrow \mathcal{A}'_{\alpha}$ is a
Hom-Gelfand-Dorfman color Hom-algebras morphism.
\end{thm}
\begin{proof}
Being a Gelfand-Dorfman color Hom-algebras
morphism, $\alpha:A\rightarrow A$ is an even linear map which is multiplicative with respect to $\cdot$ and $[\cdot,\cdot]$, that is,
$$ \forall\ x,y \in\mathcal{H}(A): \quad
\alpha (x \cdot y) =\alpha (x) \cdot \alpha (y), \quad \alpha ([x , y]) =[\alpha (x) , \alpha (y)].
$$
The equality \eqref{eq:Hom-GD} in $\mathcal{A}_{\alpha}$
is proved as follows:
\begin{align*}
&\alpha(y)\cdot_\alpha[x,z]_\alpha=\alpha(y)\cdot_\alpha\alpha([x,z])=\alpha(\alpha(y)\cdot\alpha([x,z]))\\
&
\text{\scriptsize ($\alpha$ morphism)} \\
&=\alpha^{2}(y)\cdot\alpha^{2}([x,z])
=\alpha^{2}(y)\cdot[\alpha^{2}(x),\alpha^{2}(z)]\\
&\text{\scriptsize ($\mathcal{A}$ is a Hom-G. D. color alg)} \\ &=\varepsilon(y,x)[\alpha^{2}(x),\alpha^{2}(y)\cdot\alpha^{2}(z)]
-\varepsilon(x+y,z)[\alpha^{2}(z),\alpha^{2}(y)\cdot\alpha^{2}(x)]\\
&\quad+[\alpha^{2}(y),\alpha^{2}(x)]\cdot\alpha^{2}(z)-\varepsilon(x,z)[\alpha^{2}(y),\alpha^{2}(z)]\cdot\alpha^{2}(x)\\
&\text{\scriptsize ($\alpha$ morphism)}\\
&=\varepsilon(y,x)[\alpha^{2}(x),\alpha(\alpha(y)\cdot\alpha(z))]-\varepsilon(x+y,z)[\alpha^{2}(z),\alpha(\alpha(y),\alpha(x))]
\\
&\quad+\alpha([\alpha(y),\alpha(x)])\cdot\alpha^{2}(z)-\varepsilon(x,z)\alpha([\alpha(y),\alpha(z)]\cdot\alpha^{2}(x)\\
&=\varepsilon(y,x)[\alpha^{2}(x),\alpha(y\cdot_\alpha z)]-\varepsilon(x+y,z)[\alpha^{2}(z),\alpha(y\cdot_\alpha x)]\\
&\quad+\alpha([y,x]_\alpha\cdot\alpha^{2}(z)-\varepsilon(x,z)\alpha([y,z]_\alpha)\cdot\alpha^{2}(x)\\
&=\varepsilon(y,x)[\alpha(x),y\cdot_\alpha z]_\alpha-\varepsilon(x+y,z)[\alpha(z),y\cdot_\alpha x]_\alpha\\
&\quad+[y,x]_\alpha\cdot_\alpha\alpha(z)-\varepsilon(x,z)[y,z]_\alpha\cdot_\alpha \alpha(x).
\end{align*}
The second assertion follows from
\begin{align*}
    f(x\cdot_{\alpha} y)&=f(\alpha(x\cdot y))
    =\alpha'( f(x\cdot y)) = \alpha' (f(x) \cdot' f(y))
    =f(x)\cdot'_{\alpha'} f(y), \\
    f([x,y]_{\alpha})&=f(\alpha([x, y]))
    =\alpha' (f([x,y])) = \alpha' ([f(x) , f(y)]')
    =[f(x),f(y)]'_\alpha.
\qedhere
\end{align*}
\end{proof}
\begin{cor}
If $\mathcal{A}=(A, \cdot,[\cdot,\cdot],\varepsilon,\alpha)$ is a multiplicative Hom-Gelfand-Dorfman color
algebra, then for any $n\in\mathbb{N}^{\ast}$,
\begin{enumerate}
\item
The $n{\rm th}$ derived Hom-Gelfand-Dorfman color Hom-algebra of type $1$ of $\mathcal{A}$ is
defined by
$$\mathcal{A}_{1}^{n}=(A,\cdot^{(n)}=\alpha^{n}\circ\cdot,\ast^{(n)}=\alpha^{n}\circ[\cdot,\cdot],\varepsilon,\alpha^{n+1}).$$
\item
The $n{\rm th}$ derived Hom-Gelfand-Dorfman color Hom-algebra of type $2$ of $A$ is
defined by
$$\mathcal{A}_{2}^{n}=(A,\cdot^{(2^n-1)}=\alpha^{2^n-1}\circ\cdot,[\cdot,\cdot]^{(2^n-1)}=\alpha^{2^n-1}\circ[\cdot,\cdot],\varepsilon,\alpha^{2^n}).$$
\end{enumerate}
\end{cor}
\begin{proof}
Apply Theorem \ref{twist-Gelfand-Dorfman} with $\alpha'=\alpha^{n}$ and
$\alpha'=\alpha^{2^n-1}$ respectively.
\end{proof}
\begin{ex} Let $\Gamma = \mathbb{Z}_2\times\mathbb{Z}_2$ and $A$ be a $4$-dimensional $\Gamma$-graded
linear space with $A_{(0,0)}=<e_1>, A_{(0,1)}=<e_2>, A_{(1,0)}=<e_3>, A_{(1,1)}=<e_4>.$
Then there is a multiplicative admissible Hom-Gelfand-Dorfman color Hom-algebra $(A, \cdot,[\cdot,\cdot],\varepsilon,\alpha)$
with the bicharacter $\varepsilon\big((i_1,i_2),(j_1,j_2)\big)=(-1)^{i_1j_1+i_2j_2},$ and the multiplications tables for a basis $\{e_1, e_2, e_3,e_4\}$:
\begin{center}
\begin{tabular}{c|cccc}
$\cdot$ & $e_1$ & $e_2$ & $e_3$ & $e_4$ \\ \hline
$e_1$ & $0$ & $0$ & $0$ & $0$ \\
$e_2$ & $0$ & $0$ & $2e_4$ & $0$ \\
$e_3$ & $0$ & $2e_4$ & $e_1$& $0$ \\
$e_4$ & $0$ & $0$ & $0$& $0$ \\
\end{tabular}
\hspace{1 cm}
\begin{tabular}{c|cccc}
$[\cdot,\cdot]$ & $e_1$ & $e_2$ & $e_3$ & $e_4$ \\ \hline
$e_1$ & $0$ & $0$ & $0$ & $0$ \\
$e_2$ & $0$ & $0$ & $-2e_4$ & $0$ \\
$e_3$ & $0$ & $2e_4$ & $0$& $0$ \\
$e_4$ & $0$ & $0$ & $0$& $0$ \\
\end{tabular}
\end{center}
$$\begin{array}{ll}
    \alpha(e_1)=e_1, &\alpha(e_2)=-2e_2,\\
    \alpha(e_3)=-e_3,&\alpha(e_4)=2e_4,
\end{array}$$
Then there are Hom-Gelfand-Dorfman color Hom-algebras $\mathcal{A}_{1}^{n}$ and $\mathcal{A}_{2}^{n}$ with multiplications
tables respectively:
\begin{center}
\begin{tabular}{c|cccc}
$\cdot^{(n)}$ & $e_1$ & $e_2$ & $e_3$ & $e_4$ \\ \hline
$e_1$ & $0$ & $0$ & $0$ & $0$ \\
$e_2$ & $0$ & $0$ & $2^ne_4$ & $0$ \\
$e_3$ & $0$ & $2^ne_4$ & $e_1$& $0$ \\
$e_4$ & $0$ & $0$ & $0$& $0$ \\
\end{tabular}
\hspace{1 cm}
\begin{tabular}{c|cccc}
$[\cdot,\cdot]^{(n)}$ & $e_1$ & $e_2$ & $e_3$ & $e_4$ \\ \hline
$e_1$ & $0$ & $0$ & $0$ & $0$ \\
$e_2$ & $0$ & $0$ & $(-2)^ne_4$ & $0$ \\
$e_3$ & $0$ & $2^ne_4$ & $0$& $0$ \\
$e_4$ & $0$ & $0$ & $0$& $0$
\end{tabular}
\end{center}
$$\begin{array}{ll}
    \alpha^{n+1}(e_1)=e_1, &\alpha^{n+1}(e_2)=(-2)^{n+1}e_2,\\
    \alpha^{n+1}(e_3)=(-1)^{n+1}e_3,&\alpha^{n+1}(e_4)=2^{n+1}e_4,
\end{array}$$
\vspace{0,1 cm}
\begin{center}
\begin{tabular}{c|cccc}
$\cdot^{(2^n-1)}$ & $e_1$ & $e_2$ & $e_3$ & $e_4$ \\ \hline
$e_1$ & $0$ & $0$ & $0$ & $0$ \\
$e_2$ & $0$ & $0$ & $2^{2^{n}-1}e_4$ & $0$ \\
$e_3$ & $0$ & $2^{2^{n}-1}e_4$ & $e_1$& $0$ \\
$e_4$ & $0$ & $0$ & $0$& $0$
\end{tabular}
\vspace{1 cm}
\begin{tabular}{c|cccc}
$[\cdot,\cdot]^{(2^n-1)}$ & $e_1$ & $e_2$ & $e_3$ & $e_4$ \\ \hline
$e_1$ & $0$ & $0$ & $0$ & $0$ \\
$e_2$ & $0$ & $0$ & $-2^{2^{n}-1}e_4$ & $0$ \\
$e_3$ & $0$ & $2^{2^{n}-1}e_4$ & $0$& $0$ \\
$e_4$ & $0$ & $0$ & $0$& $0$
\end{tabular}
\end{center}
$$
    \alpha^{2^n}(e_1)=e_1, \quad \alpha^{2^n}(e_2)=2^{2^n}e_2, \quad \alpha^{2^n}(e_3)=e_3,\quad \alpha^{2^n}(e_4)=2^{2^n}e_4.
$$
\end{ex}


\begin{thebibliography}{99}
\bibitem{AbdaouiAmmarMakhloufCohhomLiecolalg2015}
Abdaoui, K., Ammar, F., Makhlouf, A.: Constructions and cohomology of Hom-Lie color Hom-algebras, Comm. Algebra, \textbf{43}, 4581-4612 (2015)

\bibitem{AbdaouiMabroukMakhlouf}
Abdaoui, E., Mabrouk, S., Makhlouf, A.: Rota-Baxter Operators on Pre-Lie Superalgebras, Bulletin of the Malaysian Math. Sci. Soc. (2) \textbf{42} (4), 1567-1606 (2019)

\bibitem{AmmarEjbehiMakhlouf:homdeformation}
Ammar, F., Ejbehi, Z., Makhlouf, A.: Cohomology and deformations of Hom-algebras,  J. Lie Theory, \textbf{21}(4), 813-836 (2011)

\bibitem{arnold}
Arnold, V.I.: Mathematical methods of classical mechanics, Grad. Texts in Math, Springer, Berlin, \textbf{60}, (1978)

\bibitem{attan:SomeconstructionsofcolorHom-Novikov-Poissonalgeb2019}
Attan, S.: Some constructions of color Hom-Novikov-Poisson algebras, Mathematical Sciences and Applications E-Notes, \textbf{1}, 78-86 (2019)

\bibitem{AttanLaraiedh:2020ConstrBihomalternBihomJordan}
Attan, S., Laraiedh, I.: Construtions and bimodules of BiHom-alternative and BiHom-Jordan algebras, 	arXiv:2008.07020 [math.RA] (2020)

\bibitem{BahturinMikhPetrZaicevIDLSbk92}
Bahturin, Y. A., Mikhalev, A. A., Petrogradsky, V. M., Zaicev, M. V.: Infinite-dimensional Lie superalgebras, De Gruyter Expositions in Mathematics \textbf{7}, Walter de Gruyter, Berlin (1992)

\bibitem{Bakayoko2016arXiv:Hom-Novikovcoloralgebras}
Bakayoko, I.: Hom-Novikov color Hom-algebras, arXiv:1609.07813 [math.RA], 16pp (2016)

\bibitem{Bakayoko:LaplacehomLiequasibialg}
Bakayoko, I.: Laplacian of Hom-Lie quasi-bialgebras, International Journal of Algebra, \textbf{8}(15), 713-727 (2014)

\bibitem{Bakayoko:LmodcomodhomLiequasibialg}
Bakayoko, I.: $L$-modules, $L$-comodules and Hom-Lie quasi-bialgebras, African Diaspora Journal of Mathematics, \textbf{17}, 49-64 (2014)

\bibitem{bakayokoattan}
Bakayoko, I. Attan, S.: Some structures of Hom-Poisson color Hom-algebras,
arXiv:1912.01358v1[math.RA] (2019).

\bibitem{BakBan:bimodrotbaxt}
Bakayoko, I., Banagoura, M.: Bimodules and Rota-Baxter Relations. J. Appl. Mech. Eng. \textbf{4}(5) (2015)

\bibitem{BakyokoSilvestrov:MultiplicnHomLiecoloralg}
Bakayoko, I., Silvestrov, S.: Multiplicative $n$-Hom-Lie color Hom-algebras,
In: Silvestrov, S., Malyarenko, A., Ran\u{c}i\'{c}, M. (Eds.), Algebraic Structures and Applications, Springer Proceedings in Mathematics and Statistics \textbf{317}, Ch. 7, 159-187, Springer (2020). (arXiv:1912.10216[math.QA] (2019))

\bibitem{BakyokoSilvestrov:HomleftsymHomdendicolorYauTwi}
Bakayoko, I., Silvestrov, S.:
Hom-left-symmetric color dialgebras, Hom-tridendri\-form color Hom-algebras and Yau's twisting generalizations, Afr. Mat. (2021). \\
(arXiv:1912.01441[math.RA] (2019))

\bibitem{bn}
Balinskii, A.A., Novikov, S.P.:  Poisson brackets of hydrodynamic type, Frobenius algebras and Lie algebras, Soviet Math. Dokl. \textbf{32}, 228-231 (1985)
Translated from Russian: Dokl. Akad. Nauk SSSR, 283:5 (1985), 1036-1039)

\bibitem{BenMakh:Hombiliform}
Benayadi, S., Makhlouf, A.: Hom-Lie algebras with symmetric invariant nondegenerate bilinear forms, J. Geom. Phys. \textbf{76}, 38-60 (2014)

\bibitem{BenAbdeljElhamdKaygorMakhl201920GenDernBiHomLiealg}
Ben Abdeljelil, A., Elhamdadi, M., Kaygorodov, I., Makhlouf, A.: Generalized Derivations of $n$-BiHom-Lie algebras, In: Silvestrov, S., Malyarenko, A., Ran\u{c}i\'{c}, M. (Eds.), Algebraic Structures and Applications,  Springer Proceedings in Mathematics and Statistics \textbf{317}, Ch. 4, 81-97, Springer (2020). (arXiv:1901.09750[math.RA] (2019))

\bibitem{BenHassineChtiouiMabroukNcib19:CohomLiedeformBiHomleftsym}
Ben Hassine, A., Chtioui, T., Mabrouk, S., Ncib, O.: Cohomology and linear deformation of BiHom-left-symmetric algebras, arXiv:1907.06979 [math.RA], 19 pp (2019)

\bibitem{CalderonDelgado2012:splLiecolor}
Calderon, A., Delgado, J. S.: On the structure of split Lie color Hom-algebras, Linear Algebra Appl. \textbf{436}, 307-315 (2012)

\bibitem{CaenGoyv:MonHomHopf}
Caenepeel, S., Goyvaerts, I.: Monoidal Hom-Hopf Algebras,  Comm. Algebra, \textbf{39}(6), 2216-2240 (2011)

\bibitem{cp}
Charim, V., Pressley, A.N.:  A guide to quantum groups, Cambridge Univ. Press, Cambridge, (1994).

\bibitem{ChenSilvestrovOystaeyen:RepsCocycleTwistsColorLie}
Chen, X.-W., Silvestrov, S. D., van Oystaeyen, F.: Representations and Cocycle Twists of Color Lie Algebras,
Algebr. Represent. Theory, \textbf{9}(6), 633-650 (2006)

\bibitem{ChenPetitOystaeyenCOHCHLA}
Chen, C. W., Petit, T., Van Oystaeyen, F.: Note on cohomology of color Hopf and Lie algebras, J. Algebra, \textbf{299}, 419-442 (2006)

\bibitem{CaoChen2012:SplitregularhomLiecoloralg}
Cao, Y., Chen, L.: On split regular Hom-Lie color Hom-algebras, Comm. Algebra \textbf{40}, 575-592 (2012)

\bibitem{ChtiouiMabroukMakhlouf2}
Chtioui, T., Mabrouk, S., Makhlouf, A.:
BiHom-pre-alternative algebras and BiHom-alternative quadri-algebras,  Bull. Math. Soc. Sci. Math. Roumanie. \textbf{63 (111)}(1), 3-21 (2020)

\bibitem{DassoundoSilvestrov2021:NearlyHomass}
Dassoundo, M. L., Silvestrov, S.: Nearly associative and nearly Hom-associative algebras and bialgebras, arXiv:2101.12377 [math.RA], 24pp (2021)

\bibitem{dri87}
Drinfel'd, V. G.:  Quantum groups, in: Proc. ICM (Berkeley, 1986),  AMS, Providence, RI, 798-820 (1987).

\bibitem{dn1}
Dubrovin, B. A.,  Novikov, S. P.:  Hamiltonian formalism of one-dimensional systems of hydrodynamic type and the Bogolyubov-Whitham averaging method, Soviet Math. Dokl.\textbf{27}, 665-669 (1983)
(Translated from Russian: Dokl. Akad. Nauk SSSR 270, 781-785 (1983))

\bibitem{dn2}
Dubrovin, B. A.,  Novikov, S. P.: On Poisson brackets of hydrodynamic type, Soviet Math. Dokl. \textbf{30}, 651-654 (1984)
(Transalated from Russian:  Dokl. Akad. Nauk SSSR 279:2, 294--297  (1984))

\bibitem{EbrahimiFardGuo08}
Ebrahimi-Fard, K., Guo, L.: Rota-Baxter algebras and dendriform algebras,
J. Pure Appl. Algebra, \textbf{212}(2), 320-339 (2008)

\bibitem{fb}
Frenkel, E., Ben-Zvi, D.:  Vertex algebras and algebraic curves, Math Surveys and Monographs 88, 2nd ed., AMS, Providence, RI, (2004)

\bibitem{dg1}
Gel'fand, I. M., Dickey, L. A.:  Asymptotic behavior of the resolvent of Sturm-Liouville equations and the Lie algebras of the Korteweg-de Vries equations, Russian Math. Sur. \textbf{30}(5(185)), 77-113 (1975)
(Translated from Russian: Uspekhi Mat. Nauk, \textbf{30}(5), 67-100 (1975))

\bibitem{dg2}
Gel'fand, I. M., Dickey, L. A.:  A Lie algebra structure in a formal variational calculations, Funct. Anal. Its Appl.  \textbf{10}(1), 16-22 (1976)
(Translated from Russian: Funktsional'nyi Analiz i Ego Prilozheniya, \textbf{10}(1), 16-25 (1976))

\bibitem{GelfandDorfman1979:Hamoperatandassociatoralgebraic}
Gel'fand, I. M., Dorfman, I. Y. Hamiltonian operators and algebraic structures related to them, Funct. Anal. Its Appl. \textbf{13}(4), 248-262 (1979)
(Translated from Russian: Funktsional'nyi Analiz i Ego Prilozheniya, \textbf{13}(4), 13-30 (1979))

\bibitem{ger2}
Gerstenhaber, M.:  On the deformation of rings and algebras, Ann. Math. \textbf{79}, 59-103 (1964)

\bibitem{GrMakMenPan:Bihom1}
Graziani, G., Makhlouf, A., Menini, C., Panaite, F.: BiHom-Associative Algebras, BiHom-Lie Algebras and BiHom-Bialgebras, Symmetry, Integrability Geom.: Methods Appl. (SIGMA), \textbf{11}(086), 34 pp (2015)

\bibitem{HartwigLarSil:defLiesigmaderiv}
Hartwig, J. T., Larsson, D., Silvestrov, S. D.:
Deformations of Lie algebras using $\sigma$-derivations, J. Algebra, \textbf{295},  314-361 (2006)
(Preprint in Mathematical Sciences 2003:32, LUTFMA-5036-2003, Centre for Mathematical Sciences,
Department of Mathematics, Lund Institute of Technology, 52 pp. (2003))

\bibitem{HassanzadehShapiroSutlu:CyclichomolHomasal}
Hassanzadeh, M., Shapiro, I., S{\"u}tl{\"u}, S.: Cyclic homology for Hom-associative algebras, J. Geom. Phys. \textbf{98}, 40-56 (2015)

\bibitem{HounkonnouDassoundo:centersymalgbialg}
Hounkonnou, M. N., Dassoundo, M. L.: Center-symmetric Algebras and Bialgebras: Relevant Properties and Consequences. In: Kielanowski P., Ali S., Bieliavsky P., Odzijewicz A., Schlichenmaier M., Voronov T. (eds) Geometric Methods in Physics. Trends in Mathematics. 2016, pp. 281-293. Birkh{\"a}user, Cham (2016)

\bibitem{HounkonnouHoundedjiSilvestrov:DoubleconstrbiHomFrobalg}
Hounkonnou, M. N., Houndedji, G. D., Silvestrov, S.: Double constructions of biHom-Frobenius algebras, arXiv:2008.06645 [math.QA], 45pp (2020)

\bibitem{HounkonnouDassoundo:homcensymalgbialg}
Hounkonnou, M. N., Dassoundo, M. L.: Hom-center-symmetric algebras and bialgebras.  arXiv:1801.06539  [math.RA], 19 pp (2018)

\bibitem{kms:narygenBiHomLieBiHomassalgebras2020}
Kitouni, A., Makhlouf, A., Silvestrov, S.: On $n$-ary generalization of BiHom-Lie algebras and BiHom-associative algebras, In: Silvestrov, S., Malyarenko, A., Ran\u{c}i\'{c}, M. (Eds.), Algebraic Structures and Applications, Springer Proceedings in Mathematics and Statistics \textbf{317}, Springer, Ch. 5, 99-126 (2020)

\bibitem{kont}
Kontsevich, M.: Deformation quantization of Poisson manifolds, Lett. Math. Phys. \textbf{66}, 157-216 (2003)

\bibitem{MabroukNcibSilvestrov2020:GenDerRotaBaxterOpsnaryHomNambuSuperalgs}
Mabrouk, S., Ncib, O., Silvestrov, S.: Generalized Derivations and Rota-Baxter Operators of $n$-ary Hom-Nambu Superalgebras, Adv. Appl. Clifford Algebras, \textbf{31}, 32, (2021) (arXiv:2003.01080[math.QA] (2020)) 

\bibitem{ms:homstructure}
Makhlouf, A., Silvestrov, S. D.:
Hom-algebra structures. J. Gen. Lie Theory Appl. \textbf{2}(2), 51--64 (2008)
(Preprints in Mathematical Sciences  2006:10, LUTFMA-5074-2006, Centre for Mathematical Sciences, Department of Mathematics, Lund Institute of Technology, Lund University (2006))

\bibitem{MakhSilv:HomDeform}
Makhlouf, A., Silvestrov, S.: Notes on $1$-parameter formal deformations of Hom-associative and Hom-Lie algebras, Forum Math. \textbf{22}(4), 715-739 (2010)
(Preprints in Mathematical Sciences, Lund University, Centre for Mathematical Sciences, Centrum Scientiarum Mathematicarum, (2007:31) LUTFMA-5095-2007. arXiv:0712.3130v1 [math.RA] (2007))

\bibitem{MakhSilv:HomAlgHomCoalg}
Makhlouf, A., Silvestrov, S. D.:
Hom-algebras and Hom-coalgebras, J. Algebra Appl. \textbf{9}(04), 553-589 (2010) (Preprints in Mathematical Sciences, Lund University, Centre for Mathematical Sciences, Centrum Scientiarum Mathematicarum, (2008:19) LUTFMA-5103-2008. arXiv:0811.0400
[math.RA] (2008)) 

\bibitem{MakYau:RotaBaxterHomLieadmis}
Makhlouf, A., Yau, D.: Rota-Baxter Hom-Lie admissible algebras,
Comm. Alg., \textbf{23}(3), 1231-1257 (2014)

\bibitem{MaZheng:RotaBaxtMonoidalHomAlg}
Ma, T., Zheng, H.: Some results on Rota-Baxter monoidal Hom-algebras, Results Math. \textbf{72} (1-2), 145-170 (2017)

\bibitem{MikhZolotykhCALSbk95}
Mikhalev, A. A., Zolotykh, A. A.: Combinatorial Aspects of Lie Superalgebras, CRC Press 1995

\bibitem{Laraiedh1:2021:BimodmtchdprsBihomprepois}
Laraiedh, I.: Bimodules and matched pairs of noncommutative BiHom-(pre)-Poisson algebras, arXiv:2102.11364 [math.RA], 30 pp (2021)

\bibitem{LarssonSigSilvJGLTA2008:QuasiLiedefFttN}
Larsson, D., Sigurdsson, G., Silvestrov, S. D.: Quasi-Lie deformations on the algebra $\mathbb{F}[t]/(t^N)$, J. Gen. Lie Theory Appl. \textbf{2}(3), 201-205 (2008)

\bibitem{LarssonSilv:quasidefsl2}
Larsson, D., Silvestrov, S. D.: Quasi-deformations of $sl_2(\mathbb{F})$ using twisted derivations, Comm. Algebra, \textbf{35}, 4303-4318 (2007)
(Preprint in Mathematical Sciences 2004:26, LUTFMA-5047-2004, Centre for Mathematical Sciences, Lund Institute of Technology, Lund University (2004). arXiv:math/0506172 [math.RA] (2005))

\bibitem{LarssonSilvJA2005:QuasiHomLieCentExt2cocyid}
Larsson, D., Silvestrov, S. D.: Quasi-Hom-Lie algebras, central extensions and $2$-cocycle-like identities, J. Algebra, \textbf{288}, 321-344 (2005) (Preprints in Mathematical Sciences 2004:3, LUTFMA-5038-2004, Centre for Mathematical Sciences, Department of Mathematics, Lund Institute of Technology, Lund University (2004))

\bibitem{LarssonSilv2005:QuasiLieAlg}
Larsson, D., Silvestrov, S. D.: Quasi-Lie algebras, In: Fuchs, J., Mickelsson, J., Rozenblioum, G., Stolin, A., Westerberg, A. (Eds.),
Noncommutative Geometry and Representation Theory in Mathematical Physics, Contemp. Math. \textbf{391}, Amer. Math. Soc., Providence, RI, 241-248 (2005) (Preprints in Mathematical Sciences 2004:30, LUTFMA-5049-2004, Centre for Mathematical Sciences, Department of Mathematics, Lund Institute of Technology, Lund University (2004))

\bibitem{LarssonSilv:GradedquasiLiealg}
Larsson, D., Silvestrov, S. D.: Graded quasi-Lie agebras, Czechoslovak J. Phys. \textbf{55}, 1473-1478 (2005)

\bibitem{LarssonSilvestrovGLTMPBSpr2009:GenNComplTwistDer}
Larsson, D., Silvestrov, S. D.: On generalized $N$-complexes comming from twisted derivations, In: Silvestrov, S., Paal, E., Abramov, V., Stolin, A. (Eds.),
Generalized Lie Theory in Mathematics, Physics and Beyond, Springer-Verlag, Ch. 7, 81-88 (2009)

\bibitem{LiuMakhMenPan:RotaBaxteropsBiHomassalg}
Liu, L., Makhlouf, A., Menini, C., Panaite, F.:  Rota-Baxter operators on BiHom-associative algebras and related structures, Colloq. Math. \textbf{161}(2), 263-294 (2020) arXiv:1703.07275 [math.RA], 27pp (2017)

\bibitem{PiontkovskiSilvestrovC3dCLA}
Piontkovski, D., Silvestrov, S.: Cohomology of 3-dimensional color Lie algebras, J. Algebra, \textbf{316}(2), 499-513 (2007)

\bibitem{RichardSilvestrovJA2008}
Richard, L., Silvestrov, S. D.: Quasi-Lie structure of $\sigma$-derivations of $\mathbb{C}[t^{\pm1}]$,
J. Algebra, \textbf{319}(3), 1285-1304 (2008) (arXiv:math/0608196[math.QA] (2006). Preprints in mathematical sciences (2006:12), LUTFMA-5076-2006, Centre for Mathematical Sciences, Lund University (2006))

\bibitem{RichardSilvestrovGLTbnd2009}
Richard, L., Silvestrov, S. D.:
A note on quasi-Lie and Hom-Lie structures of $\sigma$-derivations of
${\mathbb C}[z_1^{\pm 1},\ldots,z_n^{\pm 1}]$, In: Silvestrov, S., Paal, E., Abramov, V., Stolin, A. (Eds.), Generalized Lie Theory in Mathematics, Physics and Beyond, Springer-Verlag, Ch. 22, 257-262, (2009)

\bibitem{SaadaouSilvestrov:lmgderivationsBiHomLiealgebras}
Saadaou, N., Silvestrov, S.: On $(\lambda,\mu,\gamma)$-derivations of BiHom-Lie algebras,	arXiv:2010.09148 [math.RA], (2020)

\bibitem{ss}
Schaller, P., Strobl, T.:  Poisson structure induced (topological) field theories, Mod. Phys. Lett. \textbf{9}, 3129-3136 (1994)

\bibitem{ScheunertGLA}
Scheunert, M.: Generalized Lie algebras, J. Math. Phys. \textbf{20}(4), 712-720 (1979)

\bibitem{ScheunertCOH2}
Scheunert, M.: Introduction to the cohomology of Lie superalgebras and some applications, Res. Exp. Math. \textbf{25}, 77-107 (2002)

\bibitem{ScheunertZHA}
Scheunert, M., Zhang, R. B.: Cohomology of Lie superalgebras and their generalizations, J. Math. Phys. \textbf{39}, 5024-5061 (1998)

\bibitem{Sheng:homrep}
Sheng, Y.: Representations of Hom-Lie algebras, Algebr. Reprensent. Theory \textbf{15}, 1081-1098 (2012)

\bibitem{ShengBai:homLiebialg}
Sheng, Y.,  Bai,  C.: A  new  approach  to  Hom-Lie  bialgebras,  J. Algebra,  \textbf{399},  232-250 (2014)

\bibitem{SigSilv:CzechJP2006:GradedquasiLiealgWitt}
Sigurdsson, G., Silvestrov, S.: Graded quasi-Lie algebras of Witt type, Czechoslovak J. Phys. \textbf{56}, 1287-1291 (2006)

\bibitem{SigSilv:GLTbdSpringer2009}
Sigurdsson, G., Silvestrov, S.: Lie color and Hom-Lie algebras of Witt type and their central extensions, In: Silvestrov, S., Paal, E., Abramov, V., Stolin, A. (Eds.), Generalized Lie Theory in Mathematics, Physics and Beyond, Springer-Verlag, Berlin, Heidelberg, Ch. 21, 247-255 (2009)

\bibitem{Silvestrov:class3dimcolLiealg}
Silvestrov, S. D.: On the classification of $3$-dimensional coloured Lie algebras, In: Quantum groups and quantum spaces (Warsaw, 1995), Banach Center Publ. \textbf{40}, Polish Acad. Sci., Warsaw, 159-170 (1997)

\bibitem{SilvestrovParadigmQLieQhomLie2007}
Silvestrov, S.: Paradigm of quasi-Lie and quasi-Hom-Lie algebras and quasi-defor\-mations, In "New techniques in Hopf algebras and graded ring theory", K. Vlaam. Acad. Belgie Wet. Kunsten (KVAB), Brussels, 165-177 (2007)

\bibitem{SilvestrovZardeh2021:HNNextinvolmultHomLiealg}
Silvestrov, S., Zargeh, C.: HNN-extension of involutive multiplicative Hom-Lie algebras, arXiv:2101.01319 [math.RA], 14pp (2021)

\bibitem{QSunHomPrealtBialg}
Sun, Q.: On Hom-Prealternative Bialgebras, Algebr. Represent. Theor. \textbf{19}, 657-677 (2016)

\bibitem{SunLi2017:parakahlerhomhomleftsymetric}
Sun, Q., Li, H.: On parak\"{a}hler Hom-Lie algebras and Hom-left-symmetric bialgebras, Comm. Algebra,  \textbf{45}(1), 105-120 (2017)

\bibitem{vaisman}
Vaisman, I.:  Lectures on the geometry of Poisson manifolds, Birkh\"{a}user, Basel, 1994

\bibitem{xu1}
 Xu, X.: On simple Novikov algebras and their irreducible modules, J. Algebra, \textbf{185}, 905-934 (1996)

\bibitem{xu2}
Xu, X.: Novikov-Poisson algebras, J. Algebra, \textbf{190}, 253-279 (1997)

\bibitem{Xu:quqdrqtic conformalsuperalgbs}
Xu, X.: Quadratic Conformal Superalgebras, J. Algebra \textbf{231}, 1-38 (2000)

\bibitem{LY} Yuan, L.: Hom-Lie color Hom-algebras, Comm. Alg., \textbf{40}(2), 575-592 (2012)

\bibitem{Yau:ModuleHomalg}
Yau, D.: Module Hom-algebras, arXiv:0812.4695[math.RA], 10pp (2008)

\bibitem{Yau:HomYangBaHomLiequasitribial}
Yau  D.:  The  Hom-Yang-Baxter  equation,  Hom-Lie  algebras  and  quasi-triangular  bialgebras,  J. Phys. A.: Math. Theor. \textbf{42}(16), 165-202  (2009)

\bibitem{Yau:HomHom}
Yau, D.: Hom-algebras and homology, J. Lie Theory \textbf{19}(2), 409-421 (2009)

\bibitem{yau5}
Yau, D.: Hom-Novikov algebras, arXiv:0909.0726 (2009)

\bibitem{Yau:homnovikovpoissonalg}
Yau, D.: A twisted generalization of Novikov-Poisson algebras, arXiv:1010.3410 [math.RA] (2010)

\bibitem{Yau:HombialgcomoduleHomalg}
Yau, D.: Hom-bialgebras and comodule Hom-algebras, Int. Electron.~J. Algebra, \textbf{8}, 45-64 (2010)  (arXiv:0810.4866[math.RA] (2008))

\bibitem{YauHomMalcevHomalternHomJord}
Yau, D.: Hom-Malcev, Hom-alternative, and Hom-Jordan algebras, Int. Electron. J. Algebra, \textbf{11},  177-217 (2012)

\bibitem{Hom-GelfDorf}
Yuan, L.: Hom Gel'fand-Dorfman bialgebras and Hom-Lie conformal algebras, J. Math.
Phys. \textbf{55}(4), 043507 (2014)
\end{thebibliography}
\end{document}